\documentclass[11pt]{amsart}

\usepackage{amssymb}
\usepackage{amsthm}
\usepackage{amsmath}
\usepackage{graphicx}
\usepackage{amsaddr}
\usepackage{comment}
\usepackage[usenames,dvipsnames]{xcolor}
\usepackage[margin = 1in]{geometry}
\usepackage{enumerate}
\usepackage[toc,page]{appendix}
\usepackage{lineno}
\usepackage{color}
\usepackage{hyperref}
\usepackage{mhchem}
\usepackage{siunitx}

\newcommand{\blue}{}
\definecolor{mygreen}{rgb}{0.05,0.5,0.8}

 \newtheorem{thm}{Theorem}[section]
 
 \newtheorem{lem}[thm]{Lemma}
 \newtheorem{prop}[thm]{Proposition}

 \theoremstyle{definition}

 \theoremstyle{remark}
 \newtheorem{rem}{Remark}

 \numberwithin{equation}{section}

\DeclareMathOperator{\Bo}{Bo}

\providecommand{\bbs}[1]{\left(#1\right)}

\newcommand{\la}{\langle}
\newcommand{\ra}{\rangle}
\newcommand{\pt}{\partial}
\newcommand{\pcon}{\eta}

\newcommand{\eps}{\varepsilon}

\newcommand{\ud}{\,\mathrm{d}}
\newcommand{\8}{\infty}
\newcommand{\F}{\mathcal{F}}
\newcommand{\mm}{\mathcal{M}}

\begin{document}

\title[Contact line dynamics on rough surface]{Gradient flow formulation and second order numerical method for motion by mean curvature and contact line dynamics on rough surface}

\author[Y. Gao]{Yuan Gao}
\address{Department of Mathematics, Duke University, Durham, NC}
\email{yuangao@math.duke.edu}

\author[J.-G. Liu]{Jian-Guo Liu}
\address{Department of Mathematics and Department of
  Physics, Duke University, Durham, NC}
\email{jliu@math.duke.edu}

\date{\today}

\begin{abstract}
We study the dynamics of a droplet moving on an inclined rough surface in the absence of inertial  and viscous stress effects. In this case, the dynamics of the droplet is a purely geometric motion in terms of the wetting domain and the capillary surface. Using a single  graph representation, we interpret this geometric motion as a gradient flow on a Hilbert manifold.  We propose unconditionally stable first/second order numerical schemes to simulate this geometric motion of the droplet, which is described using motion by mean curvature coupled with moving contact lines. The schemes are based on (i) explicit moving boundaries, which decouple the dynamic updates of the contact lines and the capillary surface, (ii) a semi-Lagrangian method on moving grids and (iii) a predictor-corrector method with a nonlinear elliptic solver upto second order accuracy. For the case of quasi-static dynamics with continuous spatial variable in the numerical schemes,  we prove the stability and convergence of the first/second order numerical schemes.  To demonstrate the accuracy and long-time validation of the proposed schemes,  several challenging computational examples - including breathing droplets, droplets on inhomogeneous rough surfaces and quasi-static Kelvin pendant droplets - are constructed and compared with exact solutions to quasi-static dynamics obtained by desingularized differential-algebraic system of equations (DAEs).  
\end{abstract}


\maketitle
\section{Introduction}
The dynamics and equilibrium of a  droplet  on a substrate are important  problems with many practical applications such as coating, painting in industries and the adhesion of vesicles in 	biotechnology. The capillary effect, which contributes the leading behaviors of the geometric motion of a small droplet, is characterized by the surface tension on interfaces separating two different physical phases. Particularly, the capillary effect greatly affects the dynamics of the contact angle and the  contact line of a droplet, where three phases (gas, liquid and solid) meet.  Dating back to 1805, \textsc{Young} introduced the notion of mean curvature to study the contact angle of a capillary surface and proposed Young's equation for the contact angle, between the capillary surface and the solid substrate,  of a static droplet. 
The geometric motion of a dynamic droplet is more challenging 
 and extensively studied in the literature at the modeling \cite{Davis_1980, deGennes_1985, de2013capillarity, doi2013soft, Qian_Wang_Sheng_2006}, analysis \cite{Finn_Shinbrot_1988, Grunewald_Kim_2011, Caffarelli_Mellet_2007, Caffarelli_Salsa_2005,  Freire_2010, Kim_Mellet_2014, Feldman_2018, Feldman_Kim_2018},  and computations level  \cite{ Yokoi_Vadillo_Hinch_Hutchings_2009, Schulkes_1994, pozrikidis_stability_2012, Sui_Ding_Spelt_2014}.

In this paper, we study the dynamics of a droplet placed on an inclined rough surface using vertical/horizontal graph representation. The contact angle hysteresis  depends on the instantaneous contact angle of the droplet and the spatial heterogeneity of the substrate. Due to the roughness of the hydrophilic/hydrophobic substrate  and moving contact lines, which leads to constant changes of local slope of the substrate,  interesting phenomena such as pinning or capillary rise of the droplet will occur.     After deriving the governing equations via gradient flows on a manifold, we propose first/second order unconditionally stable  numerical schemes for the dynamics of the droplet and provide the first/second order convergence analysis for the quasi-static dynamics.   Then we perform the accuracy check using quasi-static dynamics  and conduct several challenging examples to  accurately simulate the phenomena mentioned above.

{\blue Below, we give a more detailed  outline of the results and strategies of the present paper.}

{\blue First, we give a kinematic description of a droplet using an incompressible potential flow,  and derive the governing equations in a gradient flow formulation using a free energy including capillary effect and gravitational effect, and a Rayleigh dissipation function.} After neglecting the inertial effect and the contribution of viscous stress inside the droplet, the dynamics of the droplet becomes a purely geometric motion, i.e.,  motion by mean curvature of the capillary surface coupled with motion of the contact lines;  see Section \ref{sec2}. Using a single vertical/horizontal graph representation,   we give a gradient flow formulation of the dynamic droplet by regarding the geometric motion of this droplet as a trajectory on a  Hilbert manifold with boundary, where the obstacle occurs; see the resulting governing equations in Section \ref{sec2.4eq} and derivations in Appendix \ref{app_model}. Gradient flows on manifolds and the corresponding interpretation of minimizing movement with proper metrics have been the focus of recent research in both analytic and numerical aspects \cite{AGS, almeida2012mean,  Mielke_2015, Muntean_Rademacher_Zagaris_2016}.  To completely describe the dynamics of the droplet, a  free energy and a Riemannian metric (dissipation potential) in the gradient flow structure will be specific in different physical settings \cite{Davis_1980, Qian_Wang_Sheng_2006, Alberti_DeSimone_2011, doi2013soft}.   We also emphasize that there is a rich literature on  droplets  with different physical effects, such as viscosity  stress inside the droplet, Marangoni effect, electromagnetic fields, electric fields  or surfactant; see for instance \cite{Qian_Wang_Sheng_2006, ren2007boundary,  li2010augmented,  xu2014level,   xu2016variational, lai2010numerical, Sui_Ding_Spelt_2014, Klinteberg_Lindbo_Tornberg_2014}.

The dynamics of the wetting domain for a 3D droplet is a challenging problem due to the geometric complexity. For example, cusp/corner singularity, topological changes such as merging and splitting.  We refer to \cite{Grunewald_Kim_2011, feldman_dynamic_2014} for the studies of weak solutions to quasi-static  contact line dynamics in the case that motion by mean curvature of the capillary surface is replaced by a Poisson equation.    For the original mean curvature case, we also refer to \cite{Alberti_DeSimone_2011, Feldman_Kim_2018} and the references therein  for quantitative/qualitative stability theory of static droplets on rough surface.
 For the quasi-static dynamics where the capillary surface is determined by an elliptic equation, there are many analysis results on the global existence and homogenization problems; see \cite{caffarelli2006homogenization, Grunewald_Kim_2011, Kim_Mellet_2014, feldman_dynamic_2014} for capillary surface described by a harmonic equation and see \cite{Caffarelli_Mellet_2007, Caffarelli_Mellet_2007a, Chen_Wang_Xu_2013, Feldman_Kim_2018} for capillary surface described by spatial-constant mean curvature equation.

Next, we propose  unconditionally stable numerical schemes with second order accuracy for the 2D droplet dynamics described by the motion by mean curvature and the moving contact lines; see Section \ref{sec3}.  The unconditionally stable schemes are based on an explicit boundary update, which decouples the computations for the dynamics of the contact lines and the capillary surface. In Section \ref{convergence}, we will give the stability and convergence analysis for the quasi-static dynamics based on the explicit boundary updates and some properties/dependence formulas of the quasi-static profile; see Proposition \ref{prop_st}, Proposition \ref{pertur_pf0} and Theorem \ref{thm_con}, \ref{thm_con_2}.
For the full dynamic problem, the challenge of moving grids is handled by a  semi-Lagrangian method with second order accuracy; see Section \ref{sec_2nd_semiL}. To achieve a second order scheme efficiently, we also design a predictor-corrector scheme with a nonlinear elliptic solver upto second order accuracy; see Section \ref{sec-implicit}.  

Third, we construct several challenging and important computational examples to demonstrate the accuracy, validity and efficiency of our numerical schemes; see Section \ref{sec4}. (i) Using a quasi-static solution given by desingularized differential-algebraic system of equations (DAEs),  we check the second order accuracy in space and time for our numerical schemes. (ii) We construct a breathing droplet with a closed formula solution, and we use it to show the long-time validity of our numerical schemes.  (iii) We use   complicated inclined rough surfaces such as the classical Utah teapot \cite{10.1145/97879.97916} to demonstrate the contact angle hysteresis of a droplet on inhomogeneous rough surface and the competition between capillary effect and gravitational effect. (iv) Using a horizontal graph representation and a desingularized formula for quasi-static dynamics, we also present  simulations for Kelvin pendant drops with repeated bulges in the volume preserving case. For recent studies of steady solutions to Kelvin pendant droplet problem, we refer to \cite{ pozrikidis_stability_2012}; see also \cite{Finn_Shinbrot_1988, Schulkes_1994}. 

{\blue
Now we comment on existing references on different models and computations for droplets with hydrodynamic effects.
Arguably,  the contact line dynamics coupled with hydrodynamic effects is one of the  mostly  studied subjects in fluid dynamics.  The contact lines experience an infinite driven force to overcome the  no-slip boundary condition due to viscosity inside droplets. 
Various specific physical models describing  slip boundary conditions and contact angle dynamics need to be imposed. For instance, the well-known Navier slip boundary conditions   first used by \cite{dussan1979spreading} and the parameters in the Navier boundary condition can be determined by molecular dynamics \cite{ren2007boundary}.   A widely used moving contact line model coupled with fluids is derived in  \cite{ren2007boundary}; see \eqref{app_E} in Appendix \ref{app_model}.   In the  coupled hydrodynamic model of droplets, the normal traction induced by the fluids is balanced with mean curvature $F_n = \gamma_{lg} H$, while in  the purely geometric motion, the capillary surface is simply motion by mean curvature $-\zeta v_n=\gamma_{lg} H$ with friction coefficient $\zeta$.  There are  many numerical methods for the  coupled hydrodynamic model of droplets. The   level set method  is first used in \cite{li2010augmented} and with reinitialization in \cite{xu2016reinitialization} to simulate the moving contact lines.   Various other related numerical methods are comprehensively reviewed in 2014 by \cite{Sui_Ding_Spelt_2014},
 c.f. the  front-tracking method \cite{Leung_Zhao_2009, Chamakos_Kavousanakis_Boudouvis_Papathanasiou_2016}, fixed domain method \cite{Morland_1982},  the level set method \cite{Zhao_Chan_Merriman_Osher_1996, Ding_Spelt_2007, Sui_Spelt_2013}, the phase-field method \cite{GAO20121372, Jiang_Bao_Thompson_Srolovitz_2012, Wang_Jiang_Bao_Srolovitz_2015} or the threshold dynamic method  \cite{Esedoglu_Tsai_Ruuth_2008, Lu_Otto_2015, Wang_Wang_Xu_2019}.

Instead, we focus on  the purely geometric motion of droplets, in which the full dynamics    is the motion by mean curvature of the capillary surface and the contact line dynamics; see Section \ref{sec2}. 
We mention particularly numerical  methods that are closely related to the purely geometric motion case.   The mean curvature flow with obstacles is theoretically studied in \cite{almeida2012mean} in terms of weak solution constructed by a minimizing movement (implicit time-discretization).  The threshold dynamics method based on characteristic functions are first used  in \cite{xu2017efficient,  Wang_Wang_Xu_2019} to simulate the contact line dynamics, which is particularly efficient and can be easily adapted to droplets with topological changes.  They  extended the original threshold method for mean curvature flows to  the case with a solid substrate and a free energy with multi-phase surface tensions,  in the form of obstacle problems. However, since they do not enforce the contact line mechanism \cite{deGennes_1985, ren2007boundary}, i.e., relation between contact line speed and unbalanced Young's force $v_{cl}=\frac{\gamma_{lg}}{\mathcal{R}}\left(\cos \theta_Y -\cos \theta_{cl}\right)$, thus their computations on contact angle dynamics  are different with the present paper and only the  equilibrium Young's angle $\theta_Y$ is accurately recovered.
Besides, the level set method developed in \cite{li2010augmented, xu2016reinitialization} can not be directly used and also can not deal with rough surfaces.
Furthermore, to the best of our knowledge, the existing numerical methods for the contact line dynamics problem, including the level-set method, phase field method and threshold dynamics method, can not give the second order in time convergence as here  proved in Theorem \ref{thm_con_2}.
}

 The rest of this paper is organized as follows. In Section \ref{sec2}, we  describe the purely geometric motions of a  dynamic/quasi-static dynamic  droplet on inclined rough surfaces.  With a specific free energy and a dissipation function, the derivation of governing equations using  a gradient flow formulation is given in Appendix \ref{app_model}. In Section \ref{sec3}, we propose the  unconditionally stable 1st and 2nd order schemes for droplet dynamics on inclined rough surfaces. The stability and convergence analysis for the quasi-static equations are given in Section \ref{convergence}. The truncation error estimates and peusdo-codes for 1st/2nd order schemes are given in Appendix  \ref{appA} and Appendix \ref{appB} respectively. In Section \ref{sec4}, we give some accuracy validations of our schemes compared with the quasi-static solution and demonstrate several challenging  examples such as droplets in a teapot, breathing droplets and Kelvin pendant droplets.


\section{Derivation of purely geometric dynamics of partially wetting droplets}\label{sec2}
 In this section,
 we first revisit the kinematic equations of a droplet described by an incompressible potential flow  and the dynamic mechanism driven by a free energy including capillary effect and gravitational effect, and a Rayleigh dissipation function; see Section \ref{sec2.1_n}, Section \ref{sec2.2_n} and Section \ref{sec2.3_n}. After neglecting the inertial effect and the contribution of viscous stress inside the droplet, the dynamics of the droplet becomes a purely geometric motion, i.e., motion by mean curvature of the capillary surface  and motion of the contact lines. In this case, using a vertical graph representation
 we describe the geometric motion of a droplet as a gradient flow  on a Hilbert manifold; see the resulting governing equations for (quasi-static) dynamics of droplets with volume constraint  in Section \ref{sec2.4eq}. Detailed derivations for the motion of a 3D droplet driven by a  specific free energy and a Riemannian metric (dissipation function) are given in Appendix \ref{app_model}. 
The  governing equations of a 2D droplet on an inclined rough surfaces are  presented in Section \ref{sec3.1}.
 When the vertical graph representation is broken, we replace it  by a horizontal one; see Section \ref{sec4.3}. 

{ 
 \subsection{Kinematic description of a droplet on a solid substrate}\label{sec2.1_n}
 In this section, we first give a kinematic description of a droplet using a vertical graph function.
More precisely, 
we study the motion of a 3D droplet placed on a substrate $\{(x,y,z); z=0\}$. Assume the wetting domain is $(x,y)\in D \subset\mathbb{R}^2$ with boundary $\Gamma:=\pt D$. The droplet domain is identified by the area $$A:=\{(x,y, z);~ (x,y)\in D, \,0<z<u(x,y),\, \, u|_{\pt D}=0\}.$$ Assume the fluid inside the droplet follows an incompressible potential flow  with velocity $v(x,y,x)=\nabla \phi(x,y,z)$ and constant density $\rho$.   The motion of the shape of this droplet is described by the following two kinematic boundary conditions. The motion of the capillary surface on $\pt A\cap \{z>0\}$ with the outer normal $n$ is described by the normal speed
\begin{equation}
v_n(x,y,z) := n \cdot \nabla \phi(x,y,z), \quad (x,y,z)\in \pt A\cap \{z>0\}
\end{equation}
  and  the motion of $\Gamma$ (physically known as contact lines) with outer normal $n_{cl}$ in $x$-$y$ plane is described by the contact line speed
\begin{equation}
 v_{cl}(x,y):=n_{cl} \cdot \nabla_{x,y}\phi(x,y,0), \quad (x,y)\in \Gamma.
\end{equation}

Notice the incompressible  potential flow satisfies $\nabla \cdot v=0$. Hence together with kinematic boundary condition, we have
\begin{equation}\label{pot_f}
\begin{aligned}
\Delta \phi = 0, \quad  \text{ in } A;\qquad 
\pt_n \phi =\left\{  \begin{array}{cc}
v_n, \quad& \text{ on } \pt A \cap \{u>0\},\\
 0,  \quad& \text{ on } \pt A \cap \{u=0\}. 
\end{array}\right.
\end{aligned}
\end{equation} 
 {\blue For the dynamic wetting domain and droplets, all the quantities $A(t), D(t), u(x,y,t), \Gamma(t)$ depend on time variable $t$.} The compatibility condition for \eqref{pot_f} is $\int_{\pt A} v_n \ud s=0$, which turns out to be equivalent to  the volume preserving constraint. 
Indeed, using the normal speed $v_n=\frac{\pt_t u}{\sqrt{1+|\nabla u|^2}}$ in the  graph representation, we have
\begin{equation}
\int_{\pt A(t)} v_n \ud s = \int_{D(t)} \frac{\pt_t u}{\sqrt{1+|\nabla u|^2}} \sqrt{1+|\nabla u|^2} \ud x \ud y= \int_{D(t)} \pt_t u \ud x \ud y.
\end{equation}
Then by $u(x,y,t)=0$ on $\Gamma(t)$ and the Reynolds  transport theorem, we have
\begin{equation}\label{compa}
\int_{D(t)} \pt_t u(x,y,t) \ud x \ud y = \frac{\ud}{\ud t} \int_{D(t)} u(x,y,t) \ud x \ud y=0,
\end{equation}
where the last equality follows from the  volume preserving constraint
$
\frac{\ud}{\ud t} \int_{D(t)} u(x,y,t) \ud x \ud y=0.
$
 Hence in the volume preserving case, the motion of the droplet can be completely described by the motion of capillary surface $u(x,y,t)$ and the motion of  contact domain $D(t)$.

\subsection{Free energy for the droplet and Young's angle}\label{sec2.2_n} To give a specific free energy, we will follow the same notations and terminologies as in the classical book of \textsc{de Gennes} \cite{de2013capillarity}.
To consider the interactions between the three phases of materials: gas, liquid, and solid, denote  $\gamma_{sl}$ ($\gamma_{sg}, \gamma_{lg}$ resp.) as the interfacial surface energy density between solid-liquid phases (solid-gas, liquid-gas resp.). $\gamma_{sl}, \gamma_{sg}, \gamma_{lg}>0$ are also known as the surface tension coefficients. Surface tension plays the most important role in the dynamics and equilibrium of the droplet. {\blue Surface energy consists of the contributions from   the capillary surface (with surface tension $\gamma_{lg}$) and the wetting domain of the droplet (with the relative surface tension $\gamma_{sl}-\gamma_{sg}$).} Surface energy between liquid and gas   is the excess energy  due to the one half lower coordination number (in the mean field approximation) of  molecules at the surface compared with those sitting in the liquid bulk (\textsc{Doi} \cite{doi2013soft}).  
We take the total free energy  of the droplet as the summation of surface energy, gravitational energy and kinetic energy
\begin{equation}\label{energy00}
\begin{aligned}
\F=& \gamma_{lg} \int_{\pt A \cap \{u>0\}} \ud s + (\gamma_{sl}-\gamma_{sg}) \int_{D} \ud x\ud y+ \rho g  \int_{D}\frac{u^2}{2} \ud x \ud y+ \frac{\rho}{2} \int_{A} |\nabla \phi|^2 \ud x \ud y \ud z,\\
=&    \gamma_{lg} \int_{D} \sqrt{1+ |\nabla  u|^2} \ud x \ud y + (\gamma_{sl}-\gamma_{sg})\int_{D} \ud x\ud y+ \rho g  \int_{D}\frac{u^2}{2} \ud x \ud y + \frac{\rho}{2} \int_{A} |\nabla \phi|^2 \ud x \ud y \ud z ,\\
\end{aligned}
\end{equation}
where $\rho$ is the density of the liquid, $g$ is  the gravitational acceleration and $\phi$ satisfies \eqref{pot_f}.  {\blue For droplets with small Weber number, the contribution of the fluid's inertia is small  compared to  the capillary effect.}  Thus for small droplets, we  neglect the inertia effect and drop the last term in the free energy
\begin{equation}\label{energy}
\begin{aligned}
\F=    \gamma_{lg} \int_{D} \sqrt{1+ |\nabla  u|^2} \ud x \ud y + (\gamma_{sl}-\gamma_{sg})\int_{D} \ud x\ud y+ \rho g  \int_{D}\frac{u^2}{2} \ud x \ud y .\\
\end{aligned}
\end{equation}
$\F$ has units of energy. 
We neglect other effects, such as viscosity stress inside the droplet, Marangoni effect (solutocapillary and thermocapillary effect), electromagnetic fields, evaporation and condensation, etc. {\blue The free energy \eqref{energy} for small droplets in the current setup is particular useful for  
practical applications such as coating, painting in industries and the adhesion of vesicles in biotechnology.}

The competition between the three surface tension coefficients will determine uniquely the steady shape of the droplet  with  a fixed volume $V$.  
Let the density of gas outside the droplet be $\rho_0=0$. We denote the capillary coefficient as $\varsigma:= \frac{\rho g}{\gamma_{lg}}$ and the capillary length as $L_c:=\frac{1}{\sqrt{\varsigma}}.$ For a droplet with volume $V$, its  equivalent  length (characteristic length)  $L$ is defined as $V=\frac{4\pi}{3}L^3$ in 3D and $V=\pi L^2$ in 2D.  The Bond number 
$\Bo:=(\frac{L}{L_c})^2=\varsigma L^2$ shall be small enough to observe the capillary effect \cite{de2013capillarity}. Notice for  simplicity in presentations of the governing equations, we allow $\varsigma<0$ in the case of pendant droplet. Hence when $\varsigma<0$, the capillary length is $L_c=\frac{1}{\sqrt{|\varsigma|}}$ and the Bond number is $\Bo=(\frac{L}{L_c})^2=|\varsigma| L^2$.

{\blue Define  $\sigma$ as the  relative adhesion coefficient between the liquid and the solid
$$\sigma:= \frac{\gamma_{sl}-\gamma_{sg}}{\gamma_{lg}}.$$
}
Adhesive forces between the liquid and the solid cause the liquid drop to spread across the surface (called a  partially wetting liquid on a hydrophilic surface), while cohesive forces within the liquid cause the drop to ball up and avoid contact with the surface (called dewetting or non-wetting liquid on a hydrophobic surface). 
 { We remark that the spreading parameter $S:=\gamma_{lg}\left( \frac{\gamma_{sg}-\gamma_{sl}}{\gamma_{lg}} -1\right)$ could be positive in the so-called total wetting regime \cite[Section 1.2.1]{deGennes_1985}. In this case, the liquid spreads completely to a film of nanoscopic height, which can not be described using contact angle dynamics in the current paper.} 
 For the present contact angle dynamics setup, $|\sigma|\leq 1$.
 By Young's equation \cite{young1805iii}, the equilibrium contact angle 
$\theta_Y$ is determined by the Young's angle condition
\begin{equation}\label{young}
\cos \theta_Y= \frac{\gamma_{sg}-\gamma_{sl}}{\gamma_{lg}} = -\sigma.
\end{equation}
{\blue We call it the partially wetting liquid case (hydrophilic surface) if
$-1<\sigma= -\cos \theta_Y\leq 0,\, 0<\theta_Y\leq \frac{\pi}{2},$
while we call it the non-wetting liquid case (hydrophobic surface) if
$0\leq \sigma= -\cos \theta_Y<1, \, \frac{\pi}{2}\leq \theta_Y< \pi.$
The case $\theta_Y=0, \, \sigma=-1$ is called completely wetting.
}

\subsection{Friction force for the motion of droplet and Rayleigh dissipation function}\label{sec2.3_n}
There are three types of friction and viscosity force on the droplet. First,  the contact line friction force density is given by $-\mathcal{R}v_{cl}n_{cl}=-\mathcal{R}(n_{cl} \cdot \nabla_{x,y} \phi) n_{cl}$, where in 3D, $\mathcal{R}$ is the friction coefficient per unit length for the contact line with the units of  mass/(length $\cdot$ time). Second,  the  friction force density on the capillary surface is given by $-\zeta v_n n$, where in 3D, $\zeta$ is the  coefficient per unit area for the capillary surface with the units of mass/(area $\cdot$ time). Third,  the viscosity stress inside the droplet is $\mu(\nabla v+ \nabla v^\bot)$ where $\mu$ is the dynamic viscosity.  We neglect  the viscosity effect inside the droplet. Then the Rayleigh dissipation function  (with the unit of work)  is given by \cite{goldstein2002classical}
\begin{equation}\label{RQ}
Q=\frac{\mathcal{R}}{2} \int_{\Gamma (t)} |v_{cl}|^2 \ud s +  \frac{\zeta}{2} \int_{\pt A(t) \cap \{u>0\}} |v_n|^2   \ud s.
\end{equation} 
After neglecting the  kinetic energy and viscosity dissipation inside the droplet, the dynamics of the droplet is purely a geometric motion driven by the free energy \eqref{energy} and Rayleigh's dissipation function \eqref{RQ}. Therefore,  the motion of the droplet can be completely described by the geometric configurations: the boundary of wetting domain $\Gamma(t)$ and capillary surface $u(x,y,t)$, instead of by velocity potential $\phi$. 

\begin{figure}
\includegraphics[scale=0.55]{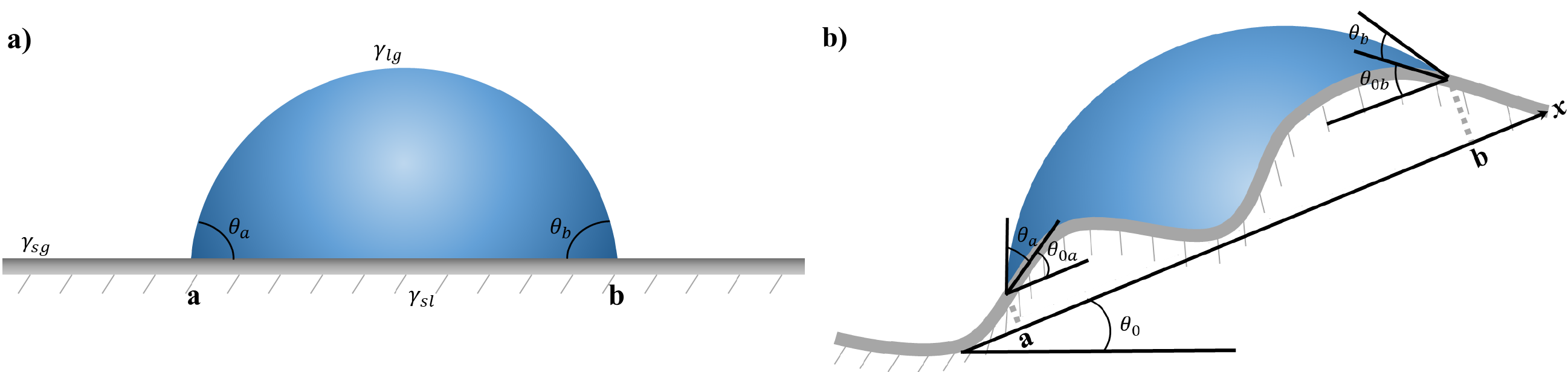} 
\caption{(a) Illustration of contact angles $\theta_a, \theta_b$; (b) Illustration of contact angles $\theta_a, \theta_b$ and the local slopes of the rough surface $\theta_{0a}, \theta_{0b}$.} \label{fig:ill}
\end{figure}

\subsection{Dynamics of a droplet derived by gradient flow on manifold}\label{sec2.4eq}
In Appendix \ref{app_model}, we will derive  the gradient flow of $\F(\pcon), \, \eta(t)=(\Gamma(t), u(t))$ on a Hilbert manifold  with respect to a Riemannian metric $g_{\pcon}$ suggested by \eqref{RQ}.
For a  3D single droplet, this yields the following governing equations
\begin{equation}\label{eq_nD}
\begin{aligned}
\frac{\zeta}{\gamma_{lg}} \frac{\pt_t u}{\sqrt{1+|\nabla u|^2}}= \nabla \cdot \left( \frac{\nabla u}{\sqrt{1+ |\nabla u|^2}} \right)-\varsigma u+\frac{1}{\gamma_{lg}}\lambda(t), \quad \text{ in } D(t),\\
 u(\Gamma(t), t)=0,\\
\frac{\mathcal{R}}{\gamma_{lg}}v_{cl}(t)= -\sigma-\frac{1}{\sqrt{1+|\nabla u|^2}},\quad \text{ on } \Gamma(t),\\
\int_{D(t)} u(x,y,t) \ud x \ud y = V,
\end{aligned}
\end{equation}
with initial data $(\Gamma(0), u(x,y,0))$ and initial volume $V$. When there are  topological changes, including merging and splitting of droplets, \eqref{eq_nD} becomes parabolic variational inequality; see Appendix \ref{app_model} for more discussions.

In the following proposition, we summarize the dissipation relation, the quasi-static dynamics and the contact line speed mechanism. The proof of Proposition \ref{Prop2.1} is given in Appendix \ref{appA.1} for completeness.
\begin{prop}\label{Prop2.1}
Assume $(\Gamma(t), u(x,y,t), \lambda(t))$, $t\in[0,T]$ are  solutions to \eqref{eq_nD} with regularities  $u\in H^2_0(D(t))$ and $\Gamma(t)$ is a $C^1$ Jordan curve with a finite perimeter. Then we have
\begin{enumerate}[(i)]
\item  the velocity relation on the contact line
\begin{equation}
\pt_t u = -(\nabla  u \cdot n_{cl}) v_{cl} = |\nabla  u| v_{cl}, \quad  \text{ on } \Gamma(t);
\end{equation}
\item  the energy-dissipation relation
\begin{equation}\label{dissipation_o}
\frac{\ud}{\ud t} \F = -\mathcal{R}\int_{\Gamma(t)} v_{cl}^2 \ud s - \zeta \int_{D(t)} \frac{(\pt_t u)^2}{\sqrt{1+|\nabla u|^2}} \ud x \ud y;
\end{equation}
\item if $\zeta=0$, the resulting quasi-static dynamics is a gradient flow for $\Gamma(t)$
\begin{equation}
\mathcal{R} v_{cl} =-\frac{\delta \F}{\delta \Gamma}=  -\gamma_{lg}\left(\sigma+\frac{1}{\sqrt{1+|\nabla u|^2}})\right ),\quad  \text{ on } \Gamma
\end{equation}
with $u\in H_0^2(D(t))$ solving $\nabla \cdot \left( \frac{\nabla u}{\sqrt{1+ |\nabla u|^2}} \right)-\varsigma u+\frac{1}{\gamma_{lg}}\lambda(t)=0$ in $D(t)$;
\item the equilibrium contact angle $\theta_{cl}^*=\theta_Y$ and on $\Gamma$
\begin{equation}
v_{cl}=\frac{\gamma_{lg}}{\mathcal{R}}(\cos \theta_Y-\cos \theta_{cl})
~\left\{\begin{array}{c}
>0, \quad \theta_{cl}>\theta_Y,\\
<0, \quad  \theta_{cl}<\theta_Y.
\end{array}
\right.
\end{equation}
\end{enumerate}

\end{prop}
For the cases that singularity occurs on $\Gamma(t)$, the solution $(\Gamma, u)$ shall be understood in the viscosity sense with some geometric assumptions on the wetting domain $D(t)$, which is beyond the scope of this paper.
We will use the statement (iii) above, together with a DAEs solver, to check the accuracy of our numerical schemes proposed in next section.

The dimensional analysis for a 2D droplet is given  in Appendix \ref{app_model}. The resulting dimensionless equations  in 2D are
\begin{equation}\label{wet-phy0}
\begin{aligned}
\beta\frac{\pt_t u(x,t)}{\sqrt{1+ (\pt_x u)^2}}= \frac{\pt }{\pt x}\left( \frac{\pt_x u}{\sqrt{1+ (\pt_x u)^2}} \right)-\kappa u+\hat{\lambda}(t), \quad x\in(a(t),b(t)),\\
 u(a(t), t)=u(b(t), t)=0,\\
a'(t)= \sigma+\frac{1}{\sqrt{1+(\pt_x u)^2}},\quad x=a(t),\\
b'(t)= -\sigma-\frac{1}{\sqrt{1+(\pt_x u)^2}},\quad x=b(t),\\
\int_{a(t)}^{b(t)} u(x,t) \ud x = V.
\end{aligned}
\end{equation}
Here, with typical length scale $L$ and time scale $T$, the coefficients  $\kappa=L^2 \varsigma$, 
 $\hat{\lambda}=\frac{L}{\gamma_{lg}}{\lambda}$, $\beta= \frac{\zeta L^2}{\gamma_{lg} T}$ and typical 2D volume \cite{de2013capillarity}   $V=\pi$  are all dimensionless. The capillary number for the capillary surface $\beta$ measures the ratio between the frictional force on capillary surface and surface tension, and indicates  the capillary relaxation time on the capillary surface. The Bond number $\kappa$ measures the ratio between the gravitational force and surface tension.
{\blue In the remaining of this paper, we will use the dimensionless formulation \eqref{wet-phy0} after dropping hat in $\hat{\lambda}$.}


\subsection{Governing equations for a single 2D droplet on a rough  and inclined surface}\label{sec3.1} In this section, with some modifications for the free energy, the Riemannian metric and the same derivation of the gradient flow formulation as in Appendix \ref{app_model}, we summarize  the governing equations for a single droplet on a rough  and inclined surface.

{\blue Given a rough  solid surface, we follow the convention for studying droplets on an inclined surface  and choose the Cartesian coordinate system built on an inclined plane with effective inclined angle $\theta_0$ such that $-\frac{\pi}{2}<\theta_0<\frac{\pi}{2}$, i.e.,  $(\tan \theta_0) x$ is the new $x$-axis we choose; see Fig \ref{fig:ill} (b).
With this Cartesian coordinate system,}   the rough surface is described by the graph of a function $w(x)$ and the droplet is then described by
\begin{equation}
A:= \{(x,y); a\leq x\leq b , w(x)\leq y\leq u(x)+w(x) \}. 
\end{equation}

The motion of this droplet is described by the relative height function (capillary surface) $u(x,t)\geq 0$ and  partially wetting domain $a(t)\leq x \leq b(t)$ with free boundaries $a(t), b(t).$  Consider a manifold
{\blue
\begin{equation}
\mm:= \{\eta=(a, b, u(x));~a,b \in \mathbb{R}, a\leq b,\, u(x)\geq 0,\, u(x)\in H_0^1(a,b)\}.
\end{equation}
  For any point $\eta=(a,b, u(x))\in \mm$, the tangential plane  $T_{\pcon} \mm$ at $\eta$ is 
\begin{equation}
\begin{aligned}
T_{\pcon} \mm:=&
\{q=(\alpha, \beta, v(x));\, \alpha, \beta \in \mathbb{R}, v(x) \in H^1(a,b), \\
&\qquad v(x)+u(x)\geq 0, ~ v(a)=  -\pt_x u(a) \alpha,\, v(b)= - \pt_x u(b) \beta\}.
\end{aligned}
\end{equation}
Given any $q_1=(\alpha_1, \beta_2, v_1), q_2=(\alpha_2, \beta_2, v_2)\in T_{\eta}\mm$, define the Riemannian metric $g_{\pcon}: T_\eta \mm \times T_\eta \mm \to \mathbb{R}$   as
\begin{align}\label{tm_metric}
&g_{\pcon}(q_1, q_2)=\mathcal{R}\alpha_1 \alpha_2 + \mathcal{R}\beta_1 \beta_2+ \zeta \int_{a}^{b}  \frac{v_1 v_2}{\sqrt{1+ (\pt_x (u+w))^2}} \ud x.
\end{align}
The dynamics of a droplet on rough surface can be regarded as  a trajectory $\eta(t)=(a(t),b(t), u(x,t))$ on $\mm$. The tangent direction of the curve $\eta(t)$ is given by
$\eta'(t):=(a'(t), b'(t), \pt_t u)\in T_{\eta(t)}\mm$.
}
%

Now we consider the energy functional $\mathcal{F}: \mm \to \mathbb{R}$ associated with the rough surface
\begin{equation}\label{energy-r}
\begin{aligned}
\F(\pcon)= &\gamma_{lg} \int_{a}^{b} \sqrt{1+ (\pt_x (u+w))^2} \ud x + (\gamma_{sl}-\gamma_{sg}) \int_{a}^{b} \sqrt{1+ (\pt_x w)^2} \ud x\\
&+ \rho g  \int_{a}^{b}\int_{w}^{u+w}(y\cos \theta_0+x\sin\theta_0)\ud y \ud x, 
\end{aligned}
\end{equation}
where $\rho$ is the density of the liquid, $g$ is  the gravitational acceleration.
In the inclined case, for a droplet with volume $V$ in 2D, the effective Bond number is
\begin{equation}\label{bo_in}
\Bo:=\left(\frac{L}{L_c}\right)^2 \cos \theta_0=\varsigma L^2 \cos \theta_0.
\end{equation}
Then by  same derivations as the gradient flow formulation in Appendix \ref{app_model},  with $h(x,t):= u(x,t)+w(x)$,
 the governing equations in the dimensionless form for a single droplet become
\begin{equation}\label{wet-eq-r}
\begin{aligned}
\beta \frac{\pt_t h(x,t)}{\sqrt{1+ (\pt_xh)^2}}= \frac{\pt }{\pt x}\left(  \frac{\pt_xh}{\sqrt{1+ (\pt_xh)^2}}\right)-\kappa (h\cos\theta_0+x \sin\theta_0)+\lambda(t), \quad x\in(a(t),b(t)),\\
 u(a(t), t)=u(b(t), t)=0,\\
a'(t)= \sigma  \sqrt{1+ (\pt_x w)^2}+\frac{1+ \pt_x h  \pt_x w}{\sqrt{1+(\pt_xh)^2}}=\frac{1}{\cos \theta_{0a}}( \cos \theta_a-\cos \theta_Y),\quad x=a(t),\\
b'(t)=- \sigma  \sqrt{1+ (\pt_x w)^2}-\frac{1+ \pt_x h \pt_x w}{\sqrt{1+(\pt_xh)^2}}=-\frac{1}{\cos \theta_{0b}}( \cos \theta_b-\cos \theta_Y),\quad x=b(t),\\
\int_{a(t)}^{b(t)} u(x,t) \ud x = V,
\end{aligned}
\end{equation}
where the two angles are defined as  $\pt_x w|_a = \tan \theta_{0a}$, $\pt_x h|_a= \tan(\theta_{0a}+ \theta_a)$ and $\pt_x w|_b =- \tan \theta_{0b}$ and $\pt_x h|_b= -\tan(\theta_{0b}+ \theta_b)$; see Fig \ref{fig:ill} (b). Recall $\beta, \kappa, V, \lambda$ are all dimensionless. It is easy to check the steady state $a'(t)=b'(t)=0$ recovers  Young's angle condition.

\begin{rem}
For $w(x)=0$, i.e., the surface is a perfect inclined plane with angle $\theta_0$, the derivation above recovers  the model for capillary droplets on an {inclined surface} \cite{finn2012equilibrium, Caffarelli_Mellet_2007, Kim_Mellet_2014}. 
\end{rem}
\begin{rem}
{We remark that changing the variable $x\in[a,b]$ to $x+x_0\in [a+x_0, b+x_0]$, the first equation in \eqref{wet-eq-r} for $\hat{h}(x)=h(x+x_0)$ with the new Lagrangian multiplier $\hat{\lambda}(t)=x_0 \sin \theta_0+\lambda(t)$ is  invariant with respect to the translation $x_0$.  As a consequence, the dynamics and the equilibrium profile do not depend on the coordinates we choose.} More importantly, at the equilibrium, the right hand side in the first equation is exactly the hydrostatic balance \cite[Section 61]{Landau1987Fluid}
\begin{equation}\label{vv}
    -(p+ \rho g h) = \gamma_{lg} \frac{\pt }{\pt x}\left(  \frac{\pt_xh}{\sqrt{1+ (\pt_xh)^2}}\right)-\gamma_{lg}\kappa (h\cos\theta_0+x \sin\theta_0)=\text{const}=-\gamma_{lg}\lambda,
\end{equation}
where we have chosen by convention the pressure outside of the drop to be zero and inside $p= -\gamma_{lg} \frac{\pt }{\pt x}\left(  \frac{\pt_xh}{\sqrt{1+ (\pt_xh)^2}}\right)$ due to the force balance on the capillary surface. With $w\equiv 0$, \cite[Theorem 8.3]{finn2012equilibrium} proved the nonexistence of the static profile for $\theta_0\neq 0, \pi$. With a sufficient substrate roughness, \cite{Caffarelli_Mellet_2007} proved the existence of the static profile.
\end{rem}
\begin{rem}\label{rem_V}
In the case without volume constraint, we can calculate the rate of change of the volume
\begin{align*}
\pt_t V=& \int_{a(t)}^{b(t)}  \pt_t h \ud x = \frac{1}{\beta}  \arctan \pt_x h\Big|_{a(t)}^{b(t)} - \kappa \int_{a(t)}^{b(t)} \sqrt{1+(\pt_x h)^2} (h\cos\theta_0+x \sin\theta_0) \ud x\\
=& -\frac{\theta_{0b}+ \theta_b+ \theta_{0a}+ \theta_a}{\beta}  - \kappa \int_{a(t)}^{b(t)} \sqrt{1+(\pt_x h)^2} (h\cos\theta_0+x \sin\theta_0) \ud x\leq 0.
\end{align*} 
The shrink estimate of the 2D droplet suggests  the volume preserving constraint is necessary to observe interesting phenomena for long enough time. For this reason, all the numerical examples are with volume constraint.
\end{rem}

\section{numerical schemes for droplets dynamics on a rough and inclined surface}\label{sec3}
In this section, we consider a droplet (described by a vertical graph function) on a rough and inclined surface in the partially wetting case, i.e., the relative adhesion coefficient  $-1< \sigma\leq 0$.  Since a
uniform estimate for the moving boundaries $a(t), b(t)$ in \eqref{wet-eq-r} can be obtained, we have an unconditionally stable explicit scheme for the time stepping of the moving
boundary, which in turn leads to the convergence of the numerical scheme.  This explicit discretization  decouples the computations for the moving boundary $a(t), b(t)$ and capillary surface $u(x,t)$.  In Section \ref{convergence}, to first illustrate the idea, we give  the stability and convergence analysis for the first/second order schemes for the quasi-static dynamics; see Proposition \ref{prop_st} and Theorems \ref{thm_con}, \ref{thm_con_2}.    To design the numerical schemes for the full dynamics of droplets and achieve a  second order scheme in time and space, we should particularly take care of the following issues. First, due to the moving grids at each time step, we need to map the capillary surface at different time to the same domain based on a semi-Lagrangian method upto second order accuracy. Second, to achieve a second order scheme efficiently, we design a predictor-corrector scheme with a nonlinear elliptic solver, which maintains the overall second order accuracy. Third, the effects from the inclined rough substrate and the volume constraint will be included. 
We will derive the first order scheme and give its truncation error in Section \ref{sec-sim}. Then we derive the second order scheme and give its truncation error  in Section \ref{sec3.2}. The proofs for truncation error estimates will be left to Appendix \ref{appA}.
Before we present the schemes, we list some key notations below in Table \ref{table00}.

\begin{table}[ht]
\caption{Commonly used notations in this paper.}\label{table00}
\begin{tabular}{|l|l|} 
\hline {\bf Symbols} & {\bf Meaning}\\ 
\hline\hline 
$t^n=n \Delta t$ &  Time steps\\
$a(t^n),\, b(t^n)$ & Exact moving boundaries at $t^n$\\
$a^n,\, b^n$ & Numerical moving boundary at $t^n$\\
$\tilde{a}^{n+1},\, \tilde{b}^{n+1}$ &Predictor numerical moving boundary at $t^{n+1}$\\
\hline \hline
$x^n\in[a(t^n), b(t^n)]$ & Spatial variable at $t^n$\\
$x^n\in[a^n, b^n]$ & Numerical spatial variable at $t^n$\\
$x^n_j=a^n+ j \tau^n, \, \tau^n= \frac{b^n-a^n}{N},\, j=0, \cdots, N$ & Moving spatial grids at $t^n$\\
$\tilde{ x}^{n+1}\in[\tilde{a}^{n+1}, \tilde{ b}^{n+1}]$ & Predictor  variable at $t^{n+1}$\\
$\tilde{x}^{n+1}_j= \tilde{a}^{n+1}+ j \tilde{\tau}^{n+1},$ & Predictor moving  grids at $t^{n+1}$\\  \qquad $\tilde{ \tau}^{n+1}= \frac{\tilde{b}^{n+1}- \tilde{a}^{n+1}}{N}, \, j=0, \cdots, N$ &\, \\
\hline\hline 
$h(x^n, t^n)$ for $x^n\in[a(t^{n}), b(t^{n})]$ & PDE solution at $t^n$\\
$h_j^n$, \, $j=0,\cdots, N$ &Numerical solution at time $t^n$ and spatial grid $x_j^n$\\
$h^n(x^n)$ for $x^n\in[a^n, b^n]$ & Numerical solution at $t^n$ (with continuous spatial variable)\\
$h^*(x^{n+1}, t^n)$ & Rescaled PDE solution at $t^n$\\
$h^{n*}(x^{n+1})$ &Numerical rescaled solution at $t^n$\\
\hline\hline 
$\tilde{h}(\tilde{x}^{n+1}, t^{n+1})$ & Predictor PDE solution at $t^{n+1}$\\
$\tilde{h}^{n+1}(\tilde{x}^{n+1})$ & Predictor numerical solution at $t^{n+1}$\\
$\tilde{h}_j^{n+1}$ & Predictor numerical solution at  $t^{n+1}$ and  grid $\tilde{x}_j^{n+1}$\\
$\tilde{h}^{n*}(x^{n+1})$ & Intermediate numerical rescaled solution from predictor\\
\hline 
\end{tabular}
\end{table}

\subsection{Stability analysis and convergence of numerical schemes  for quasi-static dynamics}\label{convergence}
In this section, to illustrate the idea, we will first  introduce   the first order/second order scheme for $w\equiv 0, \theta_0=0$ and the quasi-static case, i.e., $\beta=0$ in \eqref{wet-phy0}
\begin{equation}\label{quasi-eq}
\begin{aligned}
 \frac{\pt }{\pt x}\left( \frac{\pt_x u}{\sqrt{1+ (\pt_x u)^2}} \right)-\kappa u+\lambda(t)=0, \quad x\in(a(t),b(t)),\\
 u(a(t), t)=u(b(t), t)=0,\\
 \int_{a(t)}^{b(t)} u(x,t) \ud x = V,\\
a'(t)= \sigma+\frac{1}{\sqrt{1+(\pt_x u)^2}},\quad x=a(t),\\
b'(t)= -\sigma-\frac{1}{\sqrt{1+(\pt_x u)^2}},\quad x=b(t).\\
\end{aligned}
\end{equation}
Then we give the stability analysis and convergence result in Proposition \ref{prop_st} and Theorems \ref{thm_con}, \ref{thm_con_2} respectively. 
 Based on the observation for the unconditional stability and convergence, in Section \ref{sec-sim} and Section \ref{sec3.2}, we will design the first/second order numerical schemes for the full dynamic problem, i.e., $\beta>0$.

We first present the first/second order schemes and then prove the stability and convergence. Let $t^n = n\Delta t$, $n=0, 1, \cdots$ with time step $\Delta t$. Approximate $a(t^n), b(t^n), u(t^n) $ by $a^n, b^n, u^n$ respectively.
\\ \smallskip 
 \textit{First order scheme:}\\
Step 1 . Explicit boundary updates. Compute the one-side approximated derivative of $u^n$ at $b^n$ and $a^n$, denoted as $(\pt_x u^n)_N$ and $(\pt_x u^n)_0$. Then by  the  dynamic boundary condition in \eqref{wet-eq-r}, we update $a^{n+1}, b^{n+1}$ using 
\begin{equation}
\begin{aligned}\label{end_a00}
\frac{a^{n+1}-a^n}{\Delta t}&= \sigma +\frac{1 }{\sqrt{1+ (\pt_x u^n)_0^2}} , \qquad 
\frac{b^{n+1}-b^n}{\Delta t}&= -\sigma  -\frac{1}{\sqrt{1+ (\pt_x u^n)_N^2}}.
\end{aligned}
\end{equation}
%
%
Step 2. Update $u^{n+1}$ and $\lambda^{n+1}$ implicitly.
\begin{equation}\label{tm_stable}
\begin{aligned}
 \frac{\pt }{\pt x}\left( \frac{\pt_x u^{n+1}}{\sqrt{1+ (\pt_x u^{n+1})^2}} \right)-\kappa u^{n+1}+\lambda^{n+1}=0, &\\
u^{n+1}(a^{n+1})=0, \quad u^{n+1}(b^{n+1})=0&,\\
\int_{a^{n+1}}^{b^{n+1}} u^{n+1} \ud x = \int_{a^0}^{b^0} u^0 \ud x,&
\end{aligned}
\end{equation}
where the independent variable for $u^{n+1}$  is $x^{n+1}\in(a^{n+1},b^{n+1})$.
\\
\smallskip
\textit{Second order scheme:}\\
Step 1 . Repeat the Step 1 and Step 2 for the first order scheme. Denote the results as the predictor $\tilde{a}^{n+1}, \tilde{b}^{n+1}, \tilde{u}^{n+1}$.
\\
Step 2. Explicit boundary updates based on a predictor-corrector method. Compute the one-side approximated derivative of $u^n$ at $b^n$ and $a^n$, denoted as $(\pt_x u^n)_N$ and $(\pt_x u^n)_0$. Then  update $a^{n+1}, b^{n+1}$ using 
\begin{equation}
\begin{aligned}\label{end_a000}
\frac{a^{n+1}-a^n}{\Delta t}&= \sigma +\frac{1}{2}\left(\frac{1 }{\sqrt{1+ (\pt_x u^n)_0^2}} + \frac{1 }{\sqrt{1+ (\pt_x \tilde{u}^{n+1})_0^2}} \right) , \quad\\
\frac{b^{n+1}-b^n}{\Delta t}&= -\sigma  - \frac12\left( \frac{1}{\sqrt{1+ (\pt_x u^n)_N^2}}+ \frac{1}{\sqrt{1+ (\pt_x \tilde{u}^{n+1})_N^2}} \right).
\end{aligned}
\end{equation}
\\
Step 3. Update $u^{n+1}$ and $\lambda^{n+1}$ implicitly according to \eqref{tm_stable}.

\begin{prop}[Stability for first/second order scheme]\label{prop_st}
Suppose $\kappa>0$. Given initial data $a^0, b^0, u^0$,  assume $T< \frac{b^0-a^0}{2(1+\sigma)}$. Then for {\blue $n <\frac{T}{\Delta t}$}, we have
\begin{enumerate}[(i)]
\item the estimate for endpoints
 \begin{equation}\label{stable_ab}
\begin{aligned}
a^0 + \sigma T \leq {\blue a^{n+1}} \leq a^0 + (1+\sigma)T, \qquad b^0 - (1+\sigma) T \leq {\blue b^{n+1} } \leq b^0 - \sigma T;
\end{aligned}
\end{equation}
\item the estimate for $\lambda$
\begin{equation}\label{lambda_es}
0\leq \lambda^{n+1}\leq C;
\end{equation}
\item the estimate for $u$ and $u_x$
\begin{equation}\label{u_es}
\int_{a^{n+1}}^{b^{n+1}} \bbs{\sqrt{1+(\pt_x u^{n+1})^2} + \kappa (u^{n+1})^2 } \ud x  \leq C,
\end{equation}
where $C$ is a constant depending only on the initial data $a^0, b^0$ and $T$.
\end{enumerate}
\end{prop}
\begin{proof}
First, from \eqref{end_a00} and \eqref{end_a000}, we know for both first and second order schemes,
\begin{equation}
\sigma \Delta t \leq a^{n+1}-a^n \leq( \sigma+1) \Delta t, \quad  -(\sigma+1)\Delta t \leq  b^{n+1}-b^n\leq -\sigma \Delta t.
\end{equation}
This leads to statement (i).

Second, {\blue multiplying the first equation in \eqref{tm_stable} by $u^{n+1}$ and integration by parts show immediately that  $\lambda^{n+1}>0$ since  $\kappa>0$.}  Then integrating  the first equation in \eqref{tm_stable} from $a^{n+1}$ to $b^{n+1}$, we have
\begin{equation}
\frac{\pt_x u^{n+1}}{\sqrt{1+ (\pt_x u^{n+1})^2}} \bigg|_{a^{n+1}}^{b^{n+1}}-\kappa V+\lambda^{n+1}(b^{n+1}-a^{n+1})=0.
\end{equation}
Then by (i) we have
\begin{equation}
\lambda^{n+1}\leq \frac{\kappa V+2}{b^{n+1}-a^{n+1}}\leq  \frac{\kappa V+2}{b^0-a^0 -2(1+\sigma)T}
\end{equation}
 and thus \eqref{lambda_es}.

Third, multiplying the first equation in \eqref{tm_stable} by $u^{n+1}$ and integration by parts show that
\begin{equation}
\int_{a^{n+1}}^{b^{n+1}}\bbs{\sqrt{1+(\pt_x u^{n+1})^2}-1 + \kappa (u^{n+1})^2 } \ud x \leq \int_{a^{n+1}}^{b^{n+1}} \bbs{\frac{(\pt_x u^{n+1})^2}{\sqrt{1+(\pt_x u^{n+1})^2}} + \kappa (u^{n+1})^2 } \ud x   =\lambda^{n+1} V .
\end{equation}
This, together with the estimate for $\lambda$ in \eqref{lambda_es} and (i), gives the estimate for $\pt_x u$ and $u$ 
\begin{equation}
\begin{aligned}
\int_{a^{n+1}}^{b^{n+1}}\bbs{\sqrt{1+(\pt_x u^{n+1})^2} + \kappa (u^{n+1})^2 } \ud x  \leq& \lambda^{n+1} V + b^{n+1}-a^{n+1}\\
\leq&  \frac{\kappa V^2+2V}{b^0-a^0 -2(1+\sigma)T}+ b_0-a_0-2\sigma T
\end{aligned}
\end{equation}
and we conclude \eqref{u_es}.
\end{proof}

Before proving the convergence of the scheme, we first clarify  the existence and properties in Proposition \ref{pertur_pf0} for the quasi-static solution to
\begin{equation}\label{quasi22}
\begin{aligned}
 \frac{\pt }{\pt x}\left( \frac{\pt_x u}{\sqrt{1+ (\pt_x u)^2}} \right)-\kappa u+\lambda=0, \quad x\in(a,b),\\
 u(a)=u(b)=0,\\
 \int_{a}^{b} u(x,t) \ud x = V.
\end{aligned}
\end{equation}
 It is easy to see \eqref{quasi22}   is translation invariant for $x\to x+\epsilon$, so without loss of generality  we assume $-a= b >0$.  Due to the reflection invariance under $x\to -x$, the solution $u(x)$ is even.  Hence $\pt_x u(-x)=- \pt_x u(x)$. Let $\theta$ be the contact angle such that $\tan\theta = -\pt_x u|_{b}$ and thus
 \begin{equation}\label{cos}
\cos\theta=\frac{1}{\sqrt{1+(\pt_x u)^2}}\Big|_{x=b} .
 \end{equation}
{\blue From \cite[Theorem 1.1]{wente1980symmetry},   the solution to \eqref{quasi22} has  only one vertical axis of symmetry such that any nonempty intersection of $u$ with a horizontal  hyperplane are two points symmetric w.r.t. that  vertical axis.  Thus the maximum of $u(x)$ is attained uniquely at $x=0$ and there exists a unique horizontal graph representation using the inverse function $
X(u):=\{X; u(X)=u\}
.$}  Let $u(0)=u_m$ be the maximum of $u$. We have
\begin{equation}\label{quasi33}
\begin{aligned}
 \pt_u\left( \frac{X_u}{\sqrt{1+X_u^2}} \right) - \kappa u + \lambda=0,\quad 0\leq u \leq u_m,\\
X(u_m)=0,\quad X_u(u_m)=-\8,\\
\int_0^{u_m} X(u) \ud u=V/2, 
\end{aligned}
\end{equation}
whose derivation via gradient flow is given in Appendix \ref{app_nonwet} for completeness. 
 Equation \eqref{quasi33} can be used to describe not only the quasi-static profiles with single vertical graph representation but also  profiles with horizontal graph representation; see more details in Kelvin pendant drop problem  \eqref{dewet-eq-n}.
 For instance, for the simple case $\kappa=0$, the quasi-static profile is given by the spherical cap formula \eqref{cap} with 2D volume formula \eqref{vol2d}. If $2V<\pi b^2$ then the spherical cap has a single vertical graph representation.
 
 {\blue 
\begin{lem}\label{finn}(\cite[Theorem 3.2]{finn2012equilibrium})
Given a volume $V$ and a contact angle $0< \theta < \pi$, there exists unique $u_m, \lambda$ and $X(u), u\in[0,u_m]$ satisfying \eqref{quasi33}. Denote the contact point as $b=X(0)$.
\end{lem}
Given a contact angle $\theta$, the existence and symmetry of a static droplet has been comprehensively studied in \cite{finn2012equilibrium, finn1980}. However,  our numerical schemes require  the existence of \eqref{quasi22} for a given contact boundary $b$ and crucially relies on the continuous dependence on $b$. 
  Propostion \ref{pertur_pf0} below 
 obtains the unique critical wetting domain $b^c$ corresponding to $\theta^c=\frac{\pi}{2}$ such that 
 the quasi-static solution to \eqref{quasi22} has a single vertical graph representation. Moreover, it gives the estimate for the continuous dependence of the contact angle $\theta$ with respect to   contact boundary $b$. This is also the key for the convergence analysis later.
Before stating Proposition \ref{pertur_pf0}, we first give the following lemma for  the relations between  $b,\lambda, \theta, u_m$.
\begin{lem}\label{lem_re}
Suppose $\kappa>0$ and  the volume of the droplet is $V$. Then for any contact angle $0< \theta \leq \frac{\pi}{2}$, the  droplet profile obtained in Lemma \ref{finn} satisfies the following relations among $\theta, u_m, b$
\begin{equation}\label{rel_n_F}
\begin{aligned}
&F_1(u_m, \theta)
:= u_m^2 \int_0^1 v \frac{J(v, \theta)}{\sqrt{1-J( v, \theta)^2}} \ud v - \frac{V}{2}=0,\\
&b=F_2(u_m, \theta):= u_m \int_0^{1} \frac{J(v, \theta)}{\sqrt{1-J(v,\theta)^2}} \ud v ,
\end{aligned}
\end{equation}
where $J(v, \theta)$ is defined in \eqref{Jv}.
\end{lem}
}
\begin{proof}
Step 1.   Integrating once in the first equation of \eqref{quasi33}, we have
\begin{equation}\label{ju}
\frac{X_u}{\sqrt{1+X_u^2}}= \frac{\kappa u^2}{2} - \lambda u -\cos \theta,
\end{equation}
where we used
$\frac{X_u}{\sqrt{1+X_u^2}}\big|_{u=0}=-\cos \theta.$
Denote 
\begin{equation}\label{def_J}
J(u, \theta):=-\frac{\kappa u^2}{2} + \lambda u +\cos \theta.
\end{equation}
From the boundary condition $X_u(u_m)=-\8$ we know
\begin{equation}\label{jmn}
J(u_m, \theta) = -\frac{\kappa u_m^2}{2} + \lambda u_m +\cos \theta=1.
\end{equation}

Step 2. From \eqref{ju}, we know
\begin{equation}\label{ind}
J(u,\theta) = - \frac{X_u}{\sqrt{1+X_u^2}} \leq 1 = J(u_m, \theta).
\end{equation}
since  $\kappa>0$,  the graph of $J(u, \theta)$ for fixed $\theta$ is a  parabola open downwards.  
Hence its axis of symmetry is  located to the right of $u_m$ and thus $J(u, \theta)$ is increasing w.r.t $u$ for $0\leq u \leq u_m$, which implies 
\begin{equation}\label{ym000}
\kappa u-\lambda\leq 0 \quad \text{ for all } 0\leq u \leq u_m.
\end{equation}
In particular,  $0< \kappa u_m \leq \lambda$ and
\begin{equation}
1=J(u_m, \theta)\geq J(u, \theta)> J(0, \theta)= \cos \theta, \quad 0<  u\leq  u_m.
\end{equation}
Therefore, we know $0\leq \theta \leq  \frac{\pi}{2}$ if and only if $J(u,\theta)=- \frac{X_u}{\sqrt{1+X_u^2}}> 0$ for all $0< u \leq  u_m$.

Now we derive the relations between $\lambda$ and $u_m$.
Since $J\geq 0$
\begin{equation}\label{X_u15}
\frac{\ud X}{\ud u} = \frac{-J(u, \theta)}{\sqrt{1-J(u,\theta)^2}}, \quad 0\leq u \leq u_m.
\end{equation}
Since $X(u_m)=0$, we have the integral formula
\begin{equation}\label{intg}
X(u)= \int_u^{u_m} \frac{J(y,\theta)}{\sqrt{1-J(y,\theta)^2}} \ud y.
\end{equation}
Then from  the volume constraint,
\begin{equation}\label{int_VV}
\frac{V}{2} = \int_0^{u_m}  X(u) \ud u = -\int_0^{u_m} u X_u \ud u= \int_0^{u_m} u \frac{J(u, \theta)}{\sqrt{1-J(u,\theta)^2}} \ud u.
\end{equation}
Combining \eqref{ind}, \eqref{intg} and \eqref{int_VV}, we conclude 
\begin{equation}\label{rel_n}
\begin{aligned}
-\frac{\kappa u_m^2}{2} + \lambda u_m +\cos \theta=1,\\
\frac{V}{2} = \int_0^{u_m} u \frac{J(u, \theta)}{\sqrt{1-J(u,\theta)^2}} \ud u,\\
b=\int_0^{u_m} \frac{J(u, \theta)}{\sqrt{1-J(u,\theta)^2}} \ud u.
\end{aligned}
\end{equation} 

{\blue Step 3. To further eliminate $\lambda$, we introduce the variable $v\in[0,1]$ such that $u= u_m v$.  Then by changing variables $u=u_m v$ in 
\eqref{def_J}, we have
\begin{equation}\label{Jv}
J(v,\theta):=J(u_m v,\theta)=-\frac{\kappa u_m^2 v^2}{2} + \lambda u_m v + \cos\theta=\frac{\kappa u_m^2}{2}(v-v^2)+v+(1-v)\cos\theta,
\end{equation}
where we used the first equation in \eqref{rel_n} to eliminate $\lambda$. Thus the relations between $\theta, u_m, b$ in \eqref{rel_n} can be simplified as
\begin{equation}
\begin{aligned}
&F_1(u_m, \theta):= \int_0^{u_m} u \frac{J(u,\theta)}{\sqrt{1-J(u, \theta)^2}} \ud u- \frac{V}{2}
= u_m^2 \int_0^1 v \frac{J(v, \theta)}{\sqrt{1-J( v, \theta)^2}} \ud v - \frac{V}{2}=0,\\
&b=F_2(u_m, \theta):= \int_0^{u_m} \frac{J(u, \theta)}{\sqrt{1-J(u,\theta)^2}} \ud u =  u_m \int_0^{1} \frac{J(v, \theta)}{\sqrt{1-J(v,\theta)^2}} \ud v . 
\end{aligned}
\end{equation}
}
\end{proof}

{\blue 
\begin{prop}\label{pertur_pf0}
Suppose $\kappa>0$ and the volume of the droplet is $V$. For any $0<\theta\leq \frac{\pi}{2}$, assume $\theta, u_m, b$ satisfy \eqref{rel_n_F}. Then
\begin{enumerate}[(i)]
\item the function $u_m(\theta)$ obtained by Lemma \ref{finn} is strictly increasing w.r.t $\theta$ from $u_m(0^+)=0$ to $u_m(\frac{\pi}{2})=u_m^c$;
as a consequence, the inverse function $\theta=\theta(u_m)$ maps $u_m\in (0, u_m^c)$ onto $ \theta\in(0,\frac{\pi}{2})$ such that $F_1(u_m, \theta(u_m))=0$ and $\frac{\pt (\cos \theta(u_m))}{\pt u_m}<0$;
\item $b=F_2(u_m, \theta(u_m))$ is a function of $u_m$ mapping $u_m\in (0, u_m^c)$ onto $b\in(b^c, +\8)$ such that  $\frac{\pt b(u_m)}{\pt u_m}<0$; as a consequence, the inverse function $u_m=u_m(b)$  mapping $b\in(b^c,+\8)$ onto $u_m\in(0, u_m^c)$;
\item  the composition $\theta=\theta(u_m(b))$ is a function of $b$ mapping $b\in(b^c,+\8)$ to $\theta\in(0,\frac{\pi}{2})$;
and we have the estimate
\begin{equation}\label{theta_b}
0<\frac{\pt (\cos\theta(b))}{\pt b}  \leq  \frac{1}{V-b u_m} \left( \frac{6V}{ u_m}+\kappa u_m^3\right)=: C_m.
\end{equation}
\end{enumerate}
\end{prop}
\begin{proof}
Step 1. We first show statement (i) via inverse function theorem. 

Recall $J(v, \theta)$ defined in \eqref{Jv}. From Lemma \ref{finn}, we know $u_m=u_m(\theta)$ is a function of $\theta$ satisfying $F_1( u_m(\theta), \theta)=0$. Thus we have
$$
\frac{\pt u_m}{\pt \cos\theta} \frac{\pt F_1}{\pt u_m} +\frac{\pt F_1}{\pt \cos \theta}=0.
$$
From the first equation in \eqref{rel_n_F}, taking partial derivative of $F_1$
with respect to  $\cos \theta$, we have
\begin{equation}\label{tmFF_1}
 \frac{\pt F_1}{\pt {\cos \theta}} = u_m^2 \int_0^1 v \frac{1-v}{(1-J(v, \theta)^2)^{\frac32}} \ud v\geq  u_m^2\int_0^1 (v -v^2) \ud v= \frac{u_m^2}{6}>0
\end{equation}
due to $J(v,\theta)\leq 1$ from \eqref{ind}.
Moreover, taking partial derivatives of $F_1$ w.r.t $u_m$, we have
\begin{equation}\label{dev32}
\begin{aligned}
\frac{\pt F_1}{\pt u_m}=& 2 u_m \int_0^1 v \frac{J(v,\theta)}{\sqrt{1-J(v,\theta)^2}} \ud v + u_m^2 \int_0^1 \frac{\kappa u_m  (v^2-v^3)}{(1-J(v,\theta)^2)^\frac{3}{2}} \ud v\\
=& \frac{V}{u_m} + \kappa u_m^3 \int_0^1 \frac{v^2-v^3}{(1-J(v,\theta)^2)^\frac{3}{2}} \ud v>0.
\end{aligned}
\end{equation}
Therefore $\frac{\pt u_m}{\pt \cos\theta} = - \frac{\frac{\pt F_1}{\pt \cos \theta}}{\frac{\pt F_1}{\pt u_m}}<0$ and $u_m(\theta)$ is a strictly increasing function w.r.t $\theta$.  From the inverse function theorem we know there exists a unique function $\theta=\theta(u_m)$ such that 
\begin{equation}\label{sig1}
\frac{\pt \cos \theta}{\pt u_m} =-\frac{1}{u_m^2  \int_0^1 \frac{v-v^2}{(1-J(v)^2)^\frac{3}{2}} \ud v}\bbs{ \frac{V}{u_m} + \kappa u_m^3 \int_0^1 \frac{v^2-v^3}{(1-J(v,\theta)^2)^\frac{3}{2}} \ud v }< 0.
\end{equation}
From the definition of $F_1(u_m, \theta)$, it is easy to verify $u_m=0$ when $\theta=0$. Denote the value of $u_m$ corresponding to $\theta=\frac{\pi}{2}$ as $u_m^c$.
We conclude statement (i).

Step 2. Combining statement (i) and the second relation in \eqref{rel_n_F}, it is easy to see $b=F_2(u_m, \theta(u_m))$ is a function of $u_m$.

Step 3. We now show $u_m=u_m(b)$ is a function of $b$ via the inverse function theorem. Taking partial derivatives of $F_2$ w.r.t $u_m$, we have
\begin{equation}\label{dev33}
\begin{aligned}
\frac{\pt b}{\pt u_m} = \frac{b}{u_m} + \kappa u_m^2 \int_0^1 \frac{v-v^2}{(1-J(v)^2)^\frac{3}{2}} \ud v + u_m \frac{\pt \cos \theta}{\pt u_m}  \int_0^1 \frac{1-v}{(1-J(v)^2)^\frac{3}{2}} \ud v.
\end{aligned}
\end{equation}
From statement (i), we plug $\frac{\pt \cos \theta}{\pt u_m}$ into \eqref{dev33} to see
\begin{equation}\label{key1}
\begin{aligned}
\frac{\pt b}{\pt u_m}=&\frac{1}{-u_m^2  \int_0^1 \frac{v-v^2}{(1-J(v)^2)^\frac{3}{2}} \ud v} \Bigg[-b u_m  \int_0^1 \frac{v-v^2}{(1-J(v)^2)^\frac{3}{2}} \ud v + V \int_0^1 \frac{1-v}{(1-J(v)^2)^\frac{3}{2}} \ud v \\
& + \kappa u_m^4 \left( \int_0^1 \frac{1-v}{(1-J(v)^2)^\frac{3}{2}} \ud v \int_0^1 \frac{v^2-v^3}{(1-J(v)^2)^\frac{3}{2}} \ud v-  (\int_0^1 \frac{v-v^2}{(1-J(v)^2)^\frac{3}{2}} \ud v)^2 \right) \Bigg]\\
=:& \frac{1}{-u_m^2  \int_0^1 \frac{v-v^2}{(1-J(v)^2)^\frac{3}{2}} \ud v} [I_1 + I_2].
\end{aligned}
\end{equation}
From H\"older's inequality, we have
\begin{equation}\label{pos1}
 \left(\int_0^1 \frac{v\sqrt{1-v}\sqrt{1-v}}{(1-J(v)^2)^\frac{3}{2}} \ud v\right)^2 \leq   \int_0^1 \frac{1-v}{(1-J(v)^2)^\frac{3}{2}} \ud v \int_0^1 \frac{v^2-v^3}{(1-J(v)^2)^\frac{3}{2}} \ud v,
\end{equation} 
 which shows $I_2$ in \eqref{key1} is nonnegative. On the other hand, from \eqref{ym000}, we have
$$\frac{X_{uu}}{(1+X_u^2)^{\frac32}}=\pt_u\left( \frac{X_u}{\sqrt{1+X_u^2}} \right) = \kappa u - \lambda \leq 0, $$
so the quasi-static profile $X(u)$  is concave. Thanks to the concavity,  for any $\alpha\in(0,1)$,  
$X( (1-\alpha) u_m) > \alpha X(0) + (1-\alpha)X(u_m).$
Thus we know
 $ b u_m<V$. On the other hand, we have $\frac{V}{2}\leq b u_m$ from area formulas. Then from $0\leq v\leq 1$, we have
 \begin{equation}\label{pos2}
 \begin{aligned}
I_1=& -b u_m  \int_0^1 \frac{v-v^2}{(1-J(v,\theta)^2)^\frac{3}{2}} \ud v + V \int_0^1 \frac{1-v}{(1-J(v,\theta)^2)^\frac{3}{2}} \ud v\\
\geq& (V-b u_m) \int_0^1 \frac{v-v^2}{(1-J(v,\theta)^2)^\frac{3}{2}} \ud v \geq  \frac{V-b u_m}{6}>0.
\end{aligned}
 \end{equation}
 where we used $J(v,\theta)\leq 1$ similar to \eqref{tmFF_1}.
 Combining \eqref{pos1} and \eqref{pos2}, we obtain
 \begin{equation}\label{sig2}
 \frac{\pt b}{\pt u_m} \leq - \frac{I_1}{u_m^2  \int_0^1 \frac{v-v^2}{(1-J(v,\theta)^2)^\frac{3}{2}} \ud v}\leq -\frac{V-b u_m}{u_m^2}<0.
 \end{equation}
 Therefore $b(u_m)$ is a strictly decreasing function w.r.t. $u_m$. By the inverse function theorem, we conclude $u_m=u_m(b)$ is a function of $b$. Moreover, from the definition of $F_2$, it is easy to see when $u_m=0$, $\theta=0$, we have $b= +\8$. Denote the value of $b$ corresponding to $u_m^c$ as $b^c$. We conclude (ii).
 
 Step 4. From statement (i) and (ii), the composition $\theta(b)=\theta( u_m (b))$ is a function of $b$ and we conclude (iii).
 
Step 5.  Finally, we give the estimate in statement (iv). Combining \eqref{sig1} and \eqref{sig2}, we have
\begin{equation}\label{Cbound}
\begin{aligned}
0<  &  \frac{\pt (\cos \theta(b))}{ \pt b}=  \frac{\frac{\pt \cos\theta}{\pt u_m}}{\frac{\pt u_m}{\pt b}}  \leq  \frac{ \frac{V}{u_m}+\kappa u_m^3  \int_0^1 \frac{v^2-v^3}{(1-J(v,\theta)^2)^\frac{3}{2}} \ud v }{(V-b u_m)  \int_0^1 \frac{v-v^2}{(1-J(v,\theta)^2)^\frac{3}{2}} \ud v }\\
\leq & \frac{1}{V-b u_m} \left(\frac{ \frac{V}{u_m} }{  \int_0^1 \frac{v-v^2}{(1-J(v,\theta)^2)^\frac{3}{2}} \ud v } + \kappa u_m^3 \right)\leq  \frac{1}{V-b u_m} \left( \frac{6V}{ u_m}+\kappa u_m^3\right)=:C_m.
\end{aligned}
\end{equation}
\end{proof}

 }

{\blue Proposition \ref{pertur_pf0} gives the continuous dependence of the contact angle  $\theta=\theta(b)$ with respect to the contact point $b$ for  symmetric contact points ($a=-b$), so we conclude 
the existence of solutions to \eqref{quasi-eq} by the well-posedness of the ODE system
\begin{equation}
a' = \sigma + \cos \theta \bbs{\frac{b-a}{2}}, \qquad b' = -\sigma - \cos  \theta \bbs{\frac{b-a}{2}},
\end{equation}
where $\cos \theta(\cdot)$ is a function of $\frac{b-a}{2}.$
}
Now we state and prove the convergence result for the first order scheme \eqref{end_a00}-\eqref{tm_stable}. 

\begin{thm}\label{thm_con}
{\blue Assume $a(t), b(t)\in C^2[0,T]$ and the associated $ u(x,t)$ for $x\in[a(t), b(t)]$ are the exact solution to \eqref{quasi-eq}} and let  $a^n, b^n, u^n(x^n), \, x^n\in[a^n, b^n]$ at $t=t^n$ be the numerical solution obtained from the first order scheme \eqref{end_a00}-\eqref{tm_stable} with the same initial data $(a^0, b^0, u^0)$. Then for $n <\frac{T}{\Delta t}$, we have the convergence
\begin{equation}\label{convergence_ab}
|b(t^n)-b^{n}|\leq    e^{C_mT}   \Delta t, \quad |a(t^n)-a^{n}|\leq    e^{C_mT}  \Delta t,
\end{equation}
where $C_m$ is the bound in \eqref{theta_b}.
\end{thm}
\begin{proof}
Without loss of generality, {\blue we assume initially  $-a=b>0$ which is dynamically preserved for both exact solution and numerical scheme.}  So we have $-a(t)= b(t)$, $-a^n = b^n$ and we only prove the convergence for $b$. 

First,
from the Taylor expansion for the exact solution 
\begin{equation}
b(t^{n+1})= b(t^n)+b'(t^n)\Delta t + \frac12 b''(\xi) \Delta t^2
\end{equation}
and the boundary condition in \eqref{quasi-eq}, we have
\begin{equation}\label{tm3.55}
\frac{b(t^{n+1})-b(t^n)}{\Delta t} = b'(t^n) + \frac12 b''(\xi)\Delta t =  - \sigma - \cos \theta( b(t^{n})) + \frac12 b''(\xi)\Delta t.
\end{equation}
From  the the estimate $0\leq \frac{\pt (\cos\theta(b))}{\pt b} \leq C_m$ in \eqref{theta_b} and the boundary condition in \eqref{quasi-eq}, we have
\begin{equation}\label{b''}
|b''(\xi)|= \left|\frac{\pt (\cos\theta(b))}{\pt b}  b' \right| \leq \left|\frac{\pt (\cos\theta(b))}{\pt b} \right|\left|\sigma + \cos \theta \right| \leq 2C_m.
\end{equation}
Now, subtract \eqref{tm3.55} from the boundary update \eqref{end_a00} and denote $\eps^n:= |b(t^n)-b^n|$. From \eqref{theta_b} and \eqref{b''}, we have
{\blue 
\begin{equation}\label{recur}
\begin{aligned}
\frac{\eps^{n+1}-\eps^n}{\Delta t}\leq& |-\cos \theta(b(t^{n}))+ \cos \theta(b^n)|+C_m \Delta t\\
\leq& \left|\frac{\pt (\cos\theta(b))}{\pt b} \right| \eps^n+ C_m \Delta t \leq C_m ( \eps^n +  \Delta t).
\end{aligned}
\end{equation}
Here we used $\left| \frac{\cos \theta(b(t^n)) - \cos \theta(b^n)}{b(t^n)-b^n} \right| \leq \left| \frac{\pt (\cos \theta(b))}{\pt b} \right|\leq C_m$ since given $b^n$  the numerical profile $u^n$ (thus $\cos \theta(b^n)$) solved in the first order scheme \eqref{tm_stable} is also quasi-static profile and \eqref{theta_b} in Proposition \ref{pertur_pf0} holds.
}

Second,
\eqref{recur} gives the recurrence relation
\begin{equation}
\frac{\eps^{n+1}}{(1+C_m \Delta t)^{n+1}} \leq \frac{\eps^n}{(1+C_m \Delta t)^n} + \frac{C_m \Delta t^2}{(1+C_m \Delta t)^{n+1}}.
\end{equation} 
Thus 
\begin{equation}
\frac{\eps^{n}}{(1+C_m \Delta t)^{n}} \leq \eps^0 + C_m \Delta t^2 \sum_{k=1}^n \frac{1}{(1+C_m \Delta t)^k},
\end{equation}
which concludes
{\blue 
\begin{equation}
\eps^{n}\leq (1+C_m\Delta t)^{n} \eps^0+ (1+C_m\Delta t)^{n}  \Delta t \leq e^{C_mT} \left( \eps^0 + \Delta t\right).
\end{equation}
}
Thus $\eps^0=0$ gives the conclusion \eqref{convergence_ab}.
\end{proof}

Now we state the convergence result for the second order scheme \eqref{end_a000}-\eqref{tm_stable} and omit the proof.
\begin{thm}\label{thm_con_2}
{\blue Assume $a(t), b(t)\in C^3[0,T]$ and the associated $ u(x,t)$ for $x\in[a(t), b(t)]$ are the exact solution to \eqref{quasi-eq}}  and let $a^n, b^n, u^n(x^n), \, x^n\in[a^n, b^n]$ at $t=t^n$  be the numerical solution obtained from the second order scheme \eqref{end_a000}-\eqref{tm_stable} with the same initial data $(a^0, b^0, u^0)$. Then for $n <\frac{T}{\Delta t}$, we have the convergence
\begin{equation}
|b(t^n)-b^{n}|\leq Ce^{CT} \Delta t^2, \quad |a(t^n)-a^{n}|\leq Ce^{CT} \Delta t^2,
\end{equation}
where $C$ depends only on the bound $C_m$ in \eqref{theta_b}.
\end{thm}

\subsection{First order unconditionally stable  scheme based on explicit boundary update and semi-implicit motion by mean curvature}\label{sec-sim}
Based on the observation for the unconditional stability and convergence for the quasi-static dynamics of the droplet (in Section \ref{convergence}), in this section, we design the first order scheme and give the truncation error estimate. 

\subsubsection{First order scheme based on explicit boundary update and semi-implicit motion by mean curvature}\label{sec-1st-scheme}
Now we design a numerical algorithm for the motion by mean curvature case  \eqref{wet-eq-r}. 
With some proper spatial discretizations (such as finite difference, finite element, spectral approximation), we approximate $h(x^n, t^n)$ by $h^n(x^n)$ for $x^n\in(a^n, b^n)$ and approximate $\lambda(t^n)$ by $\lambda^n$. We propose the following three-step algorithm for updating $a^n, b^n, h^n, \lambda^n$ from time step  $t^n$ to $t^{n+1}$.

Step 1. Compute the one-side approximated derivatives of $h^n$ at $b^n$ and $a^n$, denoted as $(\pt_x h^n)_N$ and $(\pt_x h^n)_0$. Then by  the  dynamic boundary condition in \eqref{wet-eq-r}, we update $a^{n+1}, b^{n+1}$ using 
\begin{equation}
\begin{aligned}\label{end_a}
\frac{a^{n+1}-a^n}{\Delta t}&= \sigma \sqrt{1+ (\pt_x w)_0^2} +\frac{1+ (\pt_x h^n)_0(\pt_x w)_0 }{\sqrt{1+ (\pt_x h^n)_0^2}} , \quad\\
\frac{b^{n+1}-b^n}{\Delta t}&= -\sigma \sqrt{1+ (\pt_x w)_N^2} -\frac{1+ (\pt_x h^n)_N(\pt_x w)_N }{\sqrt{1+ (\pt_x h^n)_N^2}}.
\end{aligned}
\end{equation}

Step 2. Rescale $h^n$ from $[a^n, b^n]$ to $[a^{n+1}, b^{n+1}]$ with $O(\Delta t ^2)$ accuracy using a semi-Lagrangian discretization. 
For $x^{n+1}\in[a^{n+1}, b^{n+1}]$, denote the map from moving grids at $t^{n+1}$ to $t^n$ as
\begin{equation}
x^n:= a^n + \frac{b^n-a^n}{b^{n+1}-a^{n+1}}(x^{n+1}-a^{n+1})\in[a^n, b^n].
\end{equation}
Define the rescaled solution for $h^n$ as
\begin{equation}\label{inter-u-0}
    h^{n*}(x^{n+1}):= h^n(x^n)+ \pt_x h^n (x^n)(x^{n+1}-x^n).
\end{equation}
 It is easy to verify by the Taylor expansion
$
    h^{n*}(x^{n+1}) = h^n(x^{n+1}) + O(|x^n-x^{n+1}|^2)
$ for the independent variable  $x^{n+1}\in(a^{n+1},b^{n+1})$.

Step 3. Update $u^{n+1}$ and $\lambda^{n+1}$ semi-implicitly.
\begin{equation}\label{tm313}
\begin{aligned}
\frac{\beta}{\sqrt{1+ (\pt_x h^{n*})^2}} \frac{h^{n+1}(x^{n+1})-h^{n*}(x^{n+1})}{\Delta t}=  {\blue \frac{\pt_{xx} h^{n+1}}{(1+ (\pt_x h^{n*})^2)^\frac32} }-\kappa (h^{n+1}\cos \theta_0 + x^{n+1}\sin \theta_0)+\lambda^{n+1}, &\\
h^{n+1}(a^{n+1})=w(a^{n+1}), \quad h^{n+1}(b^{n+1})=w(b^{n+1})&,\\
{\blue \int_{a^{n+1}}^{b^{n+1}} (h^{n+1}-w)(x^{n+1}) \ud x^{n+1} =V},&
\end{aligned}
\end{equation}
where the independent variable  is $x^{n+1}\in(a^{n+1},b^{n+1})$ and $V$ is the initial volume.
For convenience, we provide a pseudo-code for this scheme in Appendix \ref{code1-in}.

Similar to \eqref{stable_ab} in Proposition \ref{prop_st}, from  \eqref{end_a}, we know for $n \Delta t <T$, since $\sigma<0$
\begin{equation}
\begin{aligned}\label{stability}
a^0+\left( \sigma \sqrt{1+ \max_x |w_x|^2} - \max_x |w_x| \right) T \leq a^n \leq a_0+  (1+\sigma+ \max_x|w_x|)T, \\
  b^0-(1+\sigma+ \max_x|w_x|)T\leq b^n \leq b^0-  \left( \sigma \sqrt{1+ \max_x |w_x|^2} - \max_x |w_x| \right) T,
\end{aligned}
\end{equation}
 so the explicit scheme for the moving boundaries is unconditionally stable.

\subsubsection{Truncation analysis for the first order scheme}
Here, we state the truncation error for the first order scheme.
\begin{lem}\label{first-lem1}
{\blue Assume $a(t), b(t)\in C^2([t^n,t^{n+1}])$ and $h(x,t)\in  C_{x,t}^{4,2}([a(t), b(t)]\times  [t^n,t^{n+1}])$  be the exact solution to \eqref{wet-eq-r}}   with initial data at $t=t^n$, $a^n, b^n, h^n(x^n)$ for $x^n\in[a^n, b^n]$. Then we have the first order truncation error estimates
\begin{equation}
\begin{aligned}\label{abn}
\frac{a(t^{n+1})-a^n}{\Delta t}&= \sigma \sqrt{1+ (\pt_x w)_0^2} +\frac{1+ (\pt_x h^n)_0(\pt_x w)_0 }{\sqrt{1+ (\pt_x h^n)_0^2}} +O(\Delta t), \quad\\
\frac{b(t^{n+1})-b^n}{\Delta t}&= -\sigma \sqrt{1+ (\pt_x w)_N^2} -\frac{1+ (\pt_x h^n)_N(\pt_x w)_N }{\sqrt{1+ (\pt_x h^n)_N^2}}+ O(\Delta t),
\end{aligned}
\end{equation}
\begin{equation}\label{un}
\begin{aligned}
\frac{\beta}{\sqrt{1+ (\pt_x h^{n*})^2}} \frac{h(t^{n+1})-h^{n*}}{\Delta t}=&  {\blue \frac{\pt_{xx} h(t^{n+1})}{\bbs{1+ (\pt_x h^{n*})^2}^\frac32} } -\kappa (h(t^{n+1})\cos \theta_0 + x^{n+1}\sin \theta_0)+\lambda(t^{n+1}) + O(\Delta t),\\
&\qquad\qquad\qquad\qquad\qquad\qquad\qquad\qquad x^{n+1}\in[a(t^{n+1}), b(t^{n+1})],
\end{aligned}
\end{equation}
where $h^{n*}$ is given by
\begin{equation*}
\begin{aligned}
    &h^{n*}(x^{n+1}):=  h^n(x^n)+ \pt_x h^n (x^n)(x^{n+1}-x^n),\\
    & x^{n}=a^n + \frac{b^n-a^n}{b(t^{n+1})-a(t^{n+1})}(x^{n+1}-a(t^{n+1}))\in[a^n,b^n].
    \end{aligned}
\end{equation*}
\end{lem}
For simplicity, $h(t^{n+1})$ represents $h(\cdot, t^{n+1})$ in the lemma above and the remaining contents. By mapping moving domain to  a fixed domain $Z=\frac{x-a(t)}{b(t)-a(t)}\in[0,1]$ for any $x\in[a(t), b(t)]$, the proof of this lemma is standard  so we put it in Appendix \ref{appA}.

\subsection{Second order numeric algorithm  based on a  predictor–corrector  method with an unconditionally stable explicit boundary update }\label{sec3.2} In this section, we use the  predictor–corrector  method to obtain a second order scheme.
With the notations in Table \ref{table00}, we still approximate $a(t^n), b(t^n) $ by $a^n, b^n$ respectively and  approximate $h(x^n, t^n)$ by $h^n(x^n)$ for $x^n\in(a^n, b^n)$.  However, it is more convenient to use a fixed domain variable 
\begin{equation}
Z(x,t)=\frac{x-a(t)}{b(t)-a(t)}\in[0,1], \quad \text{ for any }x\in[a(t), b(t)],
\end{equation}
 which is equivalent to
\begin{equation}\label{inverZ}
x(Z,t) = a(t) +  (b(t)-a(t)) Z \in [a(t), b(t)], \quad \text{ for } Z \in [0,1].
\end{equation}
 Denote $U(Z,t):=h(x,t)$. We will first present the second order numeric algorithm in Section \ref{sec_2nd_scheme} and then we  give the derivation of the second order scheme in Section \ref{sec-implicit} and Section \ref{sec_2nd_semiL}
based on the relation
\begin{equation}\label{Z-re}
Z= Z(x^{n+1}, t^{n+1})= Z(x^n, t^n)= Z(x^{n+\frac12}, t^{n+\frac12}).
\end{equation}
Here $t^{n+\frac12}:= (n+\frac12)\Delta t$ and $Z(x^{n+\frac12}, t^{n+\frac12}) =  \frac{x^{n+\frac12}-a(t^{n+\frac12})}{b(t^{n+\frac12})-a(t^{n+\frac12})}$ with the independent variable $x^{n+\frac12}.$

\subsubsection{Second order predictor-corrector scheme and unconditional stability for explicit boundary update}\label{sec_2nd_scheme}
Now we present the second order scheme with continuous spatial variables. 

Step 1. Predictor. Since we show in Section \ref{sec-implicit} that the  nonlinear elliptic solver  for motion by mean curvature requires second order accuracy, after updating $a^{n+1}, b^{n+1}$ by   the first order scheme in Section \ref{sec-sim}, we replace  the semi-implicit elliptic solver  by an implicit nonlinear  elliptic solver. Precisely, for the independent variable   $x^{n+1}\in(a^{n+1},b^{n+1})$,
\begin{equation}\label{newton}
\begin{aligned}
\frac{\beta}{\sqrt{1+ (\pt_x h^{n+1})^2}} \frac{h^{n+1}(x^{n+1})-h^{n*}(x^{n+1})}{\Delta t}=  {\blue \frac{\pt_{xx} h^{n+1}}{\bbs{1+ (\pt_x h^{n+1})^2}^\frac32} }-\kappa (h^{n+1}\cos \theta_0 + x^{n+1}\sin \theta_0)+\lambda^{n+1}, &\\
h^{n+1}(a^{n+1})=w(a^{n+1}), \quad h^{n+1}(b^{n+1})=w(b^{n+1})&,\\
{\blue \int_{a^{n+1}}^{b^{n+1}} (h^{n+1}-w)(x^{n+1}) \ud x^{n+1} = V,}&
\end{aligned}
\end{equation}
where $h^{n*}(x^{n+1})$ is the first order intermediate profile given in \eqref{inter-u-0} and $V$ is the initial volume.

Denote the results as the predictor $\tilde{a}^{n+1}, \tilde{b}^{n+1}, \tilde{h}^{n+1}(\tilde{x}^{n+1})$ for $\tilde{x}^{n+1}\in [\tilde{a}^{n+1}, \tilde{b}^{n+1}]$. To solve \eqref{newton}, one can use standard Newton's iterative method. 

Step 2. Explicit boundary update. Compute the one-side approximated derivative of $h^n$ at $b^n$ and $a^n$, denoted as $(\pt_x h^n)_N$ and $(\pt_x h^n)_0$. Then update
\begin{equation}\label{second-ab-move}
\begin{aligned}
\frac{a^{n+1}-a^n}{\Delta t}=\frac12\left\{ \sigma \sqrt{1+ (\pt_x w)_0^2}+ \sigma\sqrt{1+ (\pt_x \tilde{w})_0^2} +\frac{1+ (\pt_x h^n)_0(\pt_x w)_0 }{\sqrt{1+ (\pt_x h^n)_0^2}}+ \frac{1+ (\pt_x \tilde{h}^{n+1})_0(\pt_x \tilde{w})_0 }{\sqrt{1+ (\pt_x \tilde{h}^{n+1})_0^2}}\right\} ,\\
\frac{b^{n+1}-b^n}{\Delta t}=-\frac12\left\{ \sigma \sqrt{1+ (\pt_x w)_N^2}+ \sigma\sqrt{1+ (\pt_x \tilde{w})_N^2} +\frac{1+ (\pt_x h^n)_N(\pt_x w)_N }{\sqrt{1+ (\pt_x h^n)_N^2}}+ \frac{1+ (\pt_x \tilde{h}^{n+1})_N(\pt_x \tilde{w})_N }{\sqrt{1+ (\pt_x \tilde{h}^{n+1})_N^2}}\right\} 
\\
  {\blue \text{ with } (\pt_x w)_0:=\pt_x w(a^n),\,\, (\pt_x \tilde{w})_0:=\pt_x w(\tilde{a}^{n+1}),\,\, (\pt_x w)_N:=\pt_x w(b^n),\,\, (\pt_x \tilde{w})_N:=\pt_x w(\tilde{b}^{n+1}).}
\end{aligned}
\end{equation}

Step 3.
Solve $h^{n+1}(x)$ semi-implicitly.  With $h^{n+1}_0= w(a^{n+1}),\,\, h^{n+1}_N = w(b^{n+1}) $, for $x^{n+1}\in (a^{n+1}, b^{n+1}) $
\begin{align}\label{2nd-eq-r}
&\beta \frac{h^{n+1}(x^{n+1})-\tilde{h}^{n*}(x^{n+1})}{\Delta t}\frac12\left[\frac{1}{\sqrt{1+ (\pt_x h^{n})^2}}+ \frac{1}{\sqrt{1+ (\pt_x h^{n+1})^2}}\right]\\
=& \frac12 {\blue \left(\frac{\pt_{xx} {h}^{n+1}}{ \bbs{1+ (\pt_x h^{n+1})^2}^\frac32} +  \frac{\pt_{xx} h^n}{\bbs{1+ (\pt_x  h^n)^2}^\frac32} \right) }  -\frac{\kappa}{2} [(h^{n+1}+h^n) \cos \theta_0 +(x^n+x^{n+1}) \sin \theta_0  ] +\lambda^{n+\frac12},\nonumber\\
&\int_{a^{n+1}}^{b^{n+1}} (h^{n+1}-w)(x^{n+1}) \ud x^{n+1} = V\nonumber,
\end{align}
where $\tilde{h}^{n*}$ is the second order intermediate solution defined in \eqref{con-star} later.  Notice here  the equality holds in the sense of changing variables to the fixed domain $Z=\frac{x-a(t)}{b(t)-a(t)}\in[0,1]$ and $x^{n+1}, x^n$ are related to $Z$ by
$Z= Z(x^{n+1}, t^{n+1})= Z(x^n, t^n)= Z(x^{n+\frac12}, t^{n+\frac12}).$

We will give detailed derivation for the choice of the second order intermediate solution  $\tilde{h}^{n*}(x^{n+1})$ in Section \ref{sec_2nd_semiL}.
For convenience, we provide a pseudo-code for this scheme in Appendix \ref{code2-in}.

\subsubsection{Derivation of a second order scheme  based on the predictor-corrector method for DAEs with an algebraic solver upto second order accuracy}\label{sec-implicit}
To design a second order scheme, we illustrate the idea using the predictor-corrector method for an analogous DAEs. Assume we have an exact DAEs
\begin{equation}
\begin{aligned}
\dot{b}= f(b, u), \quad
0=g(b,u),
\end{aligned}
\end{equation}
where the second algebraic equation is equivalent to $u=G(b)$ for some function $G$.
However, in practice, one may not solve $u=G(b)$ exactly, which  in our case,  is to solve a nonlinear elliptic equation \eqref{newton}. Therefore, we design a predictor-corrector method to solve a DAEs with an algebraic solver upto second order accuracy.
Let $b^n, u^n$ be given such that $u^n=G(b^n)+O(\Delta t^2)$.
\\Step 1. Solve the predictor $\tilde{b}^{n+1}$ by forward Euler scheme
\begin{equation}
\frac{\tilde{b}^{n+1}-b^n}{\Delta t} = f(b^n, u^n).
\end{equation}
\\Step 2. Obtain the predictor $\tilde{u}^{n+1}$ by solving  algebraic equation up to a second order accuracy
\begin{equation}
\tilde{u}^{n+1}= G(\tilde{b}^{n+1})+O(\Delta t^2).
\end{equation}
\\Step 3. Solve the corrector $b^{n+1}$ by the trapezoidal method
\begin{equation}\label{trape}
\frac{{b}^{n+1}-b^n}{\Delta t} = \frac12[f(b^n, u^n)+f(\tilde{b}^{n+1}, \tilde{u}^{n+1})].
\end{equation}
\\Step 4. Obtain the corrector $u^{n+1}$ by solving the  algebraic equation up to a second order accuracy
\begin{equation}
{u}^{n+1}= G({b}^{n+1})+O(\Delta t^2).
\end{equation}

Indeed, we show the second order error estimate of this scheme below. Denote function $F(b):=f(b,G(b))=f(b, u)$. Then from \eqref{trape},  we have
\begin{equation}
\begin{aligned}
\frac{{b}^{n+1}-b^n}{\Delta t} &= \frac12[f(b^n, u^n)+f(\tilde{b}^{n+1}, \tilde{u}^{n+1})]\\
&=\frac12[f(b^n, G(b^n))+ f(\tilde{b}^{n+1}, G(\tilde{b}^{n+1}))]+ O(\Delta t^2)\\
&=\frac{1}{2} [F(b^n)+ F(\tilde{b}^{n+1})]+O(\Delta t^2),
\end{aligned}
\end{equation}
which gives the second order accuracy for $b^{n+1}$ and thus $u^{n+1}$.

\subsubsection{Derivation of the second order accuracy for the semi-Lagrangian term $\tilde{h}^{n*}(x^{n+1})$}\label{sec_2nd_semiL}
  Now we derive the second order scheme based on \eqref{Z-re}. Notice the spatial grids are moving along time.  We need to  map  grids at different time back to the same fixed domain $Z\in[0,1]$ based on \eqref{Z-re}. Furthermore, to achieve second order accuracy, we apply midpoint scheme and define 
\begin{equation}\label{def12}
a^{n+\frac12}:= \frac{a^n+a^{n+1}}{2}, \quad b^{n+\frac12}= \frac{b^n+ b^{n+1}}{2}.
\end{equation}
  We illustrate the second order accuracy for the following term involving time derivative, for independent variable $x^{n+\frac12}\in (a^{n+\frac12}, b^{n+ \frac12})$,
\begin{align}\label{tm_I3}
\pt_t h(x^{n+\frac12}, t^{n+\frac12}) &= \pt_t U(Z, t^{n+\frac12}) + \pt_Z U(Z, t^{n+\frac12})  \pt_t Z(x^{n+\frac12}, t^{n+\frac12}) \\
&=:I_1(x^{n+\frac12}, t^{n+\frac12})+ I_2(x^{n+\frac12}, t^{n+\frac12})I_3(x^{n+\frac12}, t^{n+\frac12}).
\end{align}
Below, we approximate $I_1, I_2, I_3$ upto second order accuracy.

First, notice $Z(x,t)$ at different time is related by \eqref{Z-re}. Thus whenever we evaluate  some quantity $U$ at different time, for instance at $t^{n+1}$, we mean $U(Z(x^{n+1}, t^{n+1}), t^{n+1})$.    Recall $U(Z,t)=h(x,t)$. Therefore, by midpoint scheme, $I_1(x^{n+\frac12}, t^{n+\frac12})=\pt_t U(Z, t^{n+\frac12})$ can be approximated by
\begin{align*}
I_1(x^{n+\frac12}, t^{n+\frac12})= \frac{U(Z, t^{n+1})-U(Z, t^n)}{\Delta t}+ O(\Delta t^2) = \frac{h^{n+1}(x^{n+1})-h^n(x^n)}{\Delta t}+ O(\Delta t^2),
\end{align*}
where we use the numerical solution $h^{n+1}(x^{n+1})\approx h(x^{n+1}, t^{n+1})=U(Z, t^{n+1})$ and $h^n(x^n)$ is similar. 
Here the  equality holds in the sense of changing variables $x^n, x^{n+\frac12}, x^{n+1}$ to the fixed domain $Z=\frac{x-a(t)}{b(t)-a(t)}\in[0,1]$ and
$Z= Z(x^{n+1}, t^{n+1})= Z(x^n, t^n)= Z(x^{n+\frac12}, t^{n+\frac12}).$

Second, 
by midpoint scheme, $I_2=  \pt_Z U(Z, t^{n+\frac12})$ can be approximated by
\begin{align*}
I_2= \frac12 \pt_Z [U(Z, t^n)+ U(Z, t^{n+1})]+O(\Delta t^2).
\end{align*}
Recall the scheme in Section \ref{sec_2nd_scheme} use  $\tilde{h}^{n+1}(\tilde{x}^{n+1})$ for $\tilde{x}^{n+1}\in [\tilde{a}^{n+1}, \tilde{b}^{n+1}]$ as predictor instead of the nonlinear unknown $h^{n+1}(x^{n+1})$.
From \eqref{accu23}, we have the $|\tilde{x}^{n+1}-x^{n+1}|=O(\Delta t^2)$, which enables us to replace the unknown term by the predictor. Then changing the intermediate  variable $Z$ back to $x$ gives
\begin{align*}
I_2&= \frac12[ {\pt_x h^{n}(x^n)}(b^n-a^n) +{\pt_x\tilde{h}^{n+1}(\tilde{x}^{n+1})}(\tilde{b}^{n+1}-\tilde{a}^{n+1})] + O(\Delta t^2).
\end{align*}

Third, we turn to approximate $I_3$. Still by the midpoint scheme, the last term $I_3= \pt_t Z(x^{n+\frac12}, t^{n+\frac12})$ in \eqref{tm_I3} can be approximated by
\begin{align*}
I_3&= \frac{1}{\Delta t} \left( \frac{x^{n+\frac12}-a^{n+1}}{b^{n+1}-a^{n+1}} - \frac{x^{n+\frac12} - a^n}{b^n- a^n} \right)  + O(\Delta t^2).
\end{align*}

Notice from \eqref{inverZ} and \eqref{def12}, we have
\begin{equation}
x^{n+\frac12}-a^{n+\frac12} = (b^{n+\frac12} - a^{n+\frac12}) Z = \frac12(b^n+b^{n+1} - a^{n+1}-a^n) Z,
\end{equation}
which is recast as
\begin{equation}\label{3.71}
x^{n+\frac12}-a^{n+\frac12}=\frac{1}{2}(b^{n+1}-a^{n+1} )( b^n -a^n) \left(  \frac{1}{b^{n+1}-a^{n+1}} + \frac{1}{b^n - a^n}   \right) Z, \quad Z\in[0,1].
\end{equation}
Notice also  the relation
$a^{n+1}-a^{n+\frac12}= a^{n+\frac12}-a^n = \frac{a^{n+1}-a^n}{2}$. 
Therefore,  the last term $I_3= \pt_t Z(x^{n+\frac12}, t^{n+\frac12})$ in \eqref{tm_I3} can be approximated by
\begin{align*}
I_3&= \frac{1}{\Delta t} \left( \frac{x^{n+\frac12}-a^{n+1}}{b^{n+1}-a^{n+1}} - \frac{x^{n+\frac12} - a^n}{b^n- a^n} \right)  + O(\Delta t^2)\\
&= \frac{x^{n+\frac12}- a^{n+\frac12}}{\Delta t} \left( \frac{1}{b^{n+1}-a^{n+1}} - \frac{1}{b^n - a^n}  \right) - \frac{a^{n+1}-a^n}{2 \Delta t} \left(  \frac{1}{b^{n+1}-a^{n+1}} + \frac{1}{b^n - a^n}   \right)+ O(\Delta t^2)\\
&=- \frac1{2\Delta t} \left(  \frac{1}{b^{n+1}-a^{n+1}} + \frac{1}{b^n - a^n}   \right)\left[ (a^{n+1}-a^n) +Z(( b^{n+1}- a^{n+1})-(b^n-a^n) )   \right]+ O(\Delta t^2)\\
&=- \frac1{2\Delta t} \left(  \frac{1}{b^{n+1}-a^{n+1}} + \frac{1}{b^n - a^n}   \right)\left( x^{n+1}-x^n \right)+ O(\Delta t^2),
\end{align*}
where we used \eqref{3.71} in the third equality.

Therefore,  we now define the intermediate variable
$\tilde{h}^{n*}(x^{n+1})$ such that
 \begin{equation}
 \pt_t h(x^{n+\frac12}, t^{n+\frac12}) = I_1+ I_2 I_3 = \frac{h^{n+1}(x^{n+1})-\tilde{h}^{n*}(x^{n+1})}{\Delta t}+ O(\Delta t^2).
 \end{equation}
 Using the approximated formulas for $I_1, I_2, I_3$ above, we propose the semi-Lagrangian term
\begin{equation}\label{con-star}
\begin{aligned}
\tilde{h}^{n*}(x^{n+1}):= h^n(x^n) &+ \frac{1}{4}\left( x^{n+1}-x^n \right)\cdot\\& \left[ \pt_x h^n(x^n) \left( 1+ \frac{b^n-a^n}{\tilde{b}^{n+1}-\tilde{a}^{n+1}} \right) +\pt_x \tilde{h}^{n+1}(\tilde{x}^{n+1}) \left( 1+  \frac{\tilde{b}^{n+1}-\tilde{a}^{n+1}}{b^n-a^n}  \right)\right].
\end{aligned}
\end{equation}
In summary, we have the second order approximation
\begin{equation}\label{inter-u-star}
\pt_t h (x^{n+\frac12}, t^{n+\frac12}) = \frac{h^{n+1}(x^{n+1})- \tilde{h}^{n*}(x^{n+1})}{\Delta t} + O(\Delta t^2).
\end{equation}

Since this is a key step in the numerical discretization, so we also 
 give the second order spatial discretization of $\tilde{h}^{n*}(x^{n+1})$ to see it has a similar form with \eqref{inter-u-0}. Denote spatial grid size $\tau^n=\frac{b^{n}-a^n}{N}$ and $\tau^{n+1}=\frac{b^{n+1}-a^{n+1}}{N}$.
Notice the second order spatial discretizations for $I_2, I_3$ are
\begin{align*}
I_2= \frac{N}{4} (h^{n}_{j+1}- h_{j-1}^n + \tilde{h}^{n+1}_{j+1}-\tilde{h}^{n+1}_{j-1}) + O(\Delta t^2 + \frac1{N^2}),\\
I_3= -\frac12 \left(  \frac{1}{b^{n+1}-a^{n+1}} + \frac{1}{b^n - a^n}   \right) \left[ (a^{n+1}-a^n) + j(h^{n+1}-h^n) \right] + O(\Delta t^2).
\end{align*}
Define the second order spatial discretization
\begin{equation}\label{tm376}
\begin{aligned}
\tilde{h}^{n*}(x^{n+1}_j):= h^n(x_j^n) + \frac{1}{8} \left( \frac{1}{\tau^{n+1}} + \frac{1}{\tau^{n}} \right) (h^{n}_{j+1}- h_{j-1}^n + \tilde{h}^{n+1}_{j+1}-\tilde{h}^{n+1}_{j-1}) \left[ (a^{n+1}-a^n) + j(\tau^{n+1}-\tau^n) \right].
\end{aligned}
\end{equation}

%

\subsubsection{Second order truncation error estimates for the predictor-corrector method}

The strategy of the second order truncation error estimates is the same as that of Lemma \ref{first-lem1}. Namely, we notice that   in a fixed domain in terms of $Z\in[0,1]$ the predictor-corrector method gives us a second order scheme and then we prove the mapping from $Z$ to  $x^{n+1}$  keeps the second order accuracy. For completeness, we put the proof of Lemma \ref{second-lem1} in Appendix \ref{appA}.

\begin{lem}\label{second-lem1}
{\blue Assume $a(t), b(t)\in C^3([t^n,t^{n+1}])$ and $h(x,t)\in  C_{x,t}^{6,3}([a(t), b(t)]\times  [t^n,t^{n+1}])$  be the exact solution to \eqref{wet-eq-r}}   with initial data at $t=t^n$, $a^n, b^n, h^n(x^n)$ for $x^n\in[a^n, b^n]$. Let $\tilde{a}^{n+1}, \tilde{b}^{n+1}$ be the predictor obtained by \eqref{end_a} and $\tilde{h}^{n+1}(\tilde{x}^{n+1})=\tilde{U}^{n+1}(Z)$ for $\tilde{x}^{n+1}\in [\tilde{a}^{n+1}, \tilde{b}^{n+1}]$ be the predictor obtained by \eqref{newton}. Then we have the second order truncation error estimates
\begin{equation}\label{second-abn}
\begin{aligned}
\frac{a(t^{n+1})-a^n}{\Delta t}=&\frac12\Big\{ \sigma \sqrt{1+ (\pt_x w)_0^2}+ \sigma\sqrt{1+ (\pt_x \tilde{w})_0^2} \\
&+\frac{1+ (\pt_x h^n)_0(\pt_x w)_0 }{\sqrt{1+ (\pt_x h^n)_0^2}}+ \frac{1+ (\pt_x \tilde{h}^{n+1})_0(\pt_x \tilde{w})_0 }{\sqrt{1+ (\pt_x \tilde{h}^{n+1})_0^2}}\Big\}+ O(\Delta t^2) ,\\
\frac{b(t^{n+1})-b^n}{\Delta t}=&-\frac12\Big\{ \sigma \sqrt{1+ (\pt_x w)_N^2}+ \sigma\sqrt{1+ (\pt_x \tilde{w})_N^2} \\
&+\frac{1+ (\pt_x h^n)_N(\pt_x w)_N }{\sqrt{1+ (\pt_x h^n)_N^2}}+ \frac{1+ (\pt_x \tilde{h}^{n+1})_N(\pt_x \tilde{w})_N }{\sqrt{1+ (\pt_x \tilde{h}^{n+1})_N^2}}\Big\} + O(\Delta t^2),
\end{aligned}
\end{equation}
where  {\blue $ (\pt_x w)_0:=\pt_x w(a^n),\,\, (\pt_x \tilde{w})_0:=\pt_x w(\tilde{a}^{n+1}),\,\, (\pt_x w)_N:=\pt_x w(b^n),\,\, (\pt_x \tilde{w})_N:=\pt_x w(\tilde{b}^{n+1})$} and
\begin{align}\label{second-hn}
\beta \frac{h(t^{n+1})-h^{n*}}{\Delta t}\frac12&\left[\frac{1}{\sqrt{1+ (\pt_x h^{n})^2}}+ \frac{1}{\sqrt{1+ (\pt_x  {h}^{n+1})^2}}\right]=\frac12 \left(\frac{\pt_{xx} {h}^{n+1}}{\bbs{1+ (\pt_x  {h}^{n+1})^2}^\frac32} +  \frac{\pt_{xx} h^n}{\bbs{1+ (\pt_x  h^n)^2}^\frac32} \right)\\
 &-\frac{\kappa}{2} [(h^{n+1}+h^n) \cos \theta_0 +(x^n+x^{n+1}) \sin \theta_0  ] +\lambda^{n+1}+ O(\Delta t^2),\nonumber
\end{align}
with $h^{n*}$  defined in \eqref{con-star}.
\end{lem}

\section{Validations and computations}\label{sec4}
In this section, we will first use the DAEs solution for the quasi-static dynamics to check the second order accuracy of the numerical schemes proposed in the last section. Then we design several challenging and important examples: (i) a periodic breathing droplet with a closed formula solution and a long-time computational validation; (ii) dynamics of a droplet on an inclined rough surface and in a ``Utah teapot";  (iii) Kelvin pendant droplet with repeated bulges and the corresponding desingularized DAEs solver for quasi-static dynamics based on horizontal  graph representation.

\subsection{Desingularized DAEs formula and accuracy check  for the quasi-static dynamics }
In this section, we will first derive the DAEs for the quasi-static dynamics using a desingularized formula. Then we give an accuracy check for the case $w(x)=0$ and $\theta_0=0$ using the corresponding quasi-static solutions, which can be obtained using the desingularized  DAEs solver.
\subsubsection{DAEs  description of the quasi-static dynamics}\label{sec4.2.1}
Given volume $V$, assume  $w(x)=0$ and $\theta_0=0$.  If we assume quasi-static condition $\beta=0$ in \eqref{wet-eq-r}, the quasi-static droplet profile $u$  for any fixed $t$
satisfies
\begin{equation}\label{u-quasi}
\begin{aligned}
\frac{\pt }{\pt x}\left( \frac{\pt_x u}{\sqrt{1+ (\pt_x u)^2}} \right)-\kappa u+ \lambda(t)=0 , \quad  a(t) < x< b(t),\\
\int_{a(t)}^{b(t)} u(x,t) \ud x = V,
\end{aligned}
\end{equation}
with boundary condition $u(a(t), t)= u(b(t),t)=0.$ Multiplying \eqref{u-quasi} by $u$ and integration by parts give immediately that $\kappa>0$ implies $\lambda>0.$
In this subsection,  we will  derive the following DAEs for $b(t), \lambda(t), u_m(t)$ in three steps, which completely describes the quasi-static motion
\begin{equation}\label{DAEwet}
\begin{aligned}
b'(t) =- \sigma- \frac{\kappa u_{m}^2(t)}{2}+ \lambda(t) u_m(t)- 1,\\
u_mb - \frac{V}{2}=  \sqrt{2}\int_0 ^{u_m} \sqrt{\frac{u_m-u}{2\lambda - {\kappa}(u_m + u)}} \frac{J(u)}{\sqrt{1+J(u)}} \ud u,\\
b=\sqrt{2} \int_0^{u_m} \frac{1}{\sqrt{u_m-u}}\frac{1}{\sqrt{2\lambda - {\kappa}(u_m + u)}} \frac{J(u)}{\sqrt{1+J(u)}} \ud u,
\end{aligned}
\end{equation}
where $u_m(t)$ is the maximal point of $u(x,t)$ . Here $J(u)$ is the shorthand notation for the function $J(u, \theta)=-\frac{\kappa (u^2-u_m^2)}{2} + \lambda (u- u_m)+1$, where $\cos \theta$ is solved from \eqref{jmn}.

 {\blue Recall $\theta$ is the contact angle such that 
 $\tan \theta = - \pt_x u |_b.$ Thus we have \eqref{cos} and from \eqref{jmn}  the boundary condition  in \eqref{quasi-eq} becomes
$$b'= -\sigma-\frac{1}{\sqrt{1+(\pt_x u)^2|_b}}= -\sigma-\cos \theta = - \sigma- \frac{\kappa u_{m}^2(t)}{2}+ \lambda(t) u_m(t)- 1$$
Then 
from Lemma \ref{lem_re}, the   first equation in \eqref{DAEwet} comes from the first equation in \eqref{rel_n}. }

The second equation is combination of the second and third equation in \eqref{rel_n},
\begin{equation} \label{desin00}
\begin{aligned}
u_m b - \frac{V}{2} 
=&  \int_0 ^{u_m} (u-u_m) \frac{-J(u)}{\sqrt{1-J(u)^2}} \ud u\\
=&  \sqrt{2}\int_0 ^{u_m} \sqrt{\frac{u_m-u}{2\lambda - {\kappa}(u_m + u)}} \frac{J(u)}{\sqrt{1+J(u)}} \ud u.
\end{aligned}
\end{equation} 

On the other hand, to desingularize $X_u$ in the numerical implementation, denote $\psi:=\sqrt{u_m -u}$. Then we have
\begin{equation}
    J(u)=1-[ \lambda-(2u_m -\psi^2) \frac{\kappa}{2}]\psi^2.
\end{equation}
{\blue Recall \eqref{ym000} and particularly in the case $u_{xx}(0)<0$, we have
$$\frac{u_{xx}}{\sqrt{1+u_x^2}^3} \Big|_{x=0} = \pt_x \bbs{\frac{u_x}{\sqrt{1+u_x^2}}}\Big|_{x=0}= \kappa u_m -\lambda <0.$$ }
Thus one can check
\begin{align*}
    \frac{\ud \psi}{\ud x}= \frac{1}{2\psi} \frac{\sqrt{1-J(u)^2}}{J}= \frac{1}{2\psi} \frac{\sqrt{1-J(u)}\sqrt{1+J(u)}}{J}
    =\frac{\sqrt{2 \lambda-\kappa(2u_m-\psi^2)}}{2\sqrt{2}}\frac{\sqrt{1+J(u)}}{J(u)}>0
\end{align*}
for all $0\leq x \leq b.$
Therefore there is no singularity for $\frac{\ud X}{\ud \psi}$.
Integrating $\frac{\ud X}{\ud \psi}$ for $x$ from $0$ to $b$ while $\psi$ from $0$ to $\sqrt{u_m}$, we have
\begin{align}\label{alg}
    b=&\int_0^{\sqrt{u_m}} \frac{2 \sqrt{2}}{\sqrt{2\lambda-\kappa(u_m+u)}}\frac{J(u)}{\sqrt{1+J(u)}} \ud \psi,
\end{align}
which yields the third equation in \eqref{DAEwet}.
 However, to implement this singular integral we need to cluster more points at the singular point near $u_m$, so we use the desingularized   midpoint rule suggested by $\psi=\sqrt{u_m-u}$. Let $\tau=\frac{\sqrt{u_m}}{N}$, $\psi_{i+\frac12}:=(i+\frac12)\tau$ and $u_{i+\frac12}:= u_m-\psi_{i+\frac12}^2$. Then \eqref{alg} can be approximated by
\begin{equation}\label{desin-dis}
b\approx 2\sqrt{2} \sum_{i=0}^{N-1} \frac{\tau}{\sqrt{2\lambda - {\kappa}(u_m + u_{i+\frac12})}} \frac{J(u_{i+\frac12})}{\sqrt{1+J(u_{i+\frac12})}}.
\end{equation}

 With the desingularized formula \eqref{desin-dis},  there is no singularity in the DAEs \eqref{DAEwet} for $b(t), \lambda(t), u_m(t)$, so it  can be solved efficiently and accurately by any ODE solver such as \textit{ode15s} in Matlab, whose results will be used to check the accuracy for our PDE solvers.  Furthermore,  we can solve the capillary surface
$X(u, t)$ by the integral formula \eqref{intg}.
Finally, we give the
equilibrium solution for DAEs \eqref{DAEwet}. Taking $b'(t)=0$ in \eqref{DAEwet}, we have the algebraic equation
\begin{equation}\label{ag1}
 -\frac{\kappa u_{m}^2(t)}{2}+ \lambda(t) u_m(t)= \sigma+ 1
\end{equation}
 and  can solve  uniquely the steady solution  $(b, u_m, \lambda)$.

\subsubsection{Accuracy check between DAEs and 1st/2nd order PDE solvers}

%

We first use the DAEs solver \textit{ode15s} in Matlab to solve   DAEs \eqref{DAEwet}  
with the initial data
\begin{equation}
\theta_{in} = \frac{1.3\pi}{8}, \quad u_m(0)=1.
\end{equation}
With $u_m(0)=1$, we start the DAEs by first solving the compatible initial data $b(0)$ and $\lambda(0)$ from \eqref{DAEwet}.
The physical parameters in DAEs \eqref{DAEwet} are 
\begin{equation}
{\blue \kappa = 0.1}, \quad  \sigma= - \cos(\theta_f),
\end{equation}
where $\theta_f=\frac{3.9\pi}{8}$ is the Young's angle.
 The final time in \textit{ode15s} is  $T=1$. We obtain {\blue  $b(0)=4.532141803665366$,   $b(1)=3.747880231652922$.}

Compared with the DAEs solution, we show below the accuracy check for the first order scheme in Section \ref{sec-1st-scheme}  and the second order scheme in Section \ref{sec3.2} in Table \ref{table_1st}. The absolute error in \textit{ode15s} is set to be $10^{-13}$, which is smaller than the absolute error in the accuracy check Table \ref{table_1st}. The residual tolerance in  Newton's iteration in the second order scheme is set to be $10^{-12}$.   We use the same parameters $\beta=0,$  {\blue $\kappa=0.1$}, initial angle $\theta_{in}=\frac{1.3\pi}{8}$, final Young's angle $\theta_Y=\frac{3.9\pi}{8}$, final time $T=1$, the  initial boundary $b(0)=4.532141803665366$ and {\blue  choose the same initial spherical cap droplet $u(x,0) = - R_0 \cos\theta_{in} + \sqrt{R_0^2-x^2}, \, R_0 :=\frac{b(0)}{\sin \theta_{in}} $}  in the first/second order schemes. For several $M_n$ listed in the tables, we take time step as $\Delta t= \frac{T}{M_n}$ and  moving grid size $\Delta x= \frac{b(t)-a(t)}{N_n}$ with $N_n=8M_n$. The absolute error $e_n$ between numeric solutions and the DAEs solution $b(1)=3.747880231652922$ is listed in the second column of the tables.
The corresponding order of accuracy $\alpha = \frac{\ln(e_n/e_{n+1})}{\ln(M_{n+1}/M_{n})}$ is listed in the last column of the tables.

\begin{table}[ht]
\caption{Accuracy check: 1st/2nd order schemes in Section \ref{sec-1st-scheme}  v.s. exact quasi-static solution to \eqref{DAEwet} using \textit{ode15s}. Parameters: $T=1$, $\kappa=0.1$, $\theta_Y=\frac{3.9\pi}{8}$, $\theta(0)=\frac{1.3\pi}{8}$, $b(0)=3.832203449327490$, time step $\Delta t= \frac{T}{M}$, $M$ listed on the table, moving grid size $\Delta x= \frac{b(t)-a(t)}{N}$,  $N=8M$. Absolute errors $e_n$ are computed by comparing with $b(1)=3.747880231652922$.}\label{table_1st}
\begin{tabular}{|c|c|c|c|c|}
\hline 
\, &  \multicolumn{2}{c|}{1st order scheme} & \multicolumn{2}{c|}{2nd order scheme} \\ 
\hline 
M & Error at $T=1$ & Order of accuracy & Error at $T=1$ & Order of accuracy \\ 
\hline \hline
$40$ & $e_1=\num{1.528E-3}$ &  \,& $e_1=\num{1.799E-5}$ & \\
$80$ & $e_2=\num{7.613E-4}$ &1.0052  & $e_2=\num{4.623E-6}$ &1.9606\\
$160$ & $e_3=\num{3.799E-4}$ &1.0029  & $e_3=\num{1.173E-5}$ &1.9792\\
$320$ & $e_4=\num{1.897E-4}$ &1.0015 & $e_4=\num{2.963E-7}$ &1.9845\\
\hline
\end{tabular} 
\end{table}

\subsection{Breathing droplet: closed formula and long-time validation}\label{sec4.3}
In this section, we construct a breathing droplet solution motivated by the  spherical cap closed-form solution and use this example to demonstrate a long time validity of our numerical schemes in Section \ref{sec3}.

Denote the mean curvature of the capillary surface $u$ in the direction of outer normal as $H$, which is given by $H= -\nabla \cdot\left( \frac{\nabla u}{\sqrt{1+|\nabla u|^2}} \right)$ in the graph representation. Here $u$ is the piecewise graph representation of capillary surface.
 When $\kappa=0$, the governing equation for the quasi-static dynamics becomes $H=\lambda$, where $\lambda$ is a function of $t$. This means the quasi-static profile has constant mean curvature $\lambda(t)$ everywhere on the capillary surface. Assume the initial droplet has the wetting domain $\{x\in \mathbb{R}^{d-1}; \,|x|\leq b(0)\}$. Due to the rotation invariance for both equation and initial wetting domain, the  solution will remain axially symmetric, denoted as $u(r,t)$.  As a consequence,  the quasi-static profile is a spherical cap, whose center may be above the ground. 
To describe this spherical cap  solution, we denote the height of the center as $u^*(t)\in \mathbb{R}$. 
Furthermore, notice the mean curvature of a $d$-dimensional ball is $H=\frac{d-1}{R}$, where $R$ is the radius of the spherical cap.
 
Consider the partially wetting case in 3D when the droplet is represented by the single graph function $u(r,t)$, $0\leq r \leq b(t)$.
Using  $H= -\nabla \cdot\left( \frac{\nabla u}{\sqrt{1+|\nabla u|^2}} \right)= - \frac{1}{r} \pt_r \left( \frac{r \pt_r u}{\sqrt{1+ (\pt_r u)^2}} \right)=\lambda(t)=\frac{2}{R(t)}$, we can solve
\begin{equation}
u_m(t) - u(r,t) = \frac{2}{\lambda(t)} \left(1-\sqrt{1-\left(\frac{\lambda(t)r}{2}\right)^2}\right).
\end{equation}  
Then  $u^*(t)=\frac{2}{\lambda(t)}-u_m(t)$ and  we have
\begin{equation}\label{cap}
(u(r,t)-u^*(t))^2 + r^2 = R(t)^2.
\end{equation}
For a droplet in the non-wetting case, the capillary surface can not be expressed uniquely by the graph of a function $u(r)$.
In the non-wetting case, in which the center $u^*(t)$ is above the ground, one shall use two graph representation (with same notations) for $0\leq u\leq u^*(t)$ and $u^*(t)\leq u\leq u_m$ respectively; see also the horizontal graph representation in Section \ref{sec4.4}.  The spherical cap formula \eqref{cap} holds true for these two pieces.

Recall the contact angle $\theta$ such as $\tan \theta = - \pt_r u|_{r=b}$.
Then by elementary calculation we obtain the classical spherical cap volume formula  in 3D
\begin{equation}\label{re-b-3}
\frac{V}{b^3}= \frac{\pi}{3\sin^3 \theta} (2-3 \cos \theta+ \cos^3 \theta)
\end{equation}
and
in 2D, the formula becomes
\begin{equation}\label{vol2d}
\frac{V}{X(0)^2}= \frac{\theta}{\sin^2 \theta}-\frac{ \cos \theta}{\sin \theta}.
\end{equation}

\subsubsection{ Construct an exact breathing droplet solution and compare with numerical simulations}
Motivated by the spherical cap solution, to check the long-time validation, we construct a breathing spherical cap solution with a prescribed oscillating contact angle $\theta(t)$ satisfying
\begin{equation}\label{breathing0}
\begin{aligned}
\beta \frac{\pt_t u(x,t)}{\sqrt{1+ (\pt_x u)^2}}= \frac{\pt }{\pt x}\left( \frac{\pt_x u}{\sqrt{1+ (\pt_x u)^2}} \right)-\kappa(t) u+\lambda(t), \quad x\in(-b(t),b(t)),\\
b'(t)= -\sigma(t)-\frac{1}{\sqrt{1+(\pt_x u)^2}},\quad x=b(t),\\
\int_{-b(t)}^{b(t)} u(x,t) \ud x  = V,
\end{aligned}
\end{equation}
where the parameters $\kappa(t), \sigma(t)$ are determined below. 

Now we proceed to derive the formulas $\kappa(t), \sigma(t)$ for this breathing droplet.  
 Given $\theta(t)$ with oscillations, we will first calculate $u(x,t)$ and $b(t)$ from the spherical cap formula and then find $\kappa(t)$ and $\sigma(t)$ such that the PDE \eqref{breathing0} holds with the Lagrangian multiplier $\lambda(t)$. 

Step 1. Given the initial data  $\theta(0)$ and $b(0)$. Calculate volume $V$ from \eqref{vol2d}.

Step 2. Calculate $u(x,t)$ and $b(t)$. From the spherical cap formula \eqref{vol2d}
\begin{equation}
b(t) =\sin\theta(t) \sqrt{\frac{2V}{2\theta(t)-\sin(2\theta(t))}},
\end{equation}
and from $R(t)=\frac{b(t)}{\sin\theta(t)}=\sqrt{\frac{2V}{2\theta(t)-\sin(2\theta(t))}}, \,R(t)-u_m(t)= R \cos \theta(t)$ we have
\begin{equation}\label{cap_u}
u(x,t) = - R(t) \cos\theta(t) + \sqrt{R(t)^2-x^2}, \quad x\in(-b(t), b(t)).
\end{equation}

This construction automatically preserves the volume $V$ and by elementary calculations, we know the following relations
\begin{equation}\label{rrr}
u_x = \frac{-x}{\sqrt{R(t)^2 - x^2}},\quad \sqrt{1+(u_x)^2} = \frac{R(t)}{\sqrt{R(t)^2-x^2}}, \quad \frac{u_x}{\sqrt{1+u_x^2}} = \frac{-x}{R(t)},\quad \frac{R'}{R} = -\frac{b^2}{V} \theta'.
\end{equation}

Step 3. We find $\kappa(t), \sigma(t)$ and $\lambda(t)$ such that \eqref{breathing0} holds.
From the \eqref{rrr},  upto some elementary calculations, we have
\begin{equation}\label{dkappa}
\begin{aligned}
\kappa(t) = -\beta\theta'(t)\left(\frac{b(t)^2}{V}\cos\theta(t)+\sin \theta(t)\right),\\
\lambda(t)=\beta b(t) \theta' \left(-\frac{b(t)^2}{V}\sin\theta(t) + \cos\theta(t)\right)+ \frac{\sin\theta(t)}{b(t)},\\
\sigma(t) =b(t) \theta'(t)(\frac{b(t)^2}{V}-\cot\theta(t))- \cos\theta(t).
\end{aligned}
\end{equation}
Particularly, for the quasi-static case $\beta=0$, we have
\begin{equation}
\kappa(t)=0,\quad  \lambda(t) = \frac{\sin\theta(t)}{b(t)}=\frac{1}{R(t)}.
\end{equation}
The constructed breathing droplet solution can be easily extended to 3D case using \eqref{re-b-3}.

 Let the oscillating contact angle be $\theta(t)=\theta_{in}+ A \sin t$, with $\theta_{in}=\frac{3\pi}{16},\, A=0.2$. Now we show the evolution  of breathing droplet and the periodic recurrence for $[0,30\pi]$. The dynamics of the breathing droplet in Fig. \ref{fig:breath} is computed by  the first order scheme in Section \ref{sec-1st-scheme} with $\kappa(t)$, $\sigma(t)$ in \eqref{dkappa} and with initial wetting domain $[-3,3]$ and initial profile calculated by \eqref{cap_u} for $t=0$. The parameters in the PDE solver are $\beta=0.1$, final time $T=30\pi$, time step $\Delta t = \frac{T}{1500}=0.0628,$  $N=1000$ for moving grids in $(a(t), b(t))$. {\blue The exact solution at $T=29\pi$ based on \eqref{cap_u} is also shown with dotted orange line in the lower left part of Fig. \ref{fig:breath} as a comparison.}
\begin{figure}
\caption{Evolution  of the breathing droplet and its periodic recurrence for $[0,30\pi]$.  Computed by the first order scheme in Section \ref{sec-1st-scheme} with $\kappa(t)$ and $\sigma(t)$ in \eqref{dkappa} and oscillating contact angle $\theta(t)=\theta_{in}+ A \sin t$, with $\theta_{in}=\frac{3\pi}{16},\, A=0.2$. The parameters in the PDE solver are $\beta=0.1$, $T=30\pi$, $\Delta t =0.0628,$  $N=1000$ for moving grids, initial domain $[-3,3]$ and initial $u(x,0)$ calculated by \eqref{cap_u}. Each subfigure shows the breathing droplet at time snapshots $[0, \frac{\pi}{2}, \pi, \frac{3\pi}{2}]$ and the recurrence after $15$ periods. Exact solution at $29\pi$ is shown using dotted orange line as comparison. }\label{fig:breath}
\includegraphics[scale=0.57]
{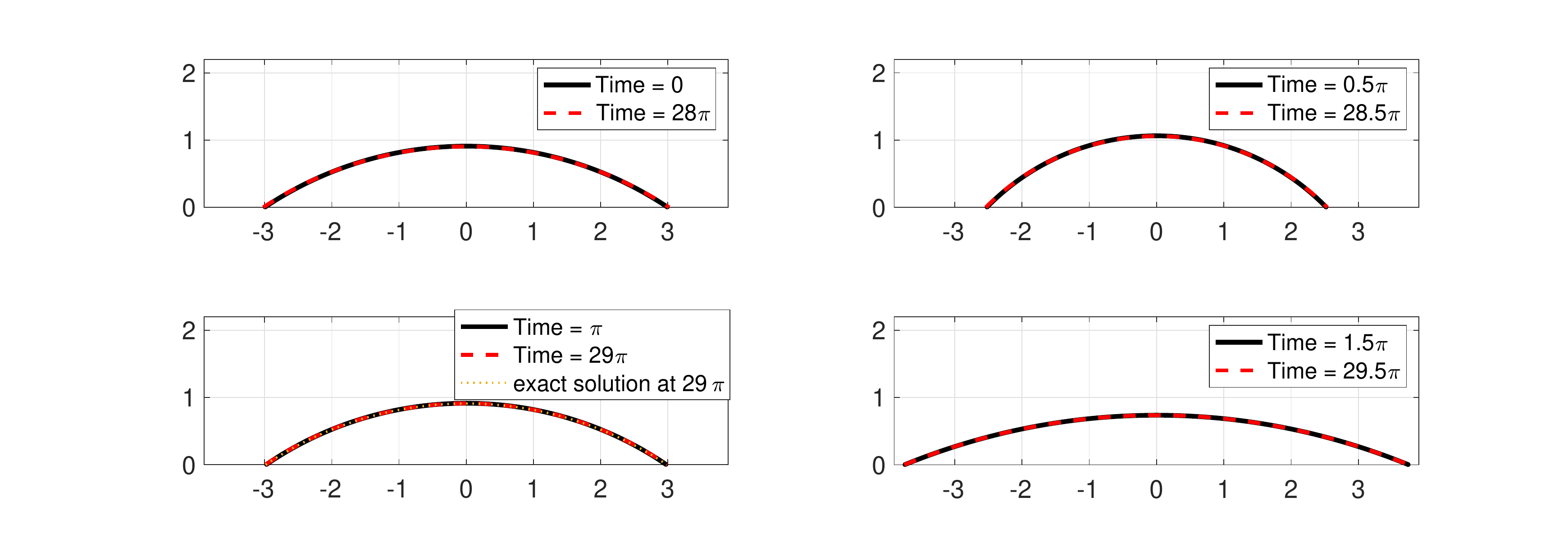} 
\end{figure}


\subsection{Capillary motion of a droplet in a Utah teapot}
The Utah teapot is an important object in computer graphics history, whose 2D cross section can be completely described by several cubic B\'ezier curves \cite{bohm1984survey}. In this section, we will use the bottom and the mouth of the Utah teapot as the inclined  substrate to demonstrate the competition between the gravitational effect and capillary effect for droplets with small Bond number. 
 
We use four points $(x_i, y_i)$, $i=1, \cdots, 4$ to construct a cubic B\'ezier curve $(x(\ell), y(\ell))$ with parameter $\ell\in[0,1]$. Denote the Bernstein basis polynomials as 
\begin{equation}
\begin{aligned}
&B_1(\ell)=(1-\ell)^3,\quad B_2(\ell)= 3(1-\ell)^2\ell, \quad B_3(\ell)= 3(1-\ell)\ell^2, \quad B_4(\ell)= \ell^3.
\end{aligned}
\end{equation}
Then the cubic B\'ezier curve is uniquely given by
\begin{equation}\label{para}
x(\ell) = \sum_{i=1}^4 B_i(\ell) x_i,\quad y(\ell) = \sum_{i=1}^4 B_i(\ell) y_i.
\end{equation}
 Now we construct the bottom and the mouth of the Utah teapot using $10$ points $x_i, y_i$, $i=1, \cdots, 10$ listed in Table \ref{table2}. 
\begin{table}
\caption{Ten points used in B\'ezier curve fitting of geometry of the Utah teapot}\label{table2}
\begin{tabular}{|c|c|c|c|c|c|c|c|c|c|c|}
\hline 
$i$ & 1 & 2 & 3 & 4 & 5 & 6 & 7 & 8 & 9 & 10 \\ 
\hline 
$x_i$ & -2 & $-\frac43$ & $-\frac23$ & 0 & $\frac23$ & $\frac43$ & 2 & $2.655$ & 2.846 & 4 \\ 
\hline 
$y_i$ & 0.78 & 0 & 0 & 0 & 0 & 0 & 0.78 & $1.142$ & 2.146 & 2.5 \\ 
\hline 
\end{tabular} 
\end{table}
For the bottom of the teapot, we use $(x_i, y_i)$ for $i=1, \cdots, 4$ and  $(x_i, y_i)$ for $i=4, \cdots, 7$. For the mouth of the teapot, we use $(x_i, y_i)$ for $i=7, \cdots, 10$. Notice the inclined rough substrate is now expressed by parametric curve  \eqref{para}. Let $\ell(x)$ be the inverse function of $x(\ell)$, then  $w(x)=y(\ell(x))$ and $\theta_0=0$ in \eqref{wet-eq-r}, {\blue which means we do not rotate the Cartesian coordinate system.} To evaluate function $w$ at endpoint $a$ in the numerical implementation, one can use linear interpolation $a=(1-\alpha)x(\ell_i)+ \alpha x(\ell_{i+1})$ for some $\alpha\in[0,1].$

Now we take   the physical parameters as $\kappa=5, \, \beta=3$ and the initial droplet as 
\begin{equation}\label{h0_tea}
h(x,0)=5.2(x-a(0))(b(0)-x)+w(a(0))+\frac{[w(b(0))-w(a(0))](x-a(0))}{b(0)-a(0)}
\end{equation}
 with initial endpoints $a(0)=2.4, b(0)=2.9$. The corresponding effective Bond number can be calculated by \eqref{bo_in} with an approximated effective inclined angle $0.226\pi$, $\Bo=0.1312$.  In the second order scheme, we use $N=600$ moving grid points distributed uniformly in $(a(t), b(t))$.  Different capillary motions corresponding to the relative adhesion coefficients  $\sigma=-0.8$ and   $\sigma=-0.6$ are shown in the upper and lower part of Fig. \ref{fig_teapot}, respectively. With time step $\Delta t = 0.002$, we take final time as $T=16$  for the rolling down case ($\sigma=-0.8$) while we take final time as $T=6$ for the rising up case ($\sigma=-0.6$). In Fig. \ref{fig_teapot}, the  dotted blue line is the initial droplet, red lines are the evolution of the droplet at equal time intervals, and the solid blue line is the final droplet at $T.$
 
   To see clearly the instantaneous contact angle dynamics, we also track the contact angle at $a(t)$ and $b(t)$ for both the rolling down case (left subfigure) and the rising up case (right subfigure)  in Fig. \ref{fig_teapot_angle_f} associated with droplets dynamics in the teapot. Those contact angles vary as the effective slope of the mouth of the teapot  changes, and then tend to the equilibrium Young's angle $\theta_Y$. Particularly, two cusps occur for the rolling down case (left subfigure) in Fig. \ref{fig_teapot_angle_f}  when the advancing (resp. receding) contact line $a(t)$ (resp. $b(t)$) pass  through the corner between the mouth and the body of the Utah teapot. The red line is the dynamics of the contact angle $\theta_a$ at $a(t)$ while the blue line is the dynamics of the contact angle $\theta_b$ at $b(t)$. One can see at the late stage, the sign of both  $a'(t)=\frac{1}{\cos \theta_{0a}}( \cos \theta_a-\cos \theta_Y)$ and $b'(t)=-\frac{1}{\cos \theta_{0b}}( \cos \theta_b-\cos \theta_Y)$ are negative (resp. positive) for the rolling down case (resp. rising up case).

\begin{figure}
\includegraphics[scale=0.38]{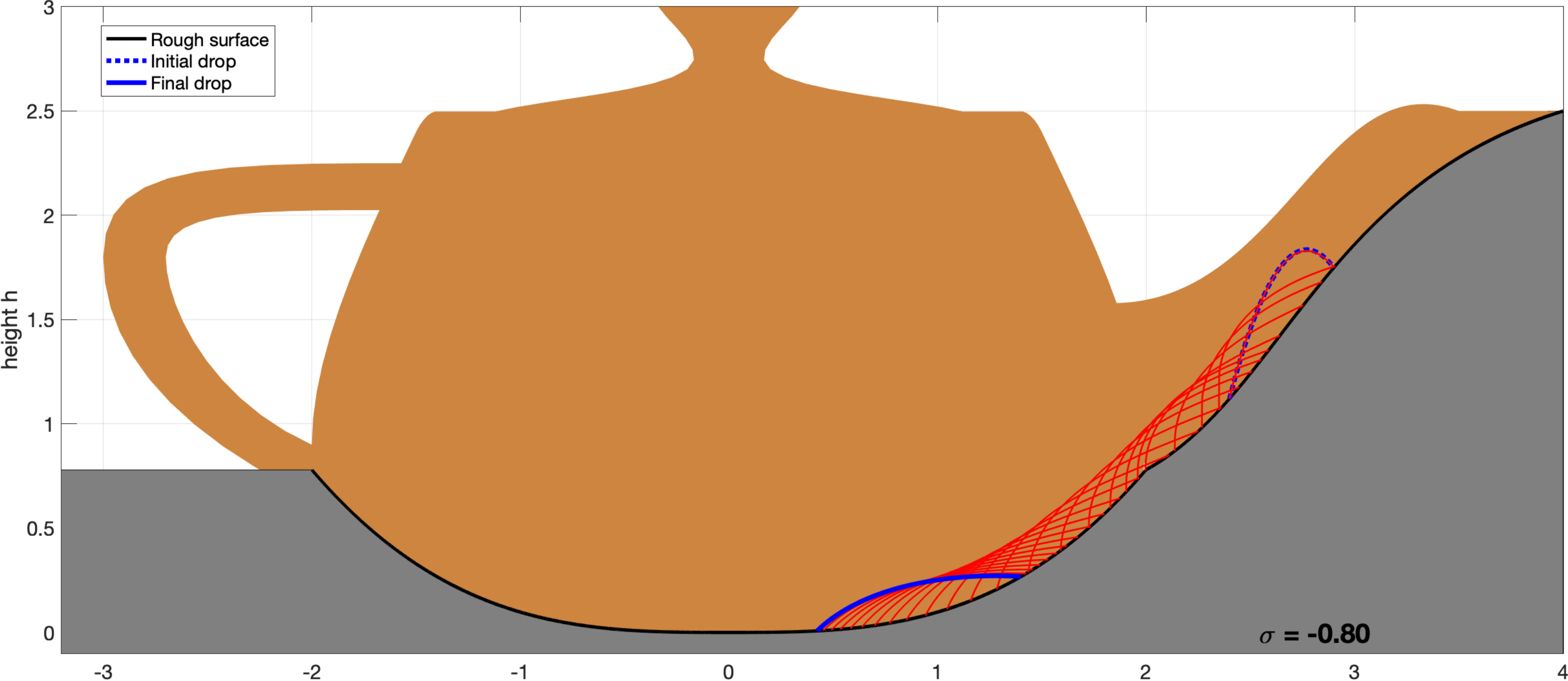}
 \includegraphics[scale=0.38]{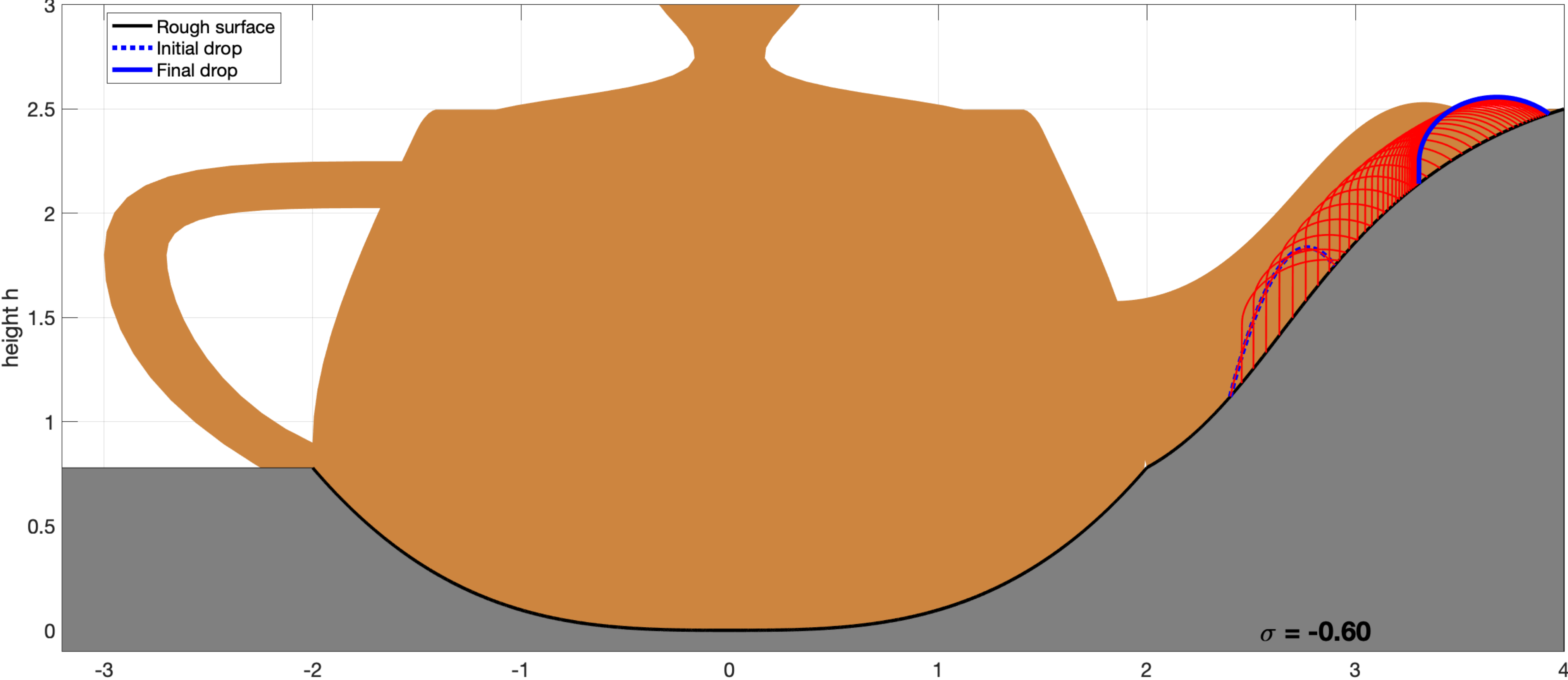} 
 \caption{Evolution of a partially wetting droplet in the Utah teapot at equal time intervals using the second order time-space scheme in Section \ref{sec_2nd_scheme}. Parameters: number of moving grids  $N=600$, time step $\Delta t = 0.002$, $\kappa=5, \, \beta=3$, Bond number $\Bo=0.1312$, initial drop profile given in \eqref{h0_tea} with $a(0)=2.4, b(0)=2.9$.  (upper) Gravity wins: relative adhesion coefficient $\sigma=-0.8$ and  final time $T=16$; (lower) capillary rise: relative adhesion coefficient $\sigma=-0.6$ and final time $T=6$.}\label{fig_teapot}
\end{figure}

\begin{figure}
\includegraphics[width=8.1cm]{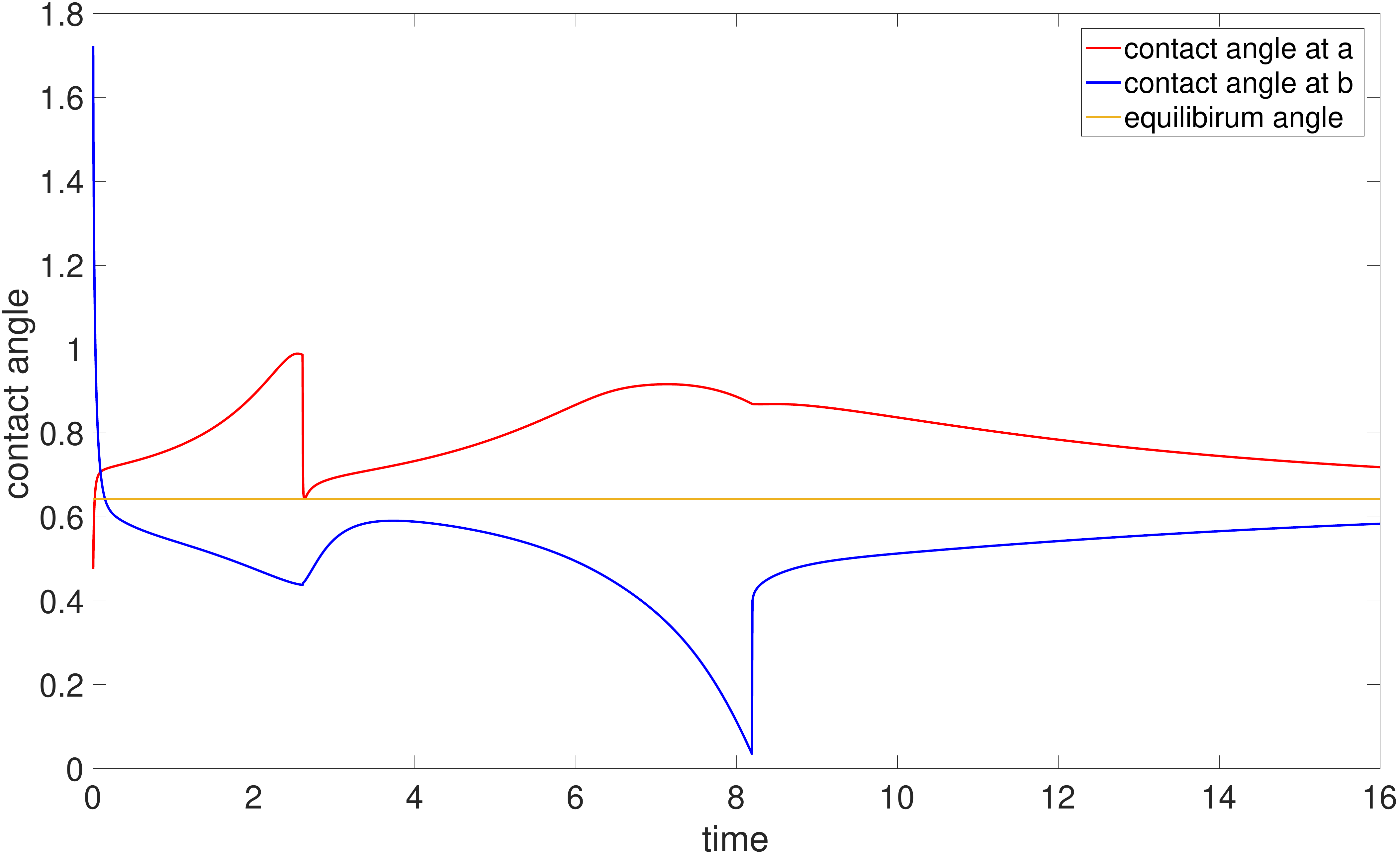} \includegraphics[width=8.1cm]{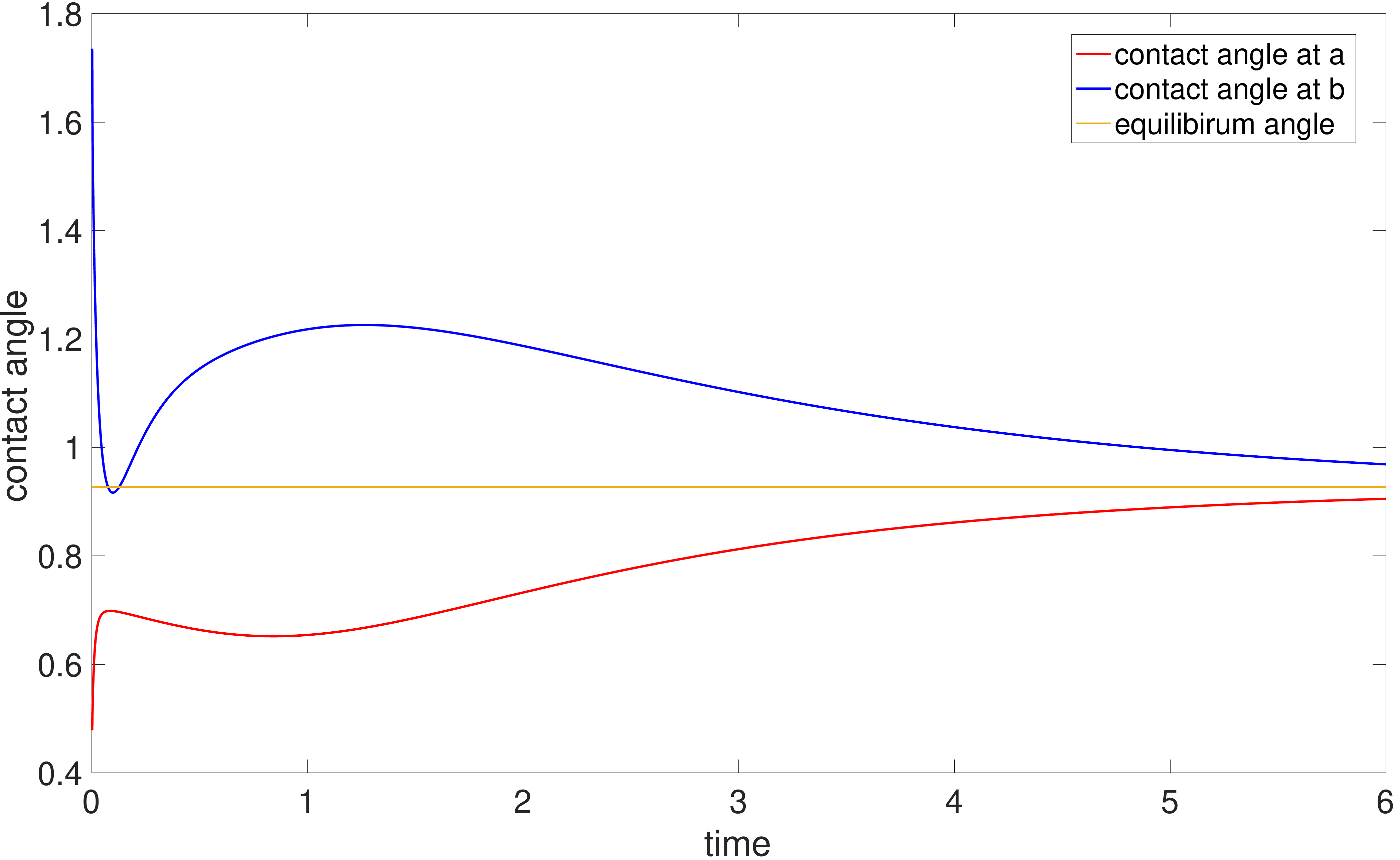} 
\caption{The contact angle dynamics for the associated  droplets dynamics in Fig \ref{fig_teapot}. (Left) Evolution of contact angles for $t\in[0,16]$ for the rolling down case with $\sigma=-0.8$;   two cusps when the advancing (resp. receding) contact line $a(t)$ (resp. $b(t)$) passing through the corner between the mouth and the body of the teapot. (Right) Evolution of contact angles for $t\in[0,6]$ for the rising up case with $\sigma=-0.6$.}\label{fig_teapot_angle_f}
\end{figure}

\subsection{Dynamics of droplets on an inclined  groove-textured surface}
In this section, we show the contact angle hysteresis (CAH) for a droplet on an  inclined rough surface. Gravity will pull the droplet down while CAH will resist its motion. Therefore, one will observe the top of the droplet becomes thin while the bottom of it becomes thick. Besides, the contact line speeds depend  on both the instantaneous contact angle $\theta_a, \theta_b$ and the local slope of the rough surface $\theta_{0a}, \theta_{0b}$ (in \eqref{wet-eq-r}), which changes constantly due to the boundary motion. Consequently, one can observe the contact line speed will change accordingly.

To demonstrate those phenomena, we take the inclined angle $\theta_0=0.3$ and a typical groove-textured surface 
\begin{align}
 w(x)=A(\sin(kx)+\sin(kx/2)+\cos(2kx)),\, A=0.1,\, k=5.\label{w2_r}
\end{align}
We take the physical parameters as $\kappa=0.3, \, \beta=0.3$, relative adhesion coefficient $\sigma=-0.95$ and initial droplet as 
\begin{equation}\label{h0_r1}
h(x,0)=0.08(x-a(0))(b(0)-x)(x^2+3x/2+1)+w(a(0))+\frac{[w(b(0))-w(a(0))](x-a(0))}{b(0)-a(0)}
\end{equation}
 with initial endpoints $a(0)=-3,\, b(0)=3$. The corresponding effective Bond number can be calculated by \eqref{bo_in} with $\theta_0=0.3$. The evolution of the partially wetting droplet is computed by the second order scheme in Section \ref{sec_2nd_scheme}.  We take final time as $T=96$ with time step $\Delta t = 0.08$ and use $N=1000$ moving grids uniformly in $(a(t), b(t))$. We show in Fig. \ref{fig_r1r2} droplet on  rough surface $w(x)$ in \eqref{w2_r}. The green line is the initial droplet, red lines are the evolution of the droplet at equal time intervals, and the blue line is the final droplet at $T=96.$
\begin{figure}
 \includegraphics[scale=0.33]{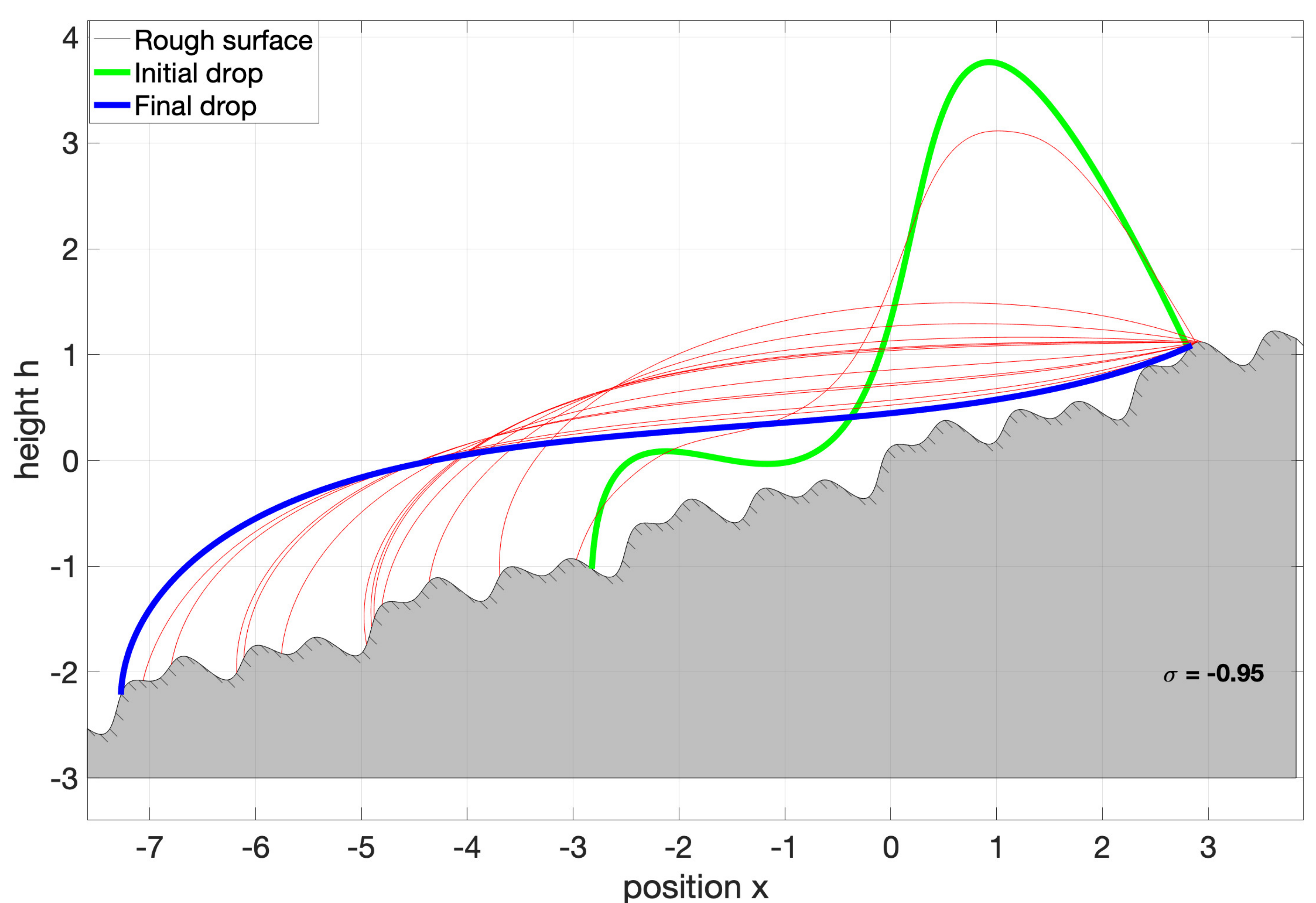}  
\caption{Evolution of a partially wetting droplet on an inclined rough surface at equal time intervals using the second order time-space scheme in Section \ref{sec_2nd_scheme}. Parameters: $\kappa=0.3, \, \beta=0.3$, number of moving grids  $N=1000$, time step $\Delta t = 0.08$,  initial drop profile $h(x,t)$ in \eqref{h0_r1} with initial endpoints $a(0)=-3, b(0)=3$,  final time $T=96$, relative adhesion coefficient $\sigma=-0.95$, inclined effective angle $\theta_0=0.3$.  Contact angle hysteresis happens on rough surface \eqref{w2_r}. }\label{fig_r1r2}
\end{figure}

\subsection{Quasi-static dynamics of Kelvin pendant drop with volume constraint }\label{sec4.4}
In 1886, Lord Kelvin proposed a geometric integration procedure that the quasistatic profile remains no long graph representation and becomes ``repeated bulges"  when the height of the pendant drop exceeds a critical height $u_c$. In this section, we compute the Kelvin pendant drop problem with volume constraint for $\kappa<0$, which is not covered in Proposition \ref{pertur_pf0}. For the cases without volume constraint, we refer to \cite{ pozrikidis_stability_2012}; see also \cite{Finn_Shinbrot_1988, Schulkes_1994} and references therein.  To simulate the ``repeated bulges", which certainly break the vertical single graph representation setting, we will first describe the droplet using inverse function $X(u)$ (in horizontal graph setting) and give the gradient flow formulation in terms of $X(u)$. By solving the DAEs for $X(u)$ with $\kappa<0$, which describes the quasi-static dynamics of a pendant droplet, we will recover multiple interfacial shapes including lightbulb and hourglass shapes with different Bond numbers.  We refer to \cite{pozrikidis_stability_2012} for simulations and stability analysis of a liquid drop in hydrostatic states.

For the case the capillary surface can not be expressed uniquely by the graph function $u(x)$, we use a horizontal graph setting  $X(u)$.  {\blue
From \cite[Theorem 1.1]{wente1980symmetry},   for the quasi-static dynamics of a pendant droplet, there is  only one vertical axis of symmetry such that any nonempty intersection of $u$ with a horizontal  hyperplane are disks centered at that vertical axis.  Thus the maximum $u_m$ of the droplet is attained uniquely at $X(u_m)=0$ and there exists a unique horizontal graph representation $X(u)$, $0\leq u \leq u_m$.
  For simplicity,  we also  assume the full dynamics of a pendant droplet has a horizontal graph representation in the derivation of \eqref{dewet-eq-n}.}   Now we identify the droplet on the right of $x=0$ as
\begin{equation}\label{inverse}
A:=\{(u,x);~0\leq  u\leq u_m, 0\leq x \leq X(u) \}.
\end{equation}

Next we give the following governing equations  for a 2D droplet in terms of $X(u)$
\begin{equation}\label{dewet-eq-n}
\begin{aligned}
\beta\frac{\pt_t X }{\sqrt{1+X_u^2}}= \pt_u\left( \frac{X_u}{\sqrt{1+X_u^2}} \right) - \kappa u + \lambda,\quad 0\leq u \leq u_m\\
X(u_m)=0,\quad X_u(u_m)=-\8,\\
\pt_t X(0, t)= \frac{X_u}{\sqrt{1+ X_u^2}} \Big|_{u=0} - \sigma,\\
\int_0^{u_m} X(u) \ud u=V/2.
\end{aligned}
\end{equation}
The derivation  using a gradient flow on a Hilbert manifold is similar to \eqref{wet-phy0}; see Appendix \ref{app_nonwet}.

To compute the Kelvin pendant droplet problem with volume constraint, we consider the quasi-static dynamics by taking $\beta=0$ in \eqref{dewet-eq-n}. After desingularization, the quasi-static dynamics can be recast as the following DAEs
 for $(X(0,t), u_m(t), \theta(t), \lambda(t))$
\begin{equation}\label{DAE_nonwet}
\begin{aligned}
\pt_t X(0, \cdot)=-\cos \theta-\sigma,\\
J(u_m, \theta) =1,\\
X(0, \cdot)  =  \frac{V}{2u_m}+\frac{1}{u_m}\int_0 ^{u_m} (u-u_m) \frac{-J(u, \theta)}{\sqrt{1-J(u,\theta)^2}} \ud u,\\
\frac{V}{2} 
=  \int_0^{\sqrt{u_m}} \frac{-2u J(u,\theta)}{\sqrt{1+J(u,\theta)}}\frac{\sqrt{u_m-u}}{\sqrt{1-J(u,\theta)}} \ud \psi,
\end{aligned}
\end{equation}
where $J(u, \theta)$ defined in \eqref{def_J}. {\blue The third desingularized formula comes from \eqref{desin00} while
the last desingularized formula comes from  the volume formula in \eqref{rel_n}. Using $\psi^2=u_m -u$, we rewrite $V$ as
\begin{equation}\label{D9}
\begin{aligned}
\frac{V}{2} =  \int_0^{\sqrt{u_m}} \frac{-2u J(u,\theta)}{\sqrt{1+J(u,\theta)}}\frac{\sqrt{u_m-u}}{\sqrt{1-J(u,\theta)}} \ud \psi 
\end{aligned}.
\end{equation}
By L'Hopital's law, $\lim_{u\to u_m} \frac{u_m-u}{1-J(u)}= \lim_{u\to u_m} \frac{-1}{\kappa u-\lambda}\neq 0$ provided $\kappa u_m \neq \lambda$. 
}

After solving  $(X(0,t), u_m(t), \theta(t), \lambda(t))$ from the above DAEs, we can further compute the formula for $X(u,t)$ 
\begin{align}\label{tm_xu}
X(u, \cdot) = \int_u^{u_m} \frac{J(u,\theta)}{\sqrt{1-J(u)^2}} \ud u.
\end{align}

{\blue We  use the DAEs solver \textit{ode15s} in Matlab to solve the DAEs \eqref{DAE_nonwet}  
with  different initial data  $\theta_{in},\, u_m(0)$ and physical parameters $\kappa$, Young's angle $\theta_Y$; see Table \ref{Tpendant}. We start the DAEs by first solving the compatible initial data $b(0)=X(0,0)$ and $\lambda(0)$ from \eqref{DAE_nonwet}. At the final time $T=4$, the final bulge  shape of the pendant droplet (blue line)    is illustrated in (left) Fig. \ref{K_wet}, while the  final lightbulb shape of the pendant droplet is  illustrated in (right) Fig. \ref{K_wet}. The corresponding Bond numbers and the final contact points $b(4)$ are also shown in Table \ref{Tpendant}.}

\begin{table}
\caption{List of the parameters in DAEs solver for pendant droplets}\label{Tpendant}
\begin{tabular}{|c|c|c|c|c|c|c|c|}
\hline 
\, & $\theta_{in}$ & $u_m(0)$ & $\theta_Y$ & $b(0)$ & $\kappa$ & $\Bo$ & $b(4)$ \\ 
\hline 
Bulge shape & $ \frac{3\pi}{16}$ & 0.3 & $\frac{2.7\pi}{8}$ & $0.37045$ & $-28.028$ & 1.213 &  0.17438\\ 
\hline 
Lightbulb shape & $\frac{5\pi}{16}$ & 0.3 & $\frac{4.7\pi}{8}$ & $0.36172$ & $-15.05$ & 0.708 &  0.05020 \\ 
\hline 
\end{tabular} 
\end{table}

\begin{figure}

\caption{Quasi-static dynamics of Kelvin pendant droplets with volume constraint. The evolution of pendant droplets at equal time intervals are computed using DAEs \eqref{DAE_nonwet} with initial data $u_m(0)=0.3$ and final time $T=4$. The left figure has an initial angle $\theta_{in} = \frac{3\pi}{16}$, final Young's angle $\theta_f=\frac{2.7\pi}{8}$ and the physical parameters $\kappa=-28.028, \Bo=1.213$.  The right figure has an initial angle $\theta_{in} = \frac{5\pi}{16}$, final Young's angle $\theta_f=\frac{4.7\pi}{8}$ and the physical parameters $\kappa=-15.0, \Bo=0.708$. }\label{K_wet}
 \includegraphics[scale=0.8]{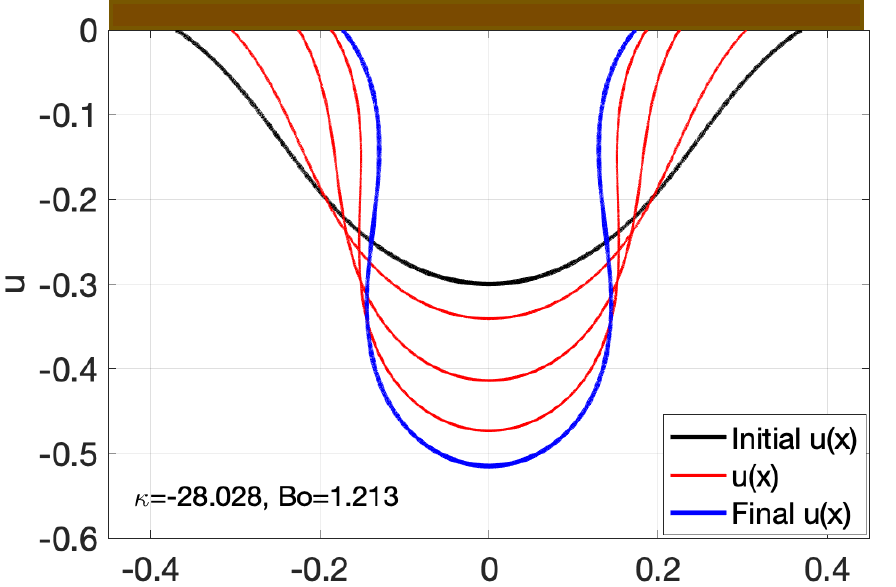} \includegraphics[scale=0.8]{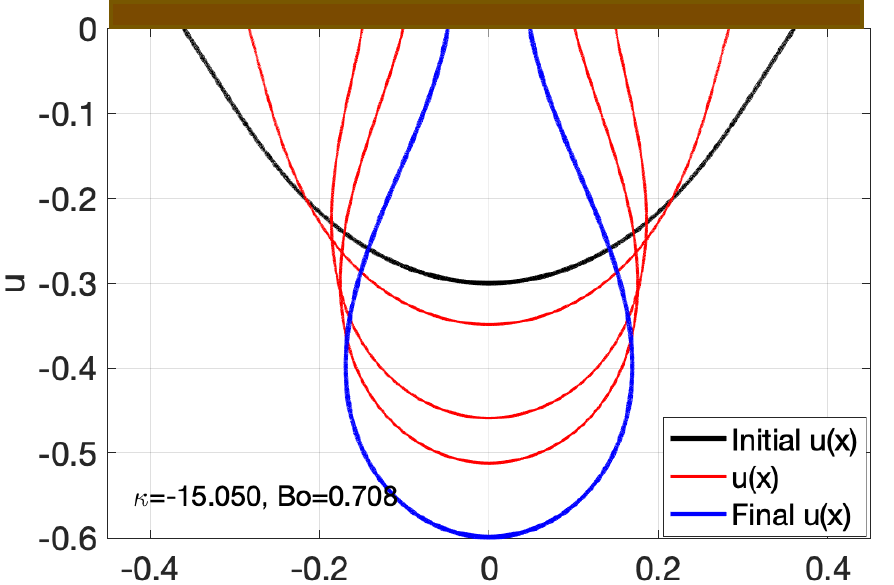} 
\end{figure}

\section*{Acknowledge}
The authors wish to thank Professor Tom Witelski for  valuable suggestions and  particularly thank anonymous  reviewers for insightful criticism and comments, which significantly improved this manuscript. J.-G. Liu was supported in part by the National Science Foundation (NSF) under award DMS-1812573.

\bibliographystyle{abbrv}
\bibliography{dropbib}

\appendix

\section{Dynamics of a droplet as a gradient flow on manifold}\label{app_model}
We use a Hilbert manifold \cite{Klingenberg_Wilhelm} to describe the configuration states
\begin{equation}\label{mm-n}
\mm:= \{(\Gamma, u);\, \Gamma:=\pt D \in C^1, \, u\in H_0^1(D),~u\geq 0 \text{ on }D\}
\end{equation}
and use a trajectory on this manifold to describe the dynamics of the droplet.
Consider a trajectory $\pcon(t) \in \mm$ starting from initial state 
$\pcon(0)=\{\Gamma(0), u(x,y,0)\}\in\mm$,
\begin{equation}
\pcon(t)=\{\Gamma(t), u(x,y,t)\}\in \mm, \quad t\in[0,T]. 
\end{equation}
It is natural to assume the motion of the droplet $\pcon(t)$ is modeled by  a gradient flow on  manifold $\mm$ described above. (i) The dynamics is driven by the  free energy $\F(\pcon)$ on manifold $\mm$; (ii) The dissipation mechanism of the dynamics is described by a Riemannian metric $g_\eta$ on the tangent plane $T_\eta\mm$, which is discussed in \eqref{metric-n} below.

We will use the vertical velocity $v=\pt_t u$ and the contact line velocity $\Gamma_t=v_{cl}n_{cl}$ to describe the tangent plane $T_{\pcon}\mm$.
Since the geometric motion has an obstacle condition $u\geq 0$,  manifold $\mm$ has a boundary, i.e., $\{\eta\in\mm; u(x,y)=0 \text{ for some }(x,y)\in D\}$ (when the droplet has a splitting-type topological change). On the boundary, the tangent plane is not a linear space and has the restriction described below. 
Notice 
\begin{equation}\label{com-angle}
\begin{aligned}
\frac{\ud u(\Gamma(t),t)}{\ud t} = \pt_t u(\Gamma(t),t) + \nabla  u(\Gamma(t),t)\cdot {\Gamma_t}
=&\pt_t u(\Gamma(t),t) +  (\nabla  u(\Gamma(t),t)\cdot n_{cl}) v_{cl} \\
=&\pt_t u(\Gamma(t),t) -|\nabla  u(\Gamma(t),t)| v_{cl}=0,
\end{aligned}
\end{equation}
where we used the fact that $\nabla  u(\Gamma(t),t)\cdot n_{cl}=-|\nabla  u(\Gamma(t),t)|$ in the graph representation.
The tangent plane at $\eta$ is given by
\begin{equation}\label{tm-n}
T_{\pcon} \mm:=\{(v_{cl}, v); ~v - |\nabla  u| v_{cl} =0 \text{ on } \Gamma,~~~v+u\geq 0 \text{ in } D\}.
\end{equation}
The last inequality in the definition of $T_\eta\mm$ in \eqref{tm-n} is only effective for $\eta$ on the boundary of the manifold $\mm$ where the obstacle occurs.  
Define the contact angles (inside the droplet $A$) as 
\begin{equation}\label{con-ang-n}
\tan \theta_{cl} := | \nabla  u(\Gamma)|;
\end{equation}
see Fig \ref{fig:ill} (a).
Then the physical meaning of the constraint in \eqref{tm-n} comes  naturally from the fact that the contact angles are proportional to the  quotient of the vertical velocity and the horizontal velocity,  i.e.,
\begin{equation}
\tan \theta_{cl} =| \nabla  u(\Gamma)|= \frac{\pt_t u|_\Gamma}{v_{cl}}.
\end{equation}
Notice the outer normal $n=\frac{1}{\sqrt{1+|\nabla u|^2}}(-\nabla u, 1)$ on the capillary surface and due to $\frac{-\nabla u}{|\nabla u|}\Big|_\Gamma =  n_{cl}$, $n\big|_\Gamma =\frac{1}{\sqrt{1+|\nabla u|^2}}(|\nabla u|n_{cl}, 1)$. Here the values on $\Gamma$ are understood as one-side limit from the interior of $D$.
Using the contact angle $\theta_{cl}$, we have
\begin{equation}
 n |_\Gamma = (\sin \theta_{cl} n_{cl}, \cos \theta_{cl}), \quad v_n |_\Gamma= \sin \theta_{cl} v_{cl}.
\end{equation}

Now we describe the dissipation mechanism of the dynamics.
From Rayleigh's dissipation function \eqref{RQ}, since $2Q$ is the rate of energy dissipation due to friction \cite{goldstein2002classical}, it is natural to 
introduce the Riemannian metric $g_{\eta}$ on $T_{\eta}\mm\times T_{\eta}\mm$ below. For any $q_1=(v_{cl1}, v_1),\, q_2=(v_{cl2},v_2)\in T_{\pcon}\mm$,
\begin{equation}\label{metric-n}
g_{\pcon} (q_1, q_2):= \mathcal{R}\int_{\Gamma}v_{cl1} v_{cl2} \ud s+ \zeta\int_{D} v_1   v_2 \frac{\ud x \ud y}{\sqrt{1+ |\nabla u|^2}}.
\end{equation}
For similar derivations of the dynamics of droplets using a variational approach with various free energies and the same Riemannian metrics \eqref{metric-n}, we also refer to  \textsc{Davis} \cite{Davis_1980},  \cite{Qian_Wang_Sheng_2006},  \textsc{Doi} \cite{xu2016variational} and \cite{Grunewald_Kim_2011}.

  We now derive the gradient flow of $\F(\pcon)$ defined in \eqref{energy} on manifold $\mm$ with respect to the Riemannian metric $g_{\pcon}$.
For an arbitrary trajectory $\tilde{\pcon}(s)=\{\tilde{\Gamma}(s), \tilde{u}(x,y,s)\}$ (physically known as a virtual displacement) passing $\tilde{\pcon}(t)=\eta(t)$ at the tangent direction
$q:= \tilde{\pcon}'(t)= \{\tilde{v}_{cl}, \tilde{v}\}\in T_{\pcon(t)} \mm$,
we know
\begin{equation}\label{com-con-n}
\tilde{v}(\Gamma)=  |\nabla  u(\Gamma(t),t)| \tilde{v}_{cl}.
\end{equation}
To ensure the volume preserving condition $\int_{D(t)} u \ud x\ud y = V, \, t\in[0,T]$, we consider the gradient flow of extended free energy $\F(\eta,\lambda)$ on manifold $\mm\times \mathbb{R}$ for $\eta(t)\in \mm$ and a Lagrangian multiplier $\lambda(t)$
\begin{equation}
\F(\eta(t), \lambda(t))= \F(\eta(t))- \lambda(t)(\int_{D(t)} u(t) \ud x \ud y -V).
\end{equation}
Then the gradient flow of $\F(\eta,\lambda)$ with respect to Riemannian metric $g_\eta$ defined in \eqref{metric-n}
is
\begin{equation}
-g_{\pcon(t)}(\tilde{\eta}'(t), \eta'(t)) \leq  \frac{\ud }{\ud s}\big|_{s=t^+} \F(\tilde{\pcon}(s), \tilde{\lambda}(s)) =  \frac{\ud }{\ud s}\big|_{s=t^+} \F(\tilde{\pcon}(s)) - \la \tilde{v}, \lambda(t)\ra - \tilde{\lambda}'(t)(\int_{D(t)} u \ud x \ud y -V).
\end{equation} 
for any $\tilde{\eta}'\in T_{\pcon(t)} \mm$, {where $\la \tilde{v}, \lambda(t)\ra = \lambda(t)\int_{D(t)}\tilde{v} \ud x \ud y$ is the inner product of $L^2(\mathbb{R})$.}
For a generic free energy density $G(u,\nabla u)$, we calculate the first variation  
$ \frac{\ud }{\ud s}\big|_{s=t^+} \int_{\tilde{D}(t)} G( \tilde{u}(x,y,s), \nabla \tilde{u}(x,y,s)) \ud x \ud y $ below.
It includes three typical free energy examples: (i) Dirichlet energy $G(u,\nabla u)=\frac12|\nabla u|^2+\sigma$, where $\sigma$ is a constant; c.f. \cite{caffarelli2006homogenization, Kim_Mellet_2014, feldman_dynamic_2014, xu2016variational};
(ii) Area functional $G(u,\nabla u)=\sqrt{1+ |\nabla u|^2}+\sigma$;  c.f. \cite{Caffarelli_Mellet_2007, Caffarelli_Mellet_2007a, Feldman_Kim_2018}, as a consequence, the choice of the last term in $g_\eta$ gives the mean curvature flow; c.f. \cite{Giga_2006}; (iii) free energy for droplets on inclined rough surface; see \eqref{energy-r} below.

From \eqref{com-con-n} and the Reynolds transport theorem,
\begin{align}
&\frac{\ud }{\ud s}\big|_{s=t^+} \int_{\tilde{D}(t)} G( \tilde{u}(x,y,s), \nabla \tilde{u}(x,y,s)) \ud x \ud y \nonumber\\
=& \int_{\Gamma(t)} G|_\Gamma \tilde{v}_{cl} \ud s + \int_{D(t)} \pt_u G \tilde{v} + \pt_{ \nabla u} G \cdot \nabla \tilde{v} \ud x \ud y \nonumber\\
=&  \int_{\Gamma(t)} G|_\Gamma \tilde{v}_{cl} \ud s + \int_{D(t)} (\pt_u G - \nabla \cdot  (\pt_{ \nabla u} G) ) \tilde{v} \ud x \ud y +  \int_{\Gamma(t)}\tilde{v} (n_{cl} \cdot \pt_{\nabla u} G)\ud s \nonumber\\
=& \int_{\Gamma}[G+|\nabla  u |(n_{cl} \cdot \pt_{\nabla u} G) ]\big|_\Gamma \tilde{v}_{cl} \ud s +   \int_{D(t)} (\pt_u G - \nabla \cdot  (\pt_{ \nabla u} G) ) \tilde{v} \ud x \ud y. \label{var}
\end{align}
Notice from $\tilde{\pcon}(t)=\eta(t)$, the Riemannian metric $g_{\pcon(t)}$, 
\begin{align}
g_{\pcon} (\tilde{\eta}(t), \eta(t)):= \mathcal{R}\int_{\Gamma(t)}\tilde{v}_{cl}  v_{cl} \ud s+ \zeta\int_{D(t)}  \tilde{v}   \frac{\pt_t u}{\sqrt{1+ |\nabla u|^2}}\ud x \ud y.
\end{align}
where $\frac{\pt_t u(x,t)}{\sqrt{1+ |\nabla  u|^2}}$ is the normal velocity in the direction of the outer normal.

Hence by taking different $\tilde{\eta}'\in T_{\pcon(t)}\mm$, the governing equations 
for  $u(\cdot,t)\in H_0^1(D(t))$ and $\lambda(t)$ are
\begin{equation}\label{eq_G_nD}
\begin{aligned}
\mathcal{R}v_{cl} =  -\left[G+|\nabla  u|(n_{cl} \cdot \pt_{\nabla u} G)\right ]\big|_\Gamma&,\\
\int_{D(t)} \left[ \zeta\frac{\pt_t u}{\sqrt{1+ |\nabla u|^2}}+ (\pt_u G -  \nabla \cdot  (\pt_{ \nabla u} G) )-\lambda(t)\right]v \ud x \ud y \geq  0&,\\
 \forall v\in H_0^1(D(t)), v(x)+u(x,t) \geq& 0,\\
\int_{D(t)} u \ud x \ud y =V&
\end{aligned}
\end{equation}
with initial data $\eta(0)=\{\Gamma(0), u(x,y,0)\}$ and initial volume $V$.

The variational inequality formulas above are able to describe the merging and splitting of several drops by using  numerical schemes for  parabolic variaional inequalities (PVI), such as splitting method with projecting operators, c.f. \cite{Kato_Masuda_1978, Lions_Mercier_1979, gao2020projection}.  {\blue However, after the splitting of two droplets, the original two-phase interface becomes three-phase triple point at the splitting point, thus the purely PVI dynamics can not describe the contact line dynamics for the emerged triple point, i.e.,  new contact line $\tilde{\Gamma}$. Indeed, according to PVI, the motion of the emerged triple points $\tilde{\Gamma}$ is only guided by the projection of the motion by mean curvature on the substrate, thus  is different from the true physical case, i.e., $\mathcal{R}v_{cl} =  -\left[G+|\nabla  u|(n_{cl} \cdot \pt_{\nabla u} G)\right ]\big|_{\tilde{\Gamma}}$.    Hence we need to detect the splitting points, enforce this contact line dynamics for  new $\tilde{\Gamma}$ and then treat them separately as two droplets.} 
 Existence of global weak solutions including topological changes (splitting and merging) is studied in \cite{Grunewald_Kim_2011} using a continuum limit of a time discretization based on variational methods.

\subsection{Gradient flow for a single droplet: without merging and splitting}\label{appA.1}
 We call a single droplet as a  sessile drop if $u(x,y,t)>0$ for $(x,y)\in D(t)$ with the  gravity downwards, i.e., $g>0$. Another  scenario is when a light drop is in a heavy fluid, the drop experiences buoyancy due to gravity. In this case, we call a single droplet as a  pendant drop if $u(x,y, t)>0$ for $(x,y)\in D(t)$ with the  gravity upwards, i.e.,  $g<0$. Another equivalent problem is a drop pendant on ceiling with $u<0$ for $(x,y)\in D(t)$ and $g>0$. In this paper, we only consider nonnegative $u$ and use negative $g$ for a pendant droplet.
For those single sessile/pendant drop cases, the variational inequalities become variational equalities and the weak formulations can be equivalently converted to a strong-form PDE. 
Therefore the governing equations  with volume constraint \eqref{eq_G_nD} become
\begin{equation}
\begin{aligned}
&\zeta\frac{\pt_t u}{\sqrt{1+ |\nabla u|^2}}+ (\pt_u G -  \nabla \cdot  (\pt_{ \nabla u} G) )-\lambda(t)=0 , \quad  \text{ in } D(t)
,\label{eq1-pde-n}\\
& u(\Gamma(t), t)=0,\\
&\mathcal{R}v_{cl} = -\left[G+|\nabla  u|(n_{cl} \cdot \pt_{\nabla u} G)\right ],  \quad \text{ on }\Gamma,\\
& \int_{D(t)} u \ud x \ud y = V,
\end{aligned}
\end{equation}
where $V$ is the initial volume of the droplet.
Particularly, for the energy \eqref{energy}, we have
\begin{align*}
G=\gamma_{lg}\sqrt{1+|\nabla u|^2}+ (\gamma_{sl}-\gamma_{sg}) + \rho g \frac{u^2}{2},\qquad 
\pt_u G = \rho g u, \quad \pt_{\nabla u} G = \frac{\gamma_{lg} \nabla u}{\sqrt{1+|\nabla u|^2}}.
\end{align*}
Therefore the governing equations are \eqref{eq_nD}.
We remark that when $G(u,\nabla u)=\frac12|\nabla u|^2+\sigma$, the kinematic boundary condition for the contact line speed $v_{cl}$ in \eqref{eq_G_nD} becomes $\frac{\mathcal{R}}{\gamma_{lg}}v_{cl}=\frac12|\nabla u|^2-\sigma$, which recovers the kinematic boundary condition used in  \cite{Grunewald_Kim_2011, Kim_Mellet_2014, feldman_dynamic_2014}.

\begin{proof}[Proof of Proposition \ref{Prop2.1}]
Statement (i) comes from \eqref{com-angle} directly.

For Statement (ii), from  the Reynolds transport theorem and similar to \eqref{var}, we have
\begin{align*}
&\frac{\ud }{\ud t} \int_{{D}(t)} G(u(x,y,t), \nabla {u}(x,y,t)) \ud x \ud y\\
&= \int_{\Gamma}[G+|\nabla  u |(n_{cl} \cdot \pt_{\nabla u} G) ]\big|_\Gamma {v}_{cl} \ud s +   \int_{D(t)} (\pt_u G - \nabla \cdot  (\pt_{ \nabla u} G) ) {\pt_t u} \ud x \ud y,
\end{align*}
which together with \eqref{com-angle} and \eqref{eq1-pde-n}, gives
\begin{align*}
\frac{\ud}{\ud t}\F=&-\mathcal{R}\int_{\Gamma(t)} v_{cl}^2 \ud s - \zeta \int_{D(t)} \frac{(\pt_t u)^2}{\sqrt{1+|\nabla u|^2}} \ud x \ud y + \lambda(t) \int_{D(t)} \pt_t u \ud x \ud y\\
=& -\mathcal{R}\int_{\Gamma(t)} v_{cl}^2 \ud s - \zeta \int_{D(t)} \frac{(\pt_t u)^2}{\sqrt{1+|\nabla u|^2}} \ud x \ud y.
\end{align*}

For (iii), we derive the the gradient flow for a single droplet with quasi-static dynamics.
If we consider the gradient flow in the quasi-static setting, i.e., $\zeta=0$, 
we can regard $u(x,y, t)$ as  being driven by $\Gamma(t)$. In other words, we consider $\{\Gamma(t), u(x,y, t)\}$ with  $u$ as the solution to
\begin{equation}\label{234}
\begin{aligned}
\pt_u G -\nabla\cdot  (\pt_{ \nabla u} G)-\lambda=0 , \quad \text{ in } D(t),\\
u(\Gamma(t),t)=0,\\
\int_{D(t)} u \ud x \ud y = V.
\end{aligned}
\end{equation}
This gives a reduced manifold  $\Gamma(t)$. Correspondingly, we have the quasi-static trajectory $\pcon_{qs}(t):=\Gamma(t)$, the quasi-static free energy
$
\F_{qs}(\Gamma(t)):= \F((\Gamma(t), u(x,y,t))).
$ and the quasi-static tangent plane $T_{\pcon_{qs}}$.
With the quasi-static metrics $g_{\pcon_r}(\pcon_r', \tilde{\pcon}_r')=\mathcal{R}\int_{\Gamma(t)}v_{cl} \tilde{v}_{cl} \ud s$, we have the  gradient flow for quasi-static dynamics
\begin{equation}
\begin{aligned}
\frac{\ud }{\ud s}\big|_{s=t} \F(\tilde{\pcon}(s)) = \frac{\ud }{\ud s}\big|_{s=t} \F_{qs}(\tilde{\pcon}_{qs}(s))
= -g_{\pcon_{qs}}(\pcon_r', \tilde{\pcon}_{qs}')= -\mathcal{R}\int_{\Gamma(t)}v_{cl} \tilde{v}_{cl} \ud s.
\end{aligned}
\end{equation}
Then by the  calculation in \eqref{var}, we have the gradient flow for $\Gamma(t)$
\begin{equation}
\begin{aligned}\label{re-gf}
\mathcal{R}v_{cl}(t) =-\frac{\delta \F}{\delta \Gamma}:=  -\left[G+|\nabla  u|(n_{cl} \cdot \pt_{\nabla u} G)\right ],\quad (x,y)\in \Gamma.
\end{aligned}
\end{equation}
Notice the right hand sides depend on $u$ which is solved by the nonlinear elliptic equation \eqref{234}. Combing the  \eqref{re-gf} with the elliptic equation \eqref{234} gives a complete description of the quasi-static dynamics of the droplet. 
\end{proof}

From \eqref{eq_nD},  the governing equations for a 2D droplet with wetting domain $D(t)=(a(t),b(t))$  become
\begin{equation}\label{wet-phy}
\begin{aligned}
\frac{\zeta}{\gamma_{lg}} \frac{\pt_t u(x,t)}{\sqrt{1+ (\pt_x u)^2}}= \frac{\pt }{\pt x}\left( \frac{\pt_x u}{\sqrt{1+ (\pt_x u)^2}} \right)-\varsigma u+\frac{1}{\gamma_{lg}}\lambda(t), \quad x\in(a(t),b(t)),\\
 u(a(t), t)=u(b(t), t)=0,\\
\frac{\mathcal{R}}{\gamma_{lg}}a'(t)= \sigma+\frac{1}{\sqrt{1+(\pt_x u)^2}},\quad x=a(t),\\
\frac{\mathcal{R}}{\gamma_{lg}}b'(t)= -\sigma-\frac{1}{\sqrt{1+(\pt_x u)^2}},\quad x=b(t),\\
\int_{a(t)}^{b(t)} u(x,t) \ud x = V.
\end{aligned}
\end{equation}
 In 2D, the units of $\gamma_{lg}$ becomes energy/length, $\mathcal{R}$ becomes mass/time, $\zeta$ becomes mass/(length$\cdot$time) and $\varsigma$ is 1/(length$^2$).
 Let $t=T \hat{t}$, $x=L \hat{x}$,  $u=L \hat{u}$, $a=L \hat{a}, b=L \hat{b}$, $V=L^2 \hat{V}$  and $\lambda=\frac{\gamma_{lg}}{L}\hat{\lambda}$ where $L$ is the typical length of the droplet and $T$ is the typical time scale to observe the motion of contact lines. In other words, we choose typical time $T$ such that $\frac{\mathcal{R}L}{\gamma_{lg}T}=1$ and typical volume for unit disk $\hat{V}=\pi$. We denote the capillary number for the capillary surface as $\beta:= \frac{\zeta L^2}{\gamma_{lg} T}$ and set $\kappa=L^2 \varsigma$,   both being dimensionless. Then the dimensionless equations (after dropping hat) in 2D are given by \eqref{wet-phy0}.

\subsection{Additional hydrodynamic effects inside the droplets}
 We also remark the relations between our pure geometric motion of droplets and the other contact line dynamics coupled with hydrodynamic effect of viscous fluid.

Consider further the hydrodynamic effect, which    specifically is (i)  adding  kinetic energy $\frac{1}{2}\rho \int_{A} |\mathbf v|^2  \ud x$ due to inertial effect in the free energy $\mathcal{F}$, (ii) adding viscosity dissipation inside the droplet and (iii) adding energy dissipation on solid-liquid interface due to Navier slip boundary condition. Then the energy dissipation relation becomes \cite[eq(39)]{ren2007boundary}
\begin{equation}
\frac{\ud}{\ud t}\left(\frac{1}{2}\rho \int_{A} |\mathbf v|^2  \ud x +  \mathcal{F}\right) = -\mu \int_{A} |\nabla \mathbf v|^2 \ud x - \frac{\mu}{\alpha} \int_{\pt A \cap \{z=0\}}  |\mathbf v|^2 \ud s  - \mathcal{R} \int_\Gamma v_{cl}^2 \ud s, 
\end{equation}
where $\alpha$ is the slip length.
The corresponding governing equations are
\begin{equation}\label{app_E}
\begin{aligned}
&\rho \left( \pt_t \mathbf v + (\mathbf v \cdot \nabla ) \mathbf  v  \right) + \nabla p =\mu \Delta \mathbf v , \quad \text{ in } A(t),\\
&\nabla \cdot \mathbf v = 0, \quad \text{ in } A(t),\\
& - p + \mathbf n^T \cdot \mu (\nabla \mathbf v+ \nabla \mathbf v^T) \cdot \mathbf n = \gamma_{lg} H , \quad \text{  on } \pt A(t) \cap \{z>0\},\\
& \mathbf \tau^T \cdot \mu (\nabla \mathbf v+ \nabla \mathbf v^T) \cdot \mathbf n=0, \quad \text{  on } \pt A(t) \cap \{z>0\},\\
& v_3=0, \quad  (v_1, v_2) = \alpha \pt_z ( v_1, v_2), \quad \text{ on } A(t) \cap\{ z=0\},\\
&\mathcal{R} v_{cl} = \gamma_{lg} (\cos \theta_Y-\cos \theta_{cl}), \text{ on } \Gamma(t).
\end{aligned}
\end{equation}
Here, $\mathbf v$ is the fluid velocity, $p$ is the pressure, $\mu$ is viscosity and $\rho$ is the fluid density. 
 Notice in the third equation in \eqref{app_E}, the normal traction (normal force) induced by the fluids $F_n= - p + \mathbf n^T \cdot \mu (\nabla \mathbf v+ \nabla \mathbf v^T) \cdot \mathbf n = \gamma_{lg} H$ is balanced with the mean curvature $\gamma_{lg} H$. Meanwhile, in the purely geometric motion, the normal velocity of the capillary surface $\zeta v_n=-\gamma_{lg} H$ can be understood as  velocity induced by an underlying normal frictional force $F_n= - \zeta v_n$.   Roughly speaking, the purely geometric motion derived in Appendix \ref{app_model} captures the same main feature (motion by mean curvature of capillary surface) as the original hydrodynamic one  in \cite{ren2007boundary}, where the normal velocity of the capillary surface is induced by the fluid velocity following Navier-Stokes equation.

Recall the  first equation in \eqref{eq_nD} in the absence of gravity, i.e.,
$$\zeta v_n = \zeta\frac{\pt_t u}{\sqrt{1+ |\nabla u|^2}} = \gamma_{lg}  \nabla\cdot \left(\frac{\nabla u}{\sqrt{1+|\nabla u|^2}}\right)+\lambda(t)=-\gamma_{lg} H + \lambda(t).$$
The energy dissipation relation \eqref{dissipation_o}   is exactly same as the dissipation relation in \cite[eq(38)]{ren2007boundary},
$$\frac{\ud}{\ud t} \mathcal{F} = -  \gamma_{lg}\int_{\pt A\cap \{z>0\}}  -H \mathbf n \cdot \mathbf v  \ud s  - \mathcal{R} \int_\Gamma v_{cl}^2 \ud s, $$
where  $\mathbf v$ is the velocity of the capillary surface. Here $\lambda(t)$ does not contribute due to  $\nabla \cdot \mathbf v=0$ inside the droplet.

\section{Proof for the truncation analysis of first and second order schemes}\label{appA}
In this section, we give some truncation error estimates for the first and second order schemes in the case $w(x)=0$, $\theta_0=0$ and thus $h(x)=u(x)$. Now the governing equation \eqref{wet-eq-r} becomes \eqref{wet-phy0}.
\begin{proof}[Proof of Lemma \ref{first-lem1} (first order truncation error estimates)]
Let $a(t^{n+1}), b(t^{n+1}), u(x^{n+1},t^{n+1})$ for $x^{n+1}\in[a(t^{n+1}), b(t^{n+1})]$ be the exact solution to \eqref{wet-phy0}  evaluated at $t=t^{n+1}$ with initial data at $t=t^n$, $a^n, b^n, u^n(x^n)$ for $x^n\in[a^n, b^n]$. We outline the idea of proof below.

Step 1. Truncation error estimate \eqref{abn} for  moving boundary. By Taylor expansion 
\begin{equation}
a(t^{n+1})=a^n + a'(t^n) \Delta t + O(\Delta t ^2),
\end{equation}
and the boundary condition in \eqref{wet-phy0}, 
\begin{equation}
a'(t^n)= (\sigma +\frac{1}{\sqrt{1+ (\pt_x u^n)^2}} )\big|_{a^n}, 
\end{equation}
we have
\begin{equation}\label{an}
a(t^{n+1}) = a^n + \Delta t   (\sigma +\frac{1}{\sqrt{1+ (\pt_x u^n)^2}} )\big|_{a^n} + O(\Delta t ^2),
\end{equation}
which corresponds to \eqref{abn} for $w(x)=0$.
Similarly for $b(t^{n+1})$, we also have
\begin{equation}\label{bn}
b(t^{n+1}) = b^n - \Delta t   (\sigma +\frac{1}{\sqrt{1+ (\pt_x u^n)^2}} )\big|_{b^n} + O(\Delta t ^2).
\end{equation}

Next, we prove truncation error estimate \eqref{un} and divide the proof into four steps.

Step 2.  Map the moving domain to fixed domain.
We map the moving domain $[a(t), b(t)]$ to the fixed domain $[0,1]$  by $Z(x,t)=\frac{x-a(t)}{b(t)-a(t)}\in [0,1]$ for any $x\in[a(t), b(t)]$. Particularly, at different times we have the relation
\begin{equation}\label{relationZ}
 Z(x^{n+1},t^{n+1})=Z(x^n, t^n)= Z(x^{n-1}, t^{n-1})=Z(x^{n+\frac12}, t^{n+\frac12})
\end{equation}
for independent variables $x^k\in[a^k, b^k],\, k=n-1, n, n+\frac12, n+1$ respectively.

Denote $U(Z,t):=u(x,t)$. Then changing of variables shows that
\begin{equation}
\begin{aligned}
     \pt_t u = \pt_t U + \pt_Z U  \pt_t Z , \quad \pt_x u = \frac{1}{b-a} \frac{\pt U}{\pt Z}.
\end{aligned}
\end{equation}
Then we recast \eqref{wet-phy0} in terms of $Z, U$ variables
\begin{equation}
\frac{\beta}{\sqrt{1+ (\frac{\pt_Z U}{b-a})^2}} \left( \pt_t U + \pt_Z U  \pt_t Z  \right) = \frac{1}{(b-a)^2} \pt_z \left( \frac{\pt_Z U}{\sqrt{1+  (\frac{\pt_Z U}{b-a})^2 }} \right)- \kappa U+\lambda.
\end{equation}

Step 3. Truncation error for the term $\pt_t u = \pt_t U + \pt_Z U  \pt_t Z$.
First, using the backward Euler approximation, we can approximate this term. From relation \eqref{relationZ}, we have for $ x^{n+1}\in[a(t^{n+1}),b(t^{n+1})],$
\begin{equation}\label{semi16}
\begin{aligned}
\pt_t u(x^{n+1}, t^{n+1}) =& \pt_t U(Z, t^{n+1}) + \pt_Z U(Z, t^{n+1})  \pt_t Z(x^{n+1}, t^{n+1}) \\
=&  \frac{U(Z, t^{n+1})-U^n(Z)}{\Delta t} + {\pt_Z U^n(Z)} \frac{Z(x^{n+1}, t^{n+1})- Z(x^{n+1}, t^n)}{\Delta t} + O(\Delta t ) \\
=&\frac{u(x^{n+1}, t^{n+1})-u^n(x^n)}{\Delta t} + \pt_x u^n(x^n) \frac{\pt x^n}{\pt Z} \frac{Z(x^{n}, t^{n})- Z(x^{n+1}, t^n)}{\Delta t} + O(\Delta t )  \\
= &\frac{u(x^{n+1}, t^{n+1})-u^n(x^n)}{\Delta t} + \frac{\pt_x u^n(x^n) (x^n-x^{n+1})}{\Delta t} + O(\Delta t ),
\end{aligned}
\end{equation}
where
\begin{equation}\label{tm_xn}
x^n= a^n+ \frac{b^n-a^n}{b(t^{n+1})-a(t^{n+1})}(x^{n+1}-a(t^{n+1})).
\end{equation}
Denote
\begin{equation}\label{86xy}
\begin{aligned}
    &u^{*}(x^{n+1}, t^n):= u^n(x^n)+ \pt_x u^n (x^n)(x^{n+1}-x^n),
    \end{aligned}
\end{equation}
which is exactly \eqref{inter-u-0}.

In summary, the semi-Lagrangian term has first order accuracy 
\begin{equation}\label{n31}
\pt_t u(x^{n+1},t^{n+1}) = \frac{u(x^{n+1},t^{n+1})-u^{*}(x^{n+1}, t^n)}{\Delta t} + O(\Delta t ).
\end{equation}

Step 4. Truncation error for the stretching term $\frac{1}{\sqrt{1+(\pt_x u)^2}}$. From the relation between $x^n$ and $x^{n+1}$ in \eqref{tm_xn} and the truncation error  in Step 1, we have
\begin{align*}
x^{n+1}-x^n&= \frac{(b( t^{n+1})-b^n)-(a(t^{n+1})-a^n)}{b( t^{n+1})-a( t^{n+1})}x^{n+1}+ \frac{(b^n-a^n)a( t^{n+1})-(b( t^{n+1})-a( t^{n+1}))a^n}{b( t^{n+1})-a( t^{n+1})}\\
&=\frac{(b( t^{n+1})-b^n)-(a( t^{n+1})-a^n)}{b( t^{n+1})-a( t^{n+1})}x^{n+1}+ \frac{b^n(a( t^{n+1})-a^n)-a^n(b( t^{n+1})-b^n)}{b( t^{n+1})-a( t^{n+1})}\\
&= O(b( t^{n+1})-b^n)+ O(a( t^{n+1})-a^n)
\end{align*}
and thus
\begin{equation}
|x^{n+1}-x^n| = O(\Delta t ).
\end{equation}
Combing \eqref{tm_xn} and \eqref{86xy}, we have
\begin{equation}\label{accu23}
\begin{aligned}
\pt_x u^{*}(x^{n+1}, t^{n}) &= \pt_x u^n (x^n) \frac{\pt x^n}{\pt x^{n+1}} + \pt_x u^n(x^n)  \pt_x(x^{n+1}-x^n)+ \pt_{xx} u^n(x^{n}) \frac{\pt x^{n}}{\pt x^{n+1}} (x^{n+1}-x^n) \\
&= \pt_x u^n(x^n)  \frac{b^n-a^n}{b(t^{n+1})-a(t^{n+1})} + \pt_x u^n(x^n) \frac{(b(t^{n+1})-b^n)-(a(t^{n+1})-a^n)}{b(t^{n+1})-a(t^{n+1})}+ O(|x^{n+1}-x^n|)\\
&= \pt_x u^n(x^n)+ O(|x^{n+1}-x^n|).
\end{aligned}
\end{equation} 
In summary, we have
\begin{equation}\label{n34}
\frac{1}{b(t^{n+1})-a(t^{n+1})}\pt_Z U^{n*}(Z)= \pt_x u^{*} (x^{n+1}, t^n) = \pt_x u^n(x^n) + O(\Delta t ) = \frac{1}{b^n-a^n}\pt_Z U^{n}(Z)+ O(\Delta t ) .
\end{equation}

Step 5. Truncation error for $u(x^{n+1},t^{n+1}), \, x^{n+1}\in[a(t^{n+1}),b(t^{n+1})]$. Plugging  $u(x^{n+1}, t^{n+1})$ into the first equation in \eqref{wet-phy0}, from \eqref{n31} and \eqref{n34}
\begin{equation}
\frac{\beta}{\sqrt{1+ (\pt_x u^{*}(t^n))^2}} \frac{u(t^{n+1})-u^{*}(t^n)}{\Delta t} = {\blue \frac{\pt_{xx} u(t^{n+1})}{ \bbs{1+(\pt_x u^{*}(t^n))^2}^\frac32}  } -\kappa u(t^{n+1})+\lambda + O(\Delta t ),
\end{equation}
for $x^{n+1}\in[a(t^{n+1}),b(t^{n+1})]$.
We conclude the proof.
\end{proof}

\begin{proof}[Proof of Lemma \ref{second-lem1} (second order truncation error estimates)]
Let $a(t^{n+1}), b(t^{n+1}), u(x^{n+1},t^{n+1})$ for $x^{n+1}\in[a(t^{n+1}), b(t^{n+1})]$ be the exact solution to \eqref{wet-phy0}  evaluated at $t=t^{n+1}$ with initial data at $t=t^n$, $a^n, b^n, u^n(x^n)$ for $x^n\in[a^n, b^n]$.  We will prove truncation error estimates \eqref{second-abn} and \eqref{second-hn} separately in Step 1 and Step 2.
We outline the idea of proof below.

Step 1. Second order truncation error for  the moving boundary \eqref{second-abn}.

We first illustrate the idea of the usual truncation error estimates for the predictor-corrector ODE solver for $v'=f(v)$ with $v''=f'(v)f(v)$.
By Taylor expansion, 
\begin{equation}
\begin{aligned}
v^{n+1} &= v^n + \Delta t  f(v^n) + \frac{(\Delta t)^2}{2} f(v^n) f'(v^n) + O(\Delta t ^3)\\
&= v^n + \frac{\Delta t}{2} \left[ f(v^n) + f(v^n)+ \Delta t f(v^n)f'(v^n)  \right] + O(\Delta t ^3)\\
&= v^n + \frac{\Delta t}{2} \left[ f(v^n) + f(v^n+ \Delta t f(v^n))  \right] + O(\Delta t ^3),
\end{aligned}
\end{equation}
which is equivalent to
\begin{equation}\label{342}
\frac{\tilde{v}^{n+1}- v^n}{\Delta t} = f(v^n), \quad  \frac{v^{n+1}-v^n}{\Delta t} = \frac12 \left[ f(v^n)+ f(\tilde{v}^{n+1}) \right] + O(\Delta t ^2).
\end{equation}
Moreover, for any smooth function $W(v)$, we have the second order estimate
\begin{equation}\label{w-43}
W(v(t^{n+\frac12})) = \frac12 \left( W(v^n)+ W(\tilde{v}^{n+1} ) \right) + O(\Delta t ^2).
\end{equation}
Indeed, since $v(t^{n+\frac12})= v^n+ \frac{\Delta t}{2}f(v^n)+ O(\Delta t ^2)$, by Taylor's expansion we have
\begin{equation}
\text{LHS} = W(v^{n}) + \frac{\Delta t}{2} W'(v^n) f(v^n) + O(\Delta t ^2)=\frac12\left[ W(v^n)+ W(v^n+ {\Delta t} f(v^n))  \right] + O(\Delta t ^2) = \text{RHS}.
\end{equation}
\eqref{w-43} also gives another method for second order truncation error estimate by evaluating the equation $v'=f(v)$ at $t^{n+\frac12}$, which is \eqref{342}. The truncation error in Step 2 will rely on this method.

Second, we again recast the equation for the moving boundary in terms of the fixed domain variable $U(Z,t)=u(x,t)$ with $Z(x,t)=\frac{x-a(t)}{b(t)-a(t)}\in[0,1]$.
\begin{equation}\label{A20}
\begin{aligned}
a'(t) = \sigma + \frac{1}{\sqrt{1+ \frac{(\pt_Z U)^2|_{Z=0}}{(b(t)-a(t))^2} }}=: g(a(t), b(t), \pt_Z U(0,t)),\\
b'(t) = - \sigma -  \frac{1}{\sqrt{1+ \frac{(\pt_Z U)^2|_{Z=1}}{(b(t)-a(t))^2} }}=: q(a(t),b(t), \pt_Z U(1,t)).
\end{aligned}
\end{equation}

Third, analogue to the usual predictor-corrector ODE solver, we calculate the truncation error for $a^{n+1}$. Notice
\begin{equation}
a''(t)= a' \pt_1 g + b' \pt_2 g  + \pt_3 g \pt_{t Z} U|_{Z=0}.
\end{equation}
Taylor's expansion gives us
\begin{align*}
a(t^{n+1}) = a^n + \Delta t g(a^n, b^n, \pt_Z U^n ) + \frac{\Delta t ^2}{2} \left[ a' \pt_1 g + b' \pt_2 g  + \pt_3 g \pt_{t Z} U |_{Z=0} \right]^n + O(\Delta t ^3).
\end{align*}
Hence from Taylor's expansion of $g(\tilde{a}^{n+1}, \tilde{b}^{n+1}, \pt_z \tilde{U}^{n+1})$
\begin{equation}
\begin{aligned}
&\frac{a(t^{n+1})-a^n}{\Delta t} \\
=& \frac12 g(a^n, b^n, \pt_Z U^n) + \frac12 g(a^n+ \Delta t \frac{\tilde{a}^{n+1}-a^n}{\Delta t}, b^n+ \Delta t \frac{\tilde{b}^{n+1}-b^n}{\Delta t}, \pt_Z U^n+ \Delta t \frac{\pt_Z \tilde{U}^{n+1}- \pt_Z U^n}{\Delta t})+ O(\Delta t ^2)\\
=&\frac12 g(a^n, b^n, \pt_Z U^n) + \frac12 g(\tilde{a}^{n+1}, \tilde{b}^{n+1}, \pt_z \tilde{U}^{n+1})+ O(\Delta t ^2),
\end{aligned}
\end{equation}
provided $\frac{\tilde{a}^{n+1}-a^n}{\Delta t}-(a')^n$ and $\frac{\pt_Z \tilde{U}^{n+1}- \pt_Z U^n}{\Delta t}-(\pt_{Zt}U)^n|_{z=0}$ have $O(\Delta t )$ accuracy.

Finally, we prove $\frac{\tilde{a}^{n+1}-a^n}{\Delta t}-(a')^n$ and $\frac{\pt_Z \tilde{U}^{n+1}- \pt_Z U^n}{\Delta t}-(\pt_{Zt}U)^n|_{Z=0}$ have $O(\Delta t )$ accuracy.
Since the predictor $\tilde{a}^{n+1}$ is given by the first order scheme in Section \ref{sec-sim}, we know $\frac{\tilde{a}^{n+1}-a^n}{\Delta t}-(a')^n$ has $O(\Delta t )$ accuracy and we obtain \eqref{abn}. To estimate $\frac{\pt_Z \tilde{U}^{n+1}- \pt_Z U^n}{\Delta t}$, we give the following claim.\\
\textit{
Claim 1: Assume we have the error estimates 
\begin{equation}
\tilde{a}^{n+1}-a^{n+1}=O(\Delta t^2),\,\, \tilde{b}^{n+1}-b^{n+1}=O(\Delta t^2),\quad \tilde{u}^{n+1}(\tilde{x}^{n+1})-u^{n+1}(x^{n+1}) = O(\Delta t^2).
\end{equation}
Then we have the second order accuracy
\begin{equation}
(\pt_x \tilde{u})^{n+1} (\tilde{x}^{n+1})= (\pt_x u)^{n+1}(x^{n+1})+ O(\Delta t^2).
\end{equation}
}
The proof of Claim 1 is based on changing moving domain to fixed domain by $Z=\frac{\tilde{x}^{n+1}-\tilde{a}^{n+1}}{\tilde{b}^{n+1}-\tilde{a}^{n+1}}=\frac{x^{n+1}-a^{n+1}}{b^{n+1}-a^{n+1}}$, which is similar to \eqref{A20} and will be omitted. Notice the first order accuracy of predictor $\tilde{a}^{n+1}, \, \tilde{b}^{n+1}$ and we used implicit elliptic solver with second order accuracy in \eqref{newton} for predictor $\tilde{u}^{n+1}$, so  the assumptions in claim 1 are satisfied automatically. Thus from the Taylor expansion and claim 1 we know
\begin{equation}
\frac{\pt_Z \tilde{U}^{n+1}- \pt_Z U^n}{\Delta t}+ O(\Delta t) = \frac{\pt_Z {U}^{n+1}- \pt_Z U^n}{\Delta t}+ O(\Delta t)= (\pt_{tZ} U)^n+O(\Delta t).
\end{equation}
Therefore, we complete the second order truncation error estimates for the moving boundary \eqref{second-abn}.

Step 2. Second order truncation error estimates \eqref{second-hn} for $u^{n+1}$.

First from the similar argument for \eqref{w-43}, we have the following generalized claim
\\\textit{
Claim 2: For any smooth function $W(v(x,t),v_x(x,t), v_{xx}(x,t), x, t)$, we have
\begin{equation}
\begin{aligned}
&W(v(t^{n+\frac12}), v_x(t^{n+\frac12}), v_{xx}(t^{n+\frac12}),x^{n+\frac12}, t^{n+\frac12}) \\
=& \frac12[W(v^n, v_x^n, v_{xx}^n, x^n, t^n)+ W(v^{n+1}, v_x^{n+1}, v_{xx}^{n+1}, x^{n+1}, t^{n+1})]+ O(\Delta t^2),
\end{aligned}
\end{equation}
where the equality holds in the sense of changing variables to fixed domain $Z=\frac{x-a(t)}{b(t)-a(t)}$ with the relation \eqref{relationZ}.
}

 Second, notice the derivation for  the term $\pt_t h$  in \eqref{inter-u-star} gives the second order accuracy
$$ \frac{u^{n+1}(x^{n+1})- \tilde{u}^{n*}}{\Delta t}=\pt_t u (x^{n+\frac12}, t^{n+\frac12})  + O(\Delta t^2).$$ 
 Using further Claim 1 and Claim 2, we obtain the second order accuracy  for \eqref{second-hn}.
 \begin{align*}
 \frac12\left[\frac{1}{\sqrt{1+ (\pt_x u^{n})^2}}+ \frac{1}{\sqrt{1+ (\pt_x \tilde{u}^{n+1})^2}}\right]= \frac{1}{\sqrt{1+ (\pt_x u)^2}}\Big|_{t^{n+\frac12}} + O(\Delta t^2),\\
 \frac12 \left(\frac{\pt_{xx} {u}^{n+1}}{\bbs{1+ (\pt_x \tilde{u}^{n+1})^2}^\frac32} +  \frac{\pt_{xx} u^n}{\bbs{1+ (\pt_x  u^n)^2}^\frac32} \right)=  \left(\frac{\pt_{xx} {u}}{\bbs{1+ (\pt_x {u})^2}^\frac32} \right)\Big|_{t^{n+\frac12}}+ O(\Delta t^2).
 \end{align*}
 Therefore, we complete the second order truncation error estimates for  \eqref{second-hn}.
\end{proof}

\section{Pseudo-codes for first and second order schemes}\label{appB}
\subsection{First order in time and second order in space}\label{code1-in}
We present a pseudo-code for the first order scheme in Section \ref{sec-1st-scheme}:
\\1. Grid for time:
$t^n = n \Delta t$, $n=0 ,1, \cdots,$ where $\Delta t$ is time step.
\\ 2. Fix $N$ and set moving grids for space: $x_{j}^n = a^n + {j}{\tau^n},\,\,\tau ^n=\frac{b^n-a^n}{N}$, $j=-1,0, 1, \cdots, N+1.$ 
\\ 3. Calculate volume $V:= \sum_{j=1}^{N-1}( h^0-w)(x^0_j) \tau^0$.
\\4. Denote the finite difference operators
\begin{equation}\label{tm_C2}
\begin{aligned}
&(\pt_x h)_0^n= \frac{4 h_1^n - h_2^n-3h_0^n}{2 \tau^n}, \quad  (\pt_x h)_N^n= \frac{-4 h_{N-1}^n + h_{N-2}^n+ 3 h_{N}^n}{2 \tau^n}, \\
 &(\pt_x h)_j^{n} = \frac{h_{j+1}^{n} - h^{n}_{j-1}}{2 \tau^{n}}, \quad   (\pt_{xx} h)_j^{n}=  \frac{h^{n}_{j}-2h^{n}_j + h_{j-1}^{n}} {(\tau^{n})^2},\,\, j=1, \cdots, N-1,
 \end{aligned}
\end{equation}
with Dirichlet boundary condition $h^{n+1}_0= w(x^{n+1}_0),\,\, h^{n+1}_N = w(x_N^{n+1}) $.
\\ 5. Update $a^{n+1}, b^{n+1}$, $j=0, 1, \cdots, N$.
\\ ~~~~  
\begin{align}\label{code-ea}
\frac{a^{n+1}-a^n}{\Delta t}&= \sigma \sqrt{1+ (\pt_x w)_0^2} +\frac{1+ (\pt_x h^n)_0(\pt_x w)_0 }{\sqrt{1+ (\pt_x h^n)_0^2}} ,  \quad  \text{ with } (\pt_x w)_0=\pt_x w(x^n_0).
\end{align}
\begin{align}\label{code-eb}
\frac{b^{n+1}-b^n}{\Delta t}&= -\sigma \sqrt{1+ (\pt_x w)_N^2} -\frac{1+ (\pt_x h^n)_N(\pt_x w)_N }{\sqrt{1+ (\pt_x h^n)_N^2}} , \quad   \text{ with } (\pt_x w)_N=\pt_x w(x^n_N).
\end{align}
\\6. Update the moving grids
$\, x_{j}^{n+1} = a^{n+1} + {j}{\tau^{n+1}}, \,\, \tau^{n+1}= \frac{b^{n+1}-a^{n+1}}{N}, \quad  j=0, 1, \cdots, N.$
\\ 7. From \eqref{inter-u-0}, 
$\,h_j^{n*}= h_j^n + (\pt_x h^n)_j (a^{n+1}-a^n+j(\tau^{n+1}-\tau^n)), \quad j=0, \cdots, N.$
\\ {\blue 8. Solve $h^{n+1}$ semi-implicitly
\\For $j=1, \cdots, N-1$, 
\begin{align}\label{code-eq-r-0}
&\beta \frac{h_j^{n+1}-h^{n*}_j}{\Delta t}\frac{1}{\sqrt{1+ (\pt_x h^{n*})_j^2}}=\frac{(\pt_{xx} h^{n+1})_j}{\bbs{1+(\pt_x h^{n*})_j^2}^\frac32}  -\kappa (h_j^{n+1} \cos \theta_0 + x^{n+1}_j \sin \theta_0  ) +\lambda^{n+1},\\
&\sum_{j=1}^{N-1} (h^{n+1}_j-w(x_j^{n+1})) \tau^{n+1} = V, \nonumber
\end{align}
Due to the $O(|x^{n+1}-x^n|)=O(\Delta t)$ accuracy for $\pt_x u^{*}(x^{n+1})$ and $\pt_x u^n(x^n)$ in \eqref{accu23}, to ensure the stability in the implementation,  we can also replace $(\pt_x h^{n*})_j$ by $(\pt_x h^{n})_j$.

Denote a positive-definite matrix $A_{(N-1)\times(N-1)}=(a_{ij})$
with
\begin{equation}\label{tri-r-0}
\begin{array}{c}
a_{j, j-1} = -1, \,\, a_{j,j+1} = -1, \quad \alpha_j=1+ (\pt_x h^{n})_j^2 \\
a_{j,j} =2+ \frac{\beta (\tau^{n+1})^2}{\Delta t}\alpha_j+ \kappa \cos \theta_0 (\tau^{n+1})^2 \alpha_j^{\frac32},\\
\end{array}
\end{equation}
 and \eqref{code-eq-r-0} becomes for $j=1, \cdots, N-1$, 
\begin{equation}\label{1st-eq-u}
\begin{aligned}
a_{j,j-1} h^{n+1}_{j-1} + a_{j,j} h^{n+1}_j + a_{j,j+1} h^{n+1}_{j+1} - \alpha_j^\frac32 (\tau^{n+1})^2\lambda^{n+1}=  \frac{\beta (\tau^{n+1})^2}{\Delta t} h_j^{n*} \alpha_j-\kappa \sin \theta_0 x_{j}^{n+1} (\tau^{n+1})^2 \alpha_j^{\frac32}=: \tilde{f}_j.
\end{aligned}
\end{equation}
Denote
\begin{equation}
f_1=\tilde{f}_1+w(x_0^{n+1}),\,\,   \{f_j= \tilde{f}_j\}_{j=2}^{N-2} , \,\, f_{N-1} = \tilde{f}_j + w(x_N^{n+1}), \,\, f_N:= \sum_{j=1}^{N-1} w(x_j^{n+1})  + \frac{V}{\tau^{n+1}}.
\end{equation}

 The resulting linear system $\bar{A} y = f$ has a non-singular   matrix 
\begin{equation}
\bar{A}=\left(
\begin{array}{cc}
A & \alpha^\frac32 \\ 
e^T & 0
\end{array} 
 \right)_{N\times N},
\end{equation}
where  $y^T=(h^{n+1}_1, \cdots, h^{n+1}_{N-1}, -(\tau^{n+1})^2\lambda^{n+1})$ and $e^T=(1,\cdots , 1)\in \mathbb{R}^{N-1}$.   
}

\subsection{Predictor-corrector scheme: second order in time and space }\label{code2-in}
We present a pseudo-code for the second order scheme in Section \ref{sec_2nd_scheme}:
\\1. Grid for time:
$t^n = n \Delta t$, $n=0 ,1, \cdots,$ where $\Delta t$ is time step.
\\ 2. 
Fix $N$ and set moving grids for space: $x_{j}^n = a^n + {j}{\tau^n},\,\,\tau ^n=\frac{b^n-a^n}{N}$, $j=-1,0, 1, \cdots, N+1.$
\\ 3. Calculate volume $V:= \sum_{j=1}^{N-1}( h^0-w)(x^0_j) \tau^0$ and denote the finite difference operator with Dirichlet boundary condition as \eqref{tm_C2}.
\\4. Repeat the first order scheme \eqref{code-ea}, \eqref{code-eb} in Section \ref{code1-in} with implicit nonlinear elliptic solver.
For $j=1, \cdots, N-1$, with $h^{n+1}_0= w(a^{n+1}),\,\, h^{n+1}_N = w(b^{n+1}) $,
\begin{align}\label{code-eq-r-im}
&\beta \frac{h_j^{n+1}-h_j^{n*}}{\Delta t}\frac{1}{\sqrt{1+( (\pt_x h)_j^{n+1})^2}}= \frac{(\pt_{xx} h)^{n+1}_j}{\bbs{1+((\pt_x h)^{n+1}_j)^2}^\frac32}  -\kappa (h_j^{n+1} \cos \theta_0 + x^{n+1}_j \sin \theta_0  ) +\lambda^{n+1},\nonumber\\
&\sum_{j=1}^{N-1} (h^{n+1}_j-w(x_j^{n+1})) \tau^{n+1} = V, \nonumber
\end{align}
 Denote the results as the predictor $\tilde{a}^{n+1}, \tilde{b}^{n+1}, \tilde{h}^{n+1}(\tilde{x}^{n+1}), \tilde{\tau}^{n+1}$ for $\tilde{x}^{n+1}\in [\tilde{a}^{n+1}, \tilde{b}^{n+1}]$.
\\ 5. Update $a^{n+1}, b^{n+1}$, $j=0, 1, \cdots, N$.
\\ ~~~~  
\begin{align*}
\frac{a^{n+1}-a^n}{\Delta t}&=\frac12\left\{ \sigma \sqrt{1+ (\pt_x w)_0^2}+ \sigma\sqrt{1+ (\pt_x \tilde{w})_0^2} +\frac{1+ (\pt_x h)^n_0(\pt_x w)_0 }{\sqrt{1+ ((\pt_x h)^n_0)^2}}+ \frac{1+ (\pt_x \tilde{h})^{n+1}_0(\pt_x \tilde{w})_0 }{\sqrt{1+ ((\pt_x \tilde{h})^{n+1}_0)^2}}\right\},
\\
\frac{b^{n+1}-b^n}{\Delta t}&=-\frac12\left\{ \sigma \sqrt{1+ (\pt_x w)_N^2}+ \sigma\sqrt{1+ (\pt_x \tilde{w})_N^2} +\frac{1+ (\pt_x h)^n_N(\pt_x w)_N }{\sqrt{1+ ((\pt_x h)^n_N)^2}}+ \frac{1+ (\pt_x \tilde{h})^{n+1}_N(\pt_x \tilde{w})_N }{\sqrt{1+ ((\pt_x \tilde{h})^{n+1}_N)^2}}\right\} , \quad\\
\text{with }& (\pt_x w)_0:=\pt_x w(x^n_0),\,\, (\pt_x \tilde{w})_0:=\pt_x w(\tilde{x}^{n+1}_0),\quad (\pt_x w)_N:=\pt_x w(x^n_N),\,\, (\pt_x \tilde{w})_N:=\pt_x w(\tilde{x}^{n+1}_N).
\end{align*}
\\6. Update the moving grids
$x_{j}^{n+1} = a^{n+1} + {j}{\tau^{n+1}}, \,\, \tau^{n+1}= \frac{b^{n+1}-a^{n+1}}{N}, \quad  j=0, 1, \cdots, N.$
\\ 7. Calculate $h_j^{n*}$ based on \eqref{tm376}, for $ j=1, \cdots, N-1$.
\\ 8. Solve $h^{n+1}$ implicitly. For $j=1, \cdots, N-1$, with $h^{n+1}_0= w(x^{n+1}_0),\,\, h^{n+1}_N = w(x_N^{n+1}) $,
\begin{equation}
\begin{aligned}\label{code-eq-r}
&\beta \frac{h_j^{n+1}-h_j^{n*}}{\Delta t}\left[\frac{1}{\sqrt{1+ ((\pt_x h)_j^{n})^2}}+ \frac{1}{\sqrt{1+ ((\pt_x  {h})_j^{n+1})^2}}\right]\\
=& \frac{(\pt_{xx} h)^{n+1}_j}{\bbs{1+((\pt_x {h})_j^{n+1})^2}^\frac32} +  \frac{(\pt_{xx} h)^n_j}{\bbs{1+(\pt_x h_{j}^{n})^2}^\frac32} -\kappa [(h_j^{n+1}+h_j^n) \cos \theta_0 +(x^n_j+x^{n+1}_j) \sin \theta_0  ] +2\lambda^{n+\frac12},\\
&\sum_{j=1}^{N-1} (h^{n+1}_j-w(x_j^{n+1})) \tau^{n+1} = V.
\end{aligned}
\end{equation}


\section{Gradient Flow for single sessile drops in non-wetting case using horizontal graph representation $X(u)$ }\label{app_nonwet}

 Recall the description of the non-wetting droplet in terms of $X(u)$ in \eqref{inverse}.
We consider the manifold based on $X(u)$
\begin{equation}
\mm:= \{\eta=(u_m, X(u)); u_m\geq 0, X(u)\in H^1(0,u_m),  \, X(u_m)=0, X_u(u_m)=-\8\}.
\end{equation}

Similar to Appendix \ref{app_model}, we calculate the gradient flow on manifold $\mm$. Below, we directly use dimensionless quantities for simplicity.  Let $\tilde{X}(u, s)$ with $\tilde{u}_m(s)$ be trajectory on $\mm$ starting from $X(u,t)$ with $u_m(t)$. Based on the relations in Lemma \ref{lem_re}, the tangent plane $T_\eta \mm$ at $\eta$ can be described by the contact line speed $v_{cl}=X_t(0)$ and the horizontal velocity $v=X_t$ of the capillary surface. 
Notice from
$X(u_m(t),t)=0$, on the tangent plane $T_\eta \mm$ we have 
\begin{equation}
\pt_t\tilde{X}= - X_u \pt_t \tilde{u}_m \quad \text{ on }u=u_m.
\end{equation}
For any $q_1=(v_{cl1}, v_1),\, q_2=(v_{cl2},v_2)\in T_{\pcon}\mm$, define the Riemannian metric $g_{\pcon}: T_\eta \mm \times T_\eta \mm \to \mathbb{R}$  
\begin{equation}
g_{\pcon} (q_1, q_2):= v_{cl1} v_{cl2} + \beta \int_0^{u_m} \frac{v_1 v_2}{\sqrt{1+ |X_u|^2}} \ud u.
\end{equation}
Consider the free energy
\begin{equation}
\begin{aligned}
\frac{1}{2}E(X) &= \int_0^{u_m} \sqrt{1+ X_u^2} \ud  u + \sigma X(0) + \kappa \int_0^{u_m} u X(u) \ud u - \lambda(\int_0^{u_m} X(u) \ud u-V/2)\\
&=  \int_0^{u_m} \sqrt{1+ X_u^2} \ud  u - \sigma \int_0^{u_m} X_u \ud u + \kappa \int_0^{u_m} u X(u) \ud u - \lambda(\int_0^{u_m} X(u) \ud u-V/2)\\
&=:\int_0^{u_m} G \ud u + \lambda V/2=: \int _0^{u_m} G_1(X_u)+ G_2(u, X(u))\ud u+ \lambda V/2,
\end{aligned}
\end{equation}
where $G_1:=\sqrt{1+ X_u^2}- \sigma X_u$ and $G_2:=\kappa u X(u)-\lambda X(u). $ 
First notice the identity
\begin{equation}\label{tm81}
G-X_u G_{X_u} \Big|_{u_m}= \frac{1}{\sqrt{1+X_u^2}} \Big|_{u_m}=0.
\end{equation}
Then we have
\begin{align*}
&\frac12\frac{\ud}{\ud s}\Big|_{s=t}E(\tilde{X}(u(s),s)= G|_{u_m} \tilde{u}_m'+ \int_0^{u_m} G_X \pt_t \tilde{X} + G_{X_u} \pt_u(\pt_t \tilde{X}) \ud u-\lambda'(t) (\int_0^{u_m} X(u) \ud u-V/2)\\
=& G|_{u_m} \tilde{u}_m'+ \int_0^{u_m} (G_X -\frac{\ud}{\ud u} G_{X_u})\pt_t \tilde{X} \ud u+ G_{X_u}\pt_t \tilde{X}\big |_{0}^{u_m}-\lambda'(t) (\int_0^{u_m} X(u) \ud u-V/2)\\
 =& (G-X_u G_{X_u} )\big|_{u_m} \tilde{u}_m'+ \int_0^{u_m} (G_X -\frac{\ud}{\ud u} G_{X_u})\pt_t \tilde{X} \ud u  - G_{X_u} \pt_t \tilde{X} \big|_{u=0}-\lambda'(t) (\int_0^{u_m} X(u) \ud u-V/2)\\
=&   \int_0^{u_m} (G_X -\frac{\ud}{\ud u} G_{X_u})\pt_t \tilde{X} \ud u - G_{X_u} \pt_t \tilde{X} \big|_{u=0}-\lambda'(t) (\int_0^{u_m} X(u) \ud u-V/2),
\end{align*}
where we used \eqref{tm81}. 
Then the gradient flow of $E$ on manifold $\mm$ with respect to the metrics $g_{\eta}$ gives the governing equations \eqref{dewet-eq-n}.
{\begin{rem}
Using the similar derivations above, we have the governing equations for a 3D axisymmetric droplet in terms of $R(u,t)$
\begin{equation}
\begin{aligned}
\beta \frac{\pt_t R}{\sqrt{1+ R_u^2}} = -\frac{ \sqrt{1+R_u^2}}{R} + \frac{1}{R}\pt_u\left( \frac{R R_u}{\sqrt{1+ R_u^2}} \right) -\kappa u  + \lambda , \\
\pt_t R(0) = -\left(\sigma - \frac{R_u}{\sqrt{1+ R_u^2}}\big|_{u=0}  \right) = - (\sigma + \cos\theta),\\
V=\int_0^{u_m}\pi R^2 \ud u,
\end{aligned}
\end{equation}
where $\theta$ is defined as $\tan\theta=- R_u|_{u=0}$.
Here
$
 \frac{ \sqrt{1+R_u^2}}{R} - \frac{1}{R}\pt_u\left( \frac{R R_u}{\sqrt{1+ R_u^2}} \right)
  =  \frac{1}{R\sqrt{1+R_u^2}}-\frac{R_{uu}}{(1+ R_u^2)^{3/2}},
$
 is the mean curvature in terms of  $R(u)$.
\end{rem}
}

\end{document}